\documentclass{amsart}
\usepackage{amsmath}
\usepackage{amssymb}
\usepackage[all]{xy}
\usepackage{comment}

\theoremstyle{plain}
\newtheorem{thm}{Theorem}[section]
\newtheorem{thm*}{Theorem}[section]
\newtheorem{cor}[thm]{Corollary}
\newtheorem{prop}[thm]{Proposition}
\newtheorem{lemma}[thm]{Lemma}

\theoremstyle{definition}
\newtheorem{defn}[thm]{Definition}
\newtheorem{remark}[thm]{Remark}

\newtheorem*{remark*}{Remark}

\newtheorem{ex}[thm]{Example}

\newtheorem{question*}{Question}

\numberwithin{equation}{thm}

\def\Rad{\operatorname{Rad}\nolimits}
\def\Spec{\operatorname{Spec}\nolimits}
\def\Ker{\operatorname{Ker}\nolimits}
\def\Res{\operatorname{Res}\nolimits}

\def\Soc{\operatorname{Soc}\nolimits}
\def\Coker{\operatorname{Coker}\nolimits}
\def\rk{\operatorname{Rk}\nolimits}
\def\Im{\operatorname{Im}\nolimits}
\def\Proj{\operatorname{Proj}\nolimits}

\def\Grass{\operatorname{Grass}\nolimits}
\def\Coh{\operatorname{Coh}\nolimits}

\def\Det{\operatorname{Det}\nolimits}

\def\rk{\operatorname{rk}\nolimits}

\def\id{\operatorname{id}\nolimits}
\def\Spec{\operatorname{Spec}\nolimits}
\def\sl2{\operatorname{SL_{2(2)}}\nolimits}
\def\Ga2{\operatorname{\mathbb G_{\rm a(2)}}\nolimits}

\def\GL{\operatorname{GL}\nolimits}

\def\HHH{\operatorname{H}\nolimits}
\def\Ext{\operatorname{Ext}\nolimits}

\def\Hom{\operatorname{Hom}\nolimits}

\def\proj{\operatorname{proj}\nolimits}

\def\diag{\operatorname{diag}\nolimits}
\def\mod{\operatorname{mod}\nolimits}
\def\det{\operatorname{det}\nolimits}
\def\Dim{\operatorname{dim}\nolimits}
\def\dim{\operatorname{dim}\nolimits}

\def\P{\operatorname{P}\nolimits}

\def\Stab{\operatorname{Stab}\nolimits}
\newcommand{\cKer}{\mathcal K\text{\it er}}
\newcommand{\cIm}{\mathcal I\text{\it m}}
\newcommand{\cCoker}{\mathcal C\text{\it oker}}

\newcommand{\bA}{\mathbb A}
\newcommand{\cA}{\mathcal A}

\newcommand{\cE}{\mathcal E}
\newcommand{\E}{\mathcal E}

\newcommand{\bF}{\mathbb F}
\newcommand{\cF}{\mathcal F}

\newcommand{\cI}{\mathcal I}
\newcommand{\CI}{\mathcal I}

\newcommand{\cL}{\mathcal L}
\newcommand{\bM}{\mathbb M}

\newcommand{\cO}{\mathcal O}

\newcommand{\bP}{\mathbb P}

\newcommand{\cR}{\mathcal R}

\newcommand{\cU}{\mathcal U}
\newcommand{\CU}{\mathcal U}

\newcommand{\CV}{\mathbb V}
\newcommand{\CW}{\mathbb W}

\newcommand{\bZ}{\mathbb Z}

\newcommand{\fg}{\mathfrak g}

\newcommand{\fp}{\mathfrak p}
\newcommand{\p}{\mathfrak p}
\newcommand{\fq}{\mathfrak q}
\newcommand{\fu}{\mathfrak u}

\newcommand{\ol}{\overline}
\newcommand{\ul}{\underline}
\newcommand{\wt}{\widetilde}
\newcommand{\bu}{\bullet}
\newcommand{\red}{\rm red}

\newcommand{\Z}{\mathbb Z}

\newcommand{\cOG}{\mathcal O_{Gr}}

\date\today

\begin{document}

 \title[Representations of elementary abelian $p$-groups]{Representations of elementary abelian $p$-groups and bundles on Grassmannians}
 
 \author[Jon F. Carlson, Eric M. Friedlander and Julia Pevtsova]
{Jon F. Carlson$^*$, Eric M. Friedlander$^{**}$ and 
Julia Pevtsova$^{***}$}

\address{Department of Mathematics, University of Georgia,
Athens, GA}
\email{jfc@math.uga.edu}

\address {Department of Mathematics, University of Southern California,
Los Angeles, CA}
\email{eric@math.northwestern.edu}

\address {Department of Mathematics, University of Washington, 
Seattle, WA}
\email{julia@math.washington.edu}

\thanks{$^*$ partially supported by the NSF grant DMS-1001102}
\thanks{$^{**}$ partially supported by the NSF grant DMS-0909314 and DMS-0966589}
\thanks{$^{***}$ partially supported by the NSF grant DMS-0800930 and DMS-0953011}



\begin{abstract}   
We initiate  the study of representations of elementary abelian $p$-groups 
via restrictions to truncated polynomial subalgebras of the group algebra generated 
by $r$ nilpotent elements, $k[t_1, \ldots, t_r]/(t^p_1, \ldots, t_r^p)$.  We introduce new geometric 
invariants based on the behavior of modules upon restrictions to such subalgebras.
We also introduce modules of constant radical and socle type generalizing modules of 
constant Jordan type and provide several general constructions of modules with these properties. 
We show that modules of constant radical and socle type lead to families of algebraic vector 
bundles on Grassmannians and illustrate our theory with numerous examples. 
\end{abstract}

\maketitle

\tableofcontents 
Quillen's fundamental ideas on applying geometry to 
the study of group cohomology in positive characteristic 
\cite{Q} opened the door to many exciting developments in both 
cohomology and modular representation theory.    
Cyclic shifted subgroups, the prototypes of the rank $r$ shifted 
subgroups studied in this paper, were introduced by Dade in \cite{Dade} 
and quickly became the subject of an intense study.  
In \cite{AS}, Avrunin and Scott proved the conjecture of the 
first author tying the cohomological support variety originating from 
Quillen's approach with the variety of shifted subgroups (rank variety) 
introduced in \cite{C}.   

These ideas were successfully applied to restricted Lie algebras  (\cite{FPa86}) 
and, more generally, infinitesimal group schemes (\cite{SFB1}, \cite{SFB2})  yielding  many 
surprising  geometric results which also underline the very different nature   
of infinitesimal group schemes and  finite groups.    Nonetheless, in \cite{FP1}, \cite{FP2}, the second and third authors 
found a unifying tool, called $\pi$-points, that allowed the generalization of shifted cyclic subgroups 
and the Avrunin-Scott's theorem to any finite group scheme.  

In a surprising twist, the $\pi$-point approach has led to new discoveries even for elementary abelian  
p-groups, the context in which cyclic shifted subgroups  were originally introduced.  Among these, 
the most relevant to the present paper  are modules of constant Jordan type (\cite{CFP}) and the 
connection between such modules and algebraic vector bundles on projective varieties (\cite{FP3},  
see also \cite{BP} and \cite{Ben2} for a treatment specific to elementary abelian $p$-groups).

Equipped with the understanding  of the versatility as well as the limits of cyclic shifted subgroups, 
we set out on the quest of studying modular representations  via 
their restrictions to rank $r$  shifted subgroups. Following the original course of the 
development of the theory, we devote this paper entirely  to modular representations of an 
elementary abelian  $p$-group $E$  over an algebraically closed field $k$ of positive characteristic $p$. 
A {\it rank $r$ shifted subgroup} of the group algebra $kE$ is a 
subalgebra $C \subset kE$ isomorphic to a group algebra of an elementary 
abelian $p$-group of rank $r$, for $1 \leq r<n$, 
with the property that $kE$ is free as a $C$--module. 
For an $E$-module $M$, we consider restrictions of $M$ to such subalgebras $C$ of $kE$.    
The concept of a ``shifted subgroup" exists in the literature 
(see, e.g., \cite{Ben}) but no systematic study  of such restrictions has been undertaken  for $r>1$.

Throughout the paper, we choose an $n$-dimensional linear subspace $\CV \subset \Rad(kE)$ which gives a splitting of the projection 
$\Rad(kE) \to \Rad(kE)/\Rad^2(kE)$. Once such a $\CV$ is fixed,  we consider only the rank $r$ shifted subgroups which are determined 
by a linear subspace of $\CV$. Such shifted subgroups are naturally parametrized by the Grassmann variety $\Grass(r,\CV)$ of 
$r$-planes in $\CV$.  In Section \ref{se:radsoc} we prove a partial generalization of the main result in \cite{FPS} showing  that some of the invariants we introduce do not depend on the choice of $\CV$. 

The paper  naturally splits into two parts.  In the first part which occupies  Sections 1 through 5,  
we introduce new geometric and numerical invariants for  modules arising from their restrictions to rank $r$ 
shifted subgroups and then construct many examples to reveal some of the interesting behavior of these invariants.   We show
how to associate  subvarieties of $\Grass(r,\CV)$ to a finite dimensional $kE$-module $M$;
for $r=1$, these subvarieties are refinements of the rank variety of $M$.  In the second half of the paper we construct and  study 
algebraic vector bundles on $\Grass(r,\CV)$ associated to certain $kE$--modules, extending the construction for $r=1$
first introduced in \cite{FP3}. 

Whereas the isomorphism type of a $k[t]/t^p$--module
$M$ is specified by a $p$--tuple of integers (the Jordan type of $M$), there is no such
classification for $r > 1$.
Indeed, except in the very special case in which $p=2 =r$, the category of finite
dimensional $C \simeq k[t_1, \ldots, t_r]/(t^p_1, \ldots, t_r^p)$--modules is wild.    For $r \geq 1$, we
consider dimensions of $C$--socles and $C$--radicals of a given $kE$--module $M$ as $C$
ranges over rank $r$ shifted subgroups of $kE$.  For $r=1$,
this numerical data  is equivalent to the Jordan type of $M$.  Although these ranks do not
determine the isomorphism types of the restrictions of a  given $kE$--module $M$ for $r > 1$, 
they do provide intriguing new invariants for $M$.

 Extending our earlier investigations
of $kE$--modules of constant Jordan type, we formulate in \eqref{const-rad} and then study the
condition on a $kE$--module $M$ that it have constant $r$--radical type or constant 
$r$--socle type. 
We introduce invariants for $kE$--modules which do not have constant $r$--radical type
(or constant $r$--socle type).   Our simplest invariant, a straight--forward generalization of the
rank variety of a $kE$--module $M$, is the $r$--rank variety $\Grass(r,\CV)_M \subset \Grass(r,\CV)$.
More elaborate geometric invariants, also closed subvarieties of $\Grass(r,\CV)$, extend the
generalized support varieties of \cite{FP4}.

The generalization to $r > 1$ raises many interesting questions for which we have only partial
answers.     For example,  even though the rank $r$ shifted subgroups 
are parametrized by  $\Grass(r,\CV)$, 
 for $r > 1$ the Zariski
topology on this Grassmannian is not easily obtained from the representation theory of $kE$.  
This stands in stark contrast with the situation for $r=1$ where the realization 
theorem asserts that any closed subvariety of the support variety of a finite group $G$ is realized 
as a support (equivalently, rank) variety of some finite dimensional representation of $G$ as proved in  
\cite{C2}.
For $r=1$, the Avrunin--Scott's theorem says that the rank variety of a $kE$-module $M$ has an 
interpretation in terms of the action of $\HHH^*(kE,k)$ on $\Ext^*_{kE}(M, M)$; we know of no 
such cohomological interpretation for $r > 1$. Theorem \ref{max-rad} is a  partial generalization to $r \geq 1$ 
of the fundamental theorem of \cite{FPS} concerning maximal Jordan type, yet we do not have the full 
generalization to all radical ranks.

We verify that the classes of $kE$--modules of constant $r$--radical type or constant $r$--socle 
type share some of the good properties of the class of modules of constant Jordan type.   Informed
by a variety of examples, we develop some sense of the complicated nature and
independence of the condition of being of constant socle  versus radical type. 
Many of our examples have very rich symmetries and, hence,  have constant $r$--radical type 
and $r$--socle type for all $r$, $\ 1 \leq r < n$.  On the other hand, in Section \ref{se:quantum} 
we introduce modules arising from quantum complete intersections which have much 
less symmetry and, therefore, much more intricate properties. In particular, we
exhibit $kE$--modules which have constant 2--radical type but not constant 2--socle type. Using Carlson 
 modules $L_\zeta$ in Section \ref{se:Lzeta}, we produce examples of modules which have constant $r$--radical type 
for a given $r$, $1 < r < n$, but not constant $s$--radical
type for any $s, \ 1 \leq s < r$. We also construct modules which have constant $s$--radical type 
all $s, \ 1 \leq s < r < n$, but not constant $r$--radical type.   Thanks to the duality of
radicals of $M$ and socles of $M^\#$, examples of constant radical types lead to examples of 
constant socle types.

The second part of the paper is dedicated to the construction 
of algebraic vector bundles on $\Grass(r,\CV)$ 
associated to $kE$-modules of constant 
$r$-radical type or constant $r$-socle type (and, more generally, to $kE$-modules
of constant $r$--$\Rad^j$ rank or constant $r$--$\Soc^j$ rank for $j, \ 1 \leq j \leq r(p-1)$, 
as defined in \eqref{const-rad}).    
All are associated to images or kernels of the restrictions of the $kE$-module $M$ to 
rank $r$ shifted  subgroups $C \subset kE$ indexed by points
of $\Grass(r, \CV)$.  We construct these bundles using various complementary techniques:
\begin{enumerate}
\item  patching  images or kernels of local operators on standard affine open subsets of  $\Grass(r, \CV)$ (Section \ref{se:local});
\item  applying equivariant descent to images or kernels of global operators on Stiefel varieties over 
$\Grass(r, \CV)$ (Section \ref{se:equiv}); 
\item  investigating explicit actions on graded modules over the homogeneous coordinate ring
of $\Grass(r, \CV)$ generated by Pl\"ucker coordinates (Section \ref{global}).
\end{enumerate}   
We mention a few specific results of this paper.  In Section \ref{se:grass}, we investigate
the generalization $\Grass(r,\CV)_M \subset \Grass(r,\CV)$ of the classical rank variety of
a $kE$-module $M$; the choice of $\CV \subset \Rad(kE)$ is less restrictive than the classical
choice of a basis of $\Rad(kE)$ modulo $\Rad^2(kE)$.    As shown in Corollary \ref{nonmax-closed},
$\Grass(r,\CV)_M \subset \Grass(r,\CV)$ and its refinements are closed subvarieties of 
$\Grass(r,\CV)$; moreover, in Corollary \ref{idp}, we show that $\Grass(r,\CV)_M$ is essentially
dependent only upon $M$ and not upon a choice of $\CV \subset \Rad(kE)$.  
In Section \ref{se:const}, we 
consider various classes of modules which have constant $r$--radical type and constant
$r$--socle type for all $r$. 
The examples of Sections \ref{se:quantum} and \ref{se:Lzeta} reveal some of the subtle possibilities
for restrictions of  $kE$-modules to rank $r$ shifted  subgroups $C$
of $kE$.  The quantum complete intersections of Section \ref{se:quantum} are perhaps new,
and certainly not fully understood.  The Carlson modules $L_\zeta$ of Section \ref{se:Lzeta}
show a surprising variability of behavior.

Section \ref{se:construction} contains two constructions of bundles arising from modules of constant 
socle or radical type. In Proposition~\ref{localdef}, we show that kernels and images of some local operators 
defined via explicit equations on principal affine opens of $\Grass(r, \CV)$ patch together to give 
globally defined coherent sheaves associated to a given $kE$-module $M$, $\cKer^\ell(M)$ and $\cIm^\ell(M)$.  
Theorem~\ref{th:bundle} proves that starting with a $kE$--module of constant socle or radical type we get a 
locally free sheaf (equivalently,  an algebraic vector bundle) on $\Grass(r, \CV)$.  Finally, in 
Theorem~\ref{pr:coincide},  we prove that the local construction of bundles coincides with the construction
by equivariant descent as described in \S \ref{se:descent}. 

In Section \ref{se:equiv}, we concentrate on algebraic vector bundles 
on $\Grass(r, \CV)$ associated to various $\GL_n$-equivariant $kE$-modules introduced in Definition~\ref{eq}.  
For such  $kE$-modules, Theorem~\ref{thm:constr} provides a useful method of determining their associated
vector bundles on $\Grass(r, \CV)$ using a standard construction from the representation theory of reductive 
algebraic groups.  We find that many familiar vector bundles on $\Grass(r, \CV)$ arise in this manner and 
fill the second half of Section \ref{se:equiv}  with examples.
To demonstrate the explicit nature of our techniques, we show 
 in Section \ref{global} how
 to calculate (typically, with the aid of a computer) ``generators" of kernel bundles
 arising from homogeneous elements of graded modules over the homogeneous coordinate
 algebra of $\Grass(r, \CV)$.  
 
 The appendix, written by the first author, shows how one can  
 calculate explicitly generalized rank varieties for small examples using MAGMA.  
 Any reader interested in obtaining the programs used for calculations should contact 
 the first author.

Throughout this paper, $k$ will denote an algebraically closed field of characteristic $p>0$.

The authors gratefully acknowledge  the hospitality of  MSRI where this project got started.  
They would also like to thank Steve Mitchell and S\'andor Kov\'acs for very helpful 
conversations.


\section{The $r$-rank variety $\Grass(r,\CV)_M$}
\label{se:grass}

Throughout this section, $E$ is an elementary abelian $p$ group of rank $n \geq 1$
and $r$ is a fixed integer satisfying $1 \leq r \leq n$.  Recall that the group algebra
$kE$ is isomorphic to the truncated polynomial algebra $k[x_1,\ldots,x_n]/(x_1^p,\ldots,x_n^p)$.
We choose a subspace $\CV \subset \Rad(kE)$ of the radical of $kE$ with
 the property that $\CV$ is a choice of splitting of the projection 
$\Rad(kE) \to \Rad(kE)/\Rad^2(kE)$;
in other words, the composition $\rho_\CV: \CV \to \Rad(kE) \to \Rad(kE)/\Rad^2(kE)$
is an isomorphism.  Observe that if $\CW \subset \Rad(kE)$ is another choice of 
splitting, then there is a unique map $\psi: \CV \to \CW$ such that 
$\rho_\CW \circ \psi = \rho_\CV$; that is, the following diagram commutes:
\[\xymatrix{\CV\ar[dr]_-{\rho_\CV}\ar[rr]^-\psi && \CW\ar[dl]^-{\rho_\CW}\\
&\frac{\Rad(kE)}{\Rad^2(kE)}}\]

Our choice of $\CV \subset \Rad(kE)$ provides an identification
\begin{equation}
\label{identify}
S^*(\CV)/\langle v^p, v \in \CV \rangle \ \cong kE
\end{equation}
which we employ throughout this paper.

For $r=1$, rank varieties were originally defined in terms of a choice of $\CV \subset \Rad(kE)$
together with a choice of ordered basis for $\CV$; these $r=1$ rank varieties have an 
interpretation in terms of cohomology, and thus are independent of such choices.  More refined 
support varieties for $r=1$ are also independent of such choices, thanks to results of \cite{FPS}.
For $r > 1$, we do not have a cohomological interpretation of $r$-rank varieties, so 
that we take some care in establishing invariance properties.  In particular, we consistently
avoid specifying an ordered basis of $\CV$.

We consider $r$-planes $U \subset \CV$ (i.e., subspaces of the $k$-vector 
space $\CV$ of dimension $r$). We recall the projective algebraic variety $\Grass(r,\CV)$ whose (closed) points are $r$-planes 
of  $\CV$.    We construct this Grassmannian by fixing some $r$-plane $U_0 \subset \CV$
and considering the set of $k$-linear maps of maximal rank
\begin{equation}
\label{hom}
\Hom_k(U_0,\CV)^o \ \subset \ \Hom_k(U_0,\CV);
\end{equation}
then 
$$\Grass(r,\CV) \ \equiv \  \GL(\CV)/\Stab(U_0) \ \cong   \  \Hom_k(U_0,\CV)^o/\GL(U_0) .$$
In particular, we observe for later use that there is a natural transitive (left) action
of $\GL(\CV)$ on $\Grass(r,\CV)$.
We view $\Hom_k(U_0,\CV)^o \to \Grass(r,\CV)$ as the principal $\GL(U_0)$--bundle whose 
fiber above an $r$-plane $U \in \Grass(r,\CV)$ consists of vector space bases of $U$.
Should we provide $\CV$ with an ordered basis and choose $U_0$ to be the span of
the first $r$ basis elements, then we can identify $\Hom_k(U_0,\CV)$
with the affine space $\bA^{nr}$ and view $\Grass(r,\CV)$ as 
\begin{equation}
\label{pinned}
 \Grass_{n,r}  \ \equiv \ \GL_n/\P_{r,n-r} \cong \ M_{n,r}^o/\GL_r,
 \end{equation}
 where $M_{n,r}$ is the affine space of $n\times r$-matrices, where $M_{n,r}^o \subset M_{n,r}$
 consists of those matrices of rank $r$, and
where $\P_{r, n-r} \simeq \ \Stab(U_0)$ is the standard parabolic subgroup stabilizing the 
vector $[\underbrace{1, \ldots, 1}_{r}, 0,\ldots, 0]$  in the standard representation 
of $\GL_n$.

We  employ the Pl\"ucker embedding
$\fp: \Grass(r,\CV)\ \hookrightarrow  \ \bP(\Lambda^r(\CV))$
of $\Grass(r,\CV)$, providing $\Grass(r,\CV)$ with the structure of a closed subvariety
of projective space.   Once we choose an ordered basis for $\CV$, this embedding can be
described explicitly as follows.   The inclusion (\ref{hom}) becomes $M_{n,r}^o \ \subset M_{n,r}$.
For any subset 
 $\Sigma \subset \{ 1,\ldots, n\}$ of cardinality $r$,  the {\it $\Sigma$--submatrix} of 
 an $n\times r$ matrix $A \in M_{n,r}$ is the $r\times r$ 
 matrix obtained by removing all rows indexed by
numbers not in $\Sigma$.
The Pl\"ucker coordinates $\{\mathfrak p_\Sigma(U) \}$ of the $r$-plane $U \in \GL_n/\P_{r,n-r}$ 
are the entries of the ordered $n$-tuple (well defined up to scalar multiple) 
obtained by taking any matrix $A \in \GL_n$ representing $U$ and setting $p_\Sigma(U)$ 
equal to the determinant of the $\Sigma$--submatrix 
of $A$.   In these terms,
the Pl\"ucker embedding  becomes
\begin{equation}
\label{puck}
\fp: \Grass_{n,r} \ \hookrightarrow   \bP^{{n \choose r}-1}, \quad  U \ \mapsto 
[\fp_\Sigma(A_U)].
\end{equation}
The homogeneous coordinate ring of the Grassmanian can be written as
the quotient of the polynomial ring on ${n \choose r}$ variables $\{ \fp_\Sigma \}$ by
the homogeneous ideal generated by standard Pl\"ucker relations.

We investigate $kE$-modules by considering their restrictions along flat maps 
$$k[t_1,\ldots,t_r]/(t_i^p) \ \to \ kE,$$
where we use $k[t_1,\ldots,t_r]/(t_i^p) $ to denote $k[t_1,\ldots,t_r]/(t_1^p,\ldots,t_r^p)$.
To give such a map is to choose an ordered $r$-tuple of elements of $\Rad(kE)$ which
are linearly independent modulo $\Rad^2(kE)$.   We formulate our consideration 
so that our maps are parametrized by $U \in \Grass(r,\CV)$.

For any $r$-plane $U \in \Grass(r,\CV)$, we define the finite dimensional commutative $k$
algebra 
$$C(U) \ \equiv \ S^*(U)/\langle u^p, u \in U \rangle \ \simeq \ k[t_1,\ldots,t_r]/(t_i^p)$$
to be the quotient of the symmetric algeba $S^*(U)$ by the ideal generated by $p$-th powers
of elements of $U \subset S^*(U)$.   We naturally associate to each $U \in \Grass(r,\CV)$ 
the map of $k$-algebras
\begin{equation}
\label{alphaU}
\alpha_U: C(U) \ \to \ kE
\end{equation}
induced by $S^*(U) \to S^*(\CV)$, the projection $S^*(\CV) \to S^*(\CV)/\langle v^p, v\in \CV\rangle$,
and the identification of (\ref{identify}).

The following characterization of flatness for certain maps of $k$-algebras applies in particular
to show that $\alpha_U$ is flat.    The essence of the 
proof of this fact (for $r=1$) is present  in \cite{C}.  Recall that a finitely generated module
over a commutative, local ring (such as $C$) is flat if and only if it is free.   If 
$\alpha: C \to A$ is a homomorphism of $k$-algebras and $M$ is a $C$-module, then
we denote by $\alpha^*(M)$ the restriction of $M$ along $\alpha$.

\begin{prop} \label{flat}
Consider a $k$-algebra homomorphism
$$\alpha: C \ \equiv \ k[t_1,\ldots,t_r]/(t_1^p,\ldots,t_r^p) \ \longrightarrow \ 
k[x_1,\ldots,x_n]/(x_1^p,\ldots,x_n^p) \equiv A. $$ 
The map $\alpha$ is flat  as a map of $C$-modules if and only if
the images of 
$\alpha(t_1), \dots, \alpha(t_r)$ in $\Rad(A)/\Rad^2(A)$ are linearly independent. 
\end{prop}

\begin{proof}   We first consider the case $r=1$ so that $C = k[t_1]/t_1^p$.  Write
$$
\alpha^\prime: C \ \to \ A, \quad \quad \alpha^\prime(t_1) \stackrel{\rm def}{=} a_1x_1 + \ldots + a_nx_n \equiv \alpha(t_1) 
 \, \mod \,   \Rad^2(A).
 $$ 
By \cite[9.5.10]{CTVZ} or 
\cite[2.2]{FP1}, $A$ is a free $C$-module with respect to $\alpha: C \to A$ if and only
if $\alpha(t_1)$ acts freely on $A$ if and only if
$A$ is a free $C$-module with respect to $\alpha^\prime: C \to A$.
Hence, we may replace $\alpha$ by $\alpha^\prime$.   Applying 
a linear automorphism to $A$ which maps $\alpha(t_1)$ to $x_1$, 
we may assume that $\alpha(t_1) = x_1$.   For $A$, given the structure of a $C$-module through
such a map $\alpha$, it is clear that $A$ is free as a $C$-module.

We now assume $r > 1$ and equip $A$ with the structure of a $C$-module through the given
$k$-algebra homomorphism $\alpha: C \to A$.
By Dade's Lemma (\cite{Dade}, \cite[5.3]{FP2}), $A$ is free as a $C$-module if and only if
$\beta^*(A)$ is free as a $k[t]/t^p$-module for every non-zero $k$-algebra homomorphism
$$\beta: k[t]/t^p\  \to \ C, \quad \beta(t) = b_1t_1 + \cdots + b_rt_r \not= 0.$$
Applying the case $r=1$, $A$ is free as a $C$-module if and only if $\alpha\circ \beta(t) \ \not\equiv  0
\ \text{mod} \Rad^2(A)$ for all (non-zero) $\beta$ which is the case if and only if the images of 
$\alpha(t_1), \dots, \alpha(t_r)$ in $\Rad(A)/\Rad^2(A)$ are linearly independent.
\end{proof}

We now introduce the $r$-rank variety of a finite dimensional $kE$-module $M$.

\begin{defn}
For any finite dimensional $kE$-module $M$, we denote by 
$$\Grass(r,\CV)_M \ \subset \ \Grass(r,\CV)$$
the set of those $r$-planes $U \in \Grass(r,\CV)$
with the property that $\alpha_U^*(M)$ is a free $C(U)$-module
(where $\alpha_U$ is given in (\ref{alphaU})).  
We say that $\Grass(r,\CV)_M$ is the $r$-rank variety of $M$.
\end{defn}

\begin{remark}
As shown in  Corollary \ref{idp}, $\Grass(r,\CV)_M$ is independent of
the choice of $\CV$ in the sense that if $\CW \subset \Rad(A)$ is another 
choice of splitting for the projection $\Rad(A) \to \Rad(A)/\Rad^2(A)$, 
then the unique isomorphism $\psi: \CV \to \CW$ commuting with the
projections to $\Rad(A)/\Rad^2(A)$ induces an isomorphism
$\Psi: \Grass(r,\CV)_M  \stackrel{\sim}{\to} \Grass_r(\CW)_M$.
\end{remark}

The following interpretation of  $\Grass(r,\CV)_M$ in terms of classical
(i.e., $r=1$) rank varieties follows immediately from Dade's Lemma asserting that
a $C$-module $N$
 is free if and only $\beta^*(N)$ is a free $k[t]/t^p$-algebra for every 
 $\beta: k[t]/t^p\  \to \ C$, with $\beta(t) = b_1t_1 + \cdots +b_rt_r \not= 0.$

\begin{prop}
\label{r1}
For any finite dimensional $kE$-module $M$ and any $r$-plane $U \in \Grass(r,\CV)$,
$$\Grass(r,\CV)_M \ = \ \{ U \in \Grass(r,\CV); \Grass(1,U)_{\alpha_U^*(M)} \ \not= \emptyset\}.$$
\end{prop}

\vskip .2in

We employ the notation of (\ref{hom}) and (\ref{pinned}).
The projective variety $\Grass_{n,r}$ has an open covering by affine pieces
$\CU_\Sigma \simeq \mathbb A^{(n-r)r}$, the $\GL_r$-orbits 
of  matrices $A = (a_{ij})$ such that $\fp_\Sigma(A)
\not= 0$,
$$\cU_\Sigma \ \equiv \ \mathfrak{p}^{-1}(\bP^{{n \choose r}-1}
\backslash Z(\mathfrak p_\Sigma))
 \ \subset \  \Grass_{n,r}.$$
 We consider the section of $M_{n,r}^o \to \Grass_{n,r}$ 
over $\CU_\Sigma$ defined by sending a $\GL_r$-orbit
to its unique representative such that the $\Sigma$--submatrix 
is the identify matrix.

Suppose that $\Sigma = \{i_1, \dots, i_r\}$ with $i_1 < \dots < i_r$.
Our choice of section identifies $k[\CU_\Sigma]$ with the quotient
\begin{equation}
\label{eq:YSigma}
k[M_{n,r}] = k[Y_{i,j}]_{1\leq i \leq n,  1 \leq j \leq r } \  \longrightarrow
 \
 k[Y_{i,j}^\Sigma]_{i \notin \Sigma, 1\leq j \leq r} = k[\CU_\Sigma]
 \end{equation}
sending $Y_{i,j}$  to 1, if $i = i_j  \in \Sigma$; to 0 if $i=i_{j^\prime} \in \Sigma$ and $j \not= j^\prime$; 
and to $Y_{i,j}^\Sigma$ otherwise.  
For notational convenience,  we set $Y_{i,j}^\Sigma$  equal to 1, if
$i\in \Sigma$ and $i = i_j$, and we set $Y_{i,j}^\Sigma = 0$ if $i=i_{j^\prime} \in \Sigma$ 
and $j \not= j^\prime$.

\begin{defn}
\label{alpha-S}
For any $\Sigma = \{i_1, \dots, i_r\}$ with $i_1 < \dots < i_r$, we define the map of $k[\cU_\Sigma]$-algebras
\[
\alpha_\Sigma: C \otimes k[\cU_\Sigma] = k[t_1,\ldots,t_r]/(t_i^p) \otimes k[\cU_\Sigma] \ \to 
k[x_1,\ldots,x_n]/(x_i^p) \otimes k[\cU_\Sigma] = kE \otimes k[\cU_\Sigma] 
\]
via 
\[
t_j \mapsto \sum_{i=1}^n x_i \otimes Y_{i,j}^\Sigma.
\]\end{defn}

\vskip .1in

Pick a basis for $\CV$ and choose $U_0$ to be the span of the first $r$ basis elements.  
For any $U \in \cU_\Sigma \subset \Grass(r,\CV)$, these choices enable us to identify 
$\alpha_U: C(U)\to kE$ with  the result of specializing  $\alpha_\Sigma$ by setting
the variables  $Y_{i,j}^\Sigma$ to values $a_{i,j} \in k$,
where $A_U = (a_{i,j}) \in M_{n,r}$ is the unique representation of $U$ whose
$\Sigma$--submatrix  is the identity.

\begin{prop}
\label{closed}
For any finite dimensional $A$-module $M$, $\Grass(r,\CV)_M \ \subset \ \Grass(r,\CV)$
is a closed subvariety.
\sloppy{

}
\end{prop}

\begin{proof}
It suffices to pick an ordered basis for $\CV$ and thus work with $\Grass_{n,r}$. 
 It further
suffices to show that for any  $\Sigma \subset \{ 1,\ldots, n\}$ of cardinality $r$, 
$$\cU_\Sigma \cap (\Grass_{n,r})_M \ \subset \ \cU_\Sigma$$
is closed.    
Having made a choice of ordered basis for $\CV$ and a choice of $\Sigma$ with 
$U \in \cU_\Sigma$, we may identify $C(U)$ with $C = k[t_1,\ldots,t_r]/(t_i^p)$
and thus identify $\alpha_U$ as a map of the form $\alpha_U: C \to kE$.
The condition that the finite dimensional $C$-module $ \alpha_U^*(M)$  is not free 
is equivalent to the condition that 
\begin{equation}
\label{cl}
\Dim(\Rad(\alpha_U^*(M)) \ < \ \frac{p^r-1}{p^r} \cdot \Dim(M).
\end{equation}

We consider the $k[\cU_\Sigma]$-linear map of free
$k[\cU_\Sigma]$-modules
\begin{equation}
\label{vector}
\sum_{i=1}^r  \alpha_\Sigma(t_i):    (M \otimes k[\cU_\Sigma])^{\oplus r} \ \to \ M \otimes k[\cU_\Sigma].
\end{equation}
 Denote by
$\Phi(M) \in M_{m,rm}(k[\cU_\Sigma])$ the associated matrix, where $m = \Dim M$.    The rank of
the specialization of $\Phi(M)$ at some point of $U \in \cU_\Sigma$ equals the dimension
of $\Rad(C) \cdot \alpha_U^*M$, 
\begin{equation}
\label{spec-rank}
\rk(\Phi(M) \otimes_{k[\cU_\Sigma]} k)  \ = \  \Dim(\Rad(  \alpha_U^*(M)), 
\end{equation} 
where $k[\cU_\Sigma] \to k$ is evaluation at $U$ represented by $A_U\in M_{n,r}$
with  $\Sigma$--submatrix equal to the identity.    

The fact that (\ref{cl}) is a closed condition follows immediately from  the lower 
semi--continuity of $\rk(\Phi(M))$ as a function on  $\cU_\Sigma$.
\end{proof}

\begin{ex} \label{nonmax-ex}
Suppose that $n = 4$, and choose $[x_1,x_2,x_3,x_4]$ spanning $\CV \subset \Rad(kE)$ 
determining an (ordered) basis for $\Rad(kE)/\Rad^2(kE)$.
Take $r=2$.  Set $M = kE/(x_1, x_2)$. 
Then $(\Grass_{4,2})_M$ consists of 
all 2-planes which intersect  non-trivially the plane $\langle x_1,x_2 \rangle$ 
spanned by $x_1$ and $x_2$.   Namely, $\alpha_U^*M$ is a free 
$C = k[t_1,t_2]/(t_1^p,t_2^p)$-module if and only if the 2-plane $U\subset \CV$
does not intersect $\langle x_1,x_2 \rangle$.    Take $u_1 = \sum_{j=1}^4 u_{1,j}x_j, 
u_2 = \sum_{j=1}^4 u_{2,j}x_j$ spanning $U$. Then $U$ does not intersect 
$\langle x_1,x_2 \rangle$ if and  only if 
the vectors $\{x_1, x_2, u_1, u_2\}$ span $\CV$. This is equivalent to non-singularity of the matrix 
\[\begin{pmatrix} 1&0&0&0\\
0&1&0&0\\
u_{11}&u_{12}&u_{13}&u_{14}\\
u_{21}&u_{22}&u_{23}&u_{24}
\end{pmatrix}. 
\]
Hence, 
in Pl\"ucker coordinates, $(\Grass_{4,2})_M$ is the zero locus of  $\fp_{\{3,4\}} = u_{13}u_{24}-u_{23}u_{14} = 0$.
\end{ex}

\vspace{0.1in}

For $r=1$, $\Grass(1,\CV)_M \ \subset \  \Grass(1,\CV) \  \simeq \ \bP^{n-1}$ 
can be naturally identified with the projectivized support variety of the $kE$-module $M$
(see \cite{CTVZ}).  The following proposition extends to all $r \geq 1$ various familiar properties
of support varieties.  As usual,  $\Omega^s(M)$ is the name of  the
$s^{th}$ syzygy or $s^{th}$ Heller shift of the $kE$-module $M$. 
(We also use this notation for the Heller shift of any $C$-module, where $C$
is a commutative $k$-algebra of the form $k[t_1,\ldots,t_r]/(t_i^p)$.)
Recall that  $\Omega(M)$ is the kernel of a projective 
cover $Q \to M$ of $M$, and $\Omega^{-1}(M)$ is the cokernel of an 
injective hull $ M \to I$. Then inductively, 

\begin{equation}
\label{omegaZ}
\Omega^s(M) \ = \ \Omega(\Omega^{s-1}(M)), \quad \Omega^{-s}(M) \ = \
\Omega^	{-1}(\Omega^{-s+1}(M)), \quad s>1.
\end{equation}

\begin{prop}
\label{prop}
Let $M$ and $N$ be finite dimensional $kE$-modules, and fix an integer $r \geq 1$.
\begin{enumerate}
\item $M$ is projective as a $kE$-module if and only if $\Grass(r,\CV)_M = \emptyset$.
\item $\Grass(r,\CV)_{M\oplus N} \ = \ \Grass(r,\CV)_M  \cup \Grass(r,\CV)_N$.
\item $\Grass(r,\CV)_{\Omega^i (M)} = \Grass(r,\CV)_M$ for any $i\in \mathbb Z$.   
\item If $0 \to M_1 \to M_2 \to M_3 \to 0$ is an exact sequence of $kE$-modules, then 
$$\Grass(r,\CV)_{M_2}  \ \subset \  \Grass(r,\CV)_{M_1} \cap \Grass(r,\CV)_{M_3}.$$
\item $\Grass(r,\CV)_{M\otimes N}  \ \subset \ \Grass(r,\CV)_M \cap \Grass(r,\CV)_N$.
\end{enumerate}
\end{prop}

\begin{proof}
The assertion (1) follows from Proposition \ref{r1} together with Dade's Lemma.  
Assertion (2) is immediate.  The assertion (3) follows from Proposition ~\ref{r1}, the observation 
that the restriction of $\Omega^i(M)$ along some $\alpha_U: C(U) \to kE$ is stably isomorphic
to the $i$-th Heller shift of the restriction of $M$ along $\alpha_U$, and the corresponding
result for $r=1$. 

To prove (4), we first observe that if the restrictions along
$\alpha_U$ of both $M_1$ and $M_3$ are free, then the pull-back along
$\alpha_U$ of $0 \to M_1 \to M_2 \to M_3 \to 0$  splits and thus $M_2$ is also free.

Complicating the proof of (5) is the fact that, in general, the restriction functor 
along $\alpha_U$ does not commute with tensor products, for the tensor product operation
depends upon on the choice of Hopf algebra structure.
We use the fact proved in \cite{C} (see also \cite{FP2}), that
\begin{equation}
\label{gr}
\Grass(1,\CV)_{M\otimes N}  \ = \ \Grass(1,\CV)_M \cap \Grass(1,\CV)_N,
\end{equation}
without regard to Hopf algebra structures.
If $U$ is in $\Grass(r,\CV)_{M \otimes N}$, then
$\alpha_U^*(M \otimes N)$ is not a free module. So there exists $\beta:k[t]/(t^p) \to C(U)$
such that $\beta^*(\alpha_U^*(M \otimes N))$ is not a free $k[t]/t^p$-module. 
Consequently, the line $W \in \Grass(1,\CV)$  generated by $\alpha(\beta(t))$ is in 
$\Grass(1,\CV)_{M\otimes N}$.   Thus, the line $W^\prime \subset U$ generated by $\beta(t)$
is in $\Grass(1,U)_{\alpha_U^*(M\otimes N)}$.   By (\ref{gr}), (with $\CV$ replaced by $U$),
$W^\prime$
is in both  $\Grass(1,U)_{\alpha_U^*(M)}$ and $\Grass(1,U)_{\alpha_U^*(N)}$. 
Therefore, neither $\alpha_U^*(M)$ nor $\alpha_U^*(N)$ is free, so that 
$U \in \Grass(r,\CV)_M \cap \Grass(r,\CV)_N$.
\end{proof}

\begin{ex}
The reverse inclusion of Proposition \ref{prop}(4) does not hold if $r \geq 2$. 
Retain the notation of Example \ref{nonmax-ex}.
 Let $U = \langle x_1,x_2 \rangle \subset \CV$, and let  $M = kE/(x_1)$ and $N = kE/(x_2)$. 
  Then $M \otimes N$ is a free $kE$-module
so that $\Grass(2,\CV)_{M\otimes N} = \emptyset$; however, neither $\alpha_U^*(M)$
or $\alpha_U^*(N)$ is free as a $C(U)$-module, so that 
$U \in \Grass(2,\CV)_M \cap \Grass(2,\CV)_N$.
\end{ex}

To end this section, we observe that it is not possible, in general, to realize all
of the closed sets of  $\Grass(2,\CV)$ as $2$-support varieties of $kE$-modules.  This contrasts
with the case $r=1$:  every closed subvariety of the usual support
variety $\Grass(1,\CV)$ is the  support
variety of a tensor product of Carlson modules $L_\zeta$ for suitably chosen cohomology
classes $\zeta \in \HHH^*(kE,k)$ (\cite{C2}).

\begin{ex} \label{not-somany-closed}
Take $n= 3$, so that $\Grass(2,\CV) \simeq \bP^2$. 
Recall that the complexity of a $kE$-module $M$ is the dimension of 
the affine support variety of $M$ (whose projectivization is $\Grass(1,\CV)_M$).
\begin{itemize}
\item  If $M$ has complexity 0, then $M$ is projective and $\Grass(2,\CV)_M = \emptyset$.
\item If $M$ has complexity 1, then the affine support variety of $M$ is a finite union of lines.  
Under the identification  $\Grass(2,\CV) \simeq \bP^2$,
the subvariety of planes $U \in \Grass(2,\CV)$ containing a given line is a line in $\bP^2$.
By Proposition \ref{r1}, $\Grass(2,\CV)_M$ consists of those $U \in \CV$ such that $U$
contains one of the lines whose union is the affine support variety of $M$.  Hence, the subsets in
 $\Grass(2,\CV) \simeq \bP^2$ of the form $\Grass(2,\CV)_M $ for $M$ of complexity 1
 are finite unions of lines.
\item If $M$ has complexity 2 or 3, then there are no 2-planes in $\CV$ 
which fail to intersect $\Grass(1,\CV)$. Consequently, $\Grass(2,\CV)_M = \Grass(2,\CV)$. 
\end{itemize}
Hence, the closed subsets of $\Grass(2,\CV)$ of the form $\Grass(2,\CV)_M$ do not generate
the Zariski topology of $\Grass(2,\CV)$.
\end{ex}


\section{Radicals and Socles} 
\label{se:radsoc}

We retain the notation of Section\,\ref{se:grass}: $E$ is an elementary abelian $p$-group
of rank $n$ and $\CV \subset \Rad(kE)$ is a splitting of the projection 
$\Rad(kE) \to \Rad(kE)/\Rad^2(kE)$.
As in Definition \ref{rad-U}, for a given $kE$-module $M$ we consider radicals and socles 
with respect to rank $r$ elementary subgroups parametrized by $U \in \Grass(r,\CV)$.
The dimensions of these radicals and socles are numerical invariants which in some sense
are the extension to
$r > 1$ of the Jordan type of a $kE$-module at a cyclic shifted subgroup (or the Jordan type
of a $\fu(\fg)$-module at a 1-parameter subgroup of a $p$-restricted Lie algebra $\fg$).

\begin{defn}
\label{rad-U}
Let $M$ be a $kE$-module,  $U \in \Grass(r,\CV)$ be an $r$-plane of $\CV$, and 
take $\alpha_U$ as in (\ref{alphaU}).   We define 
\[
\Rad_U(M) \equiv \Rad(\alpha_U^*(M)) = \sum_{u\in U}u \cdot M, \]
\[
\Soc_U(M)  \equiv \Soc(\alpha_U^*(M)) = \{m \in M \, | \, u \cdot m = 0 \  \forall u \in U\},
\]
the radical and socle of $M$ as a $C(U)$-module.
For $j > 1$, we inductively define the $kE$-submodules of $M$
$$\Rad_U^j(N) \ = \ \Rad_U(\Rad^{j-1}_U(M))$$ and
$$ \Soc^j_U(M) = \{m \in M \, | \overline m \in \Soc_U(M/\Soc_U^{j-1}(M)\}.$$
\end{defn}

\vskip .1in

Thus, if $\{u_1,\ldots,u_r\}$ spans $U$ and if $S_j(u_1, \ldots, u_r)  \subset \Rad(kE)$ denotes 
the subspace generated by all monomials on $\{u_1, \ldots, u_r\}$ of degree $j$, then 
\begin{equation}
\label{eq:rad}
\Rad_U^j(M) \ = \ \ \sum_{s\in S_j(u_1, \ldots, u_r)} s \cdot M
\end{equation}
and 
\begin{equation}
\label{eq:soc}
\Soc^j_U(M) \ = \ \{m \in M \ | \ s\cdot m = 0 \text{ for all } s \in S_j(u_1, \ldots, u_r) \}.
\end{equation}
The commutativity of $E$ implies that each $\Rad^j_U(M)$ and each $\Soc_U^j(M)$ is a 
$kE$-submodule of $M$.

If $A$ is a Hopf algebra and $f: L \subset M$ is an embedding of $A$-modules, then we 
denote by $f^\#: M^\# \to L^\#$ the induced map of $A$-modules and denote by 
$L^\perp \ \equiv \ \Ker\{ f^\# \}$.  Explicitly, the action of $A$ on $M^\#$ is given by sending
$a\in A, \ \phi: M \to k$ to  
$$a \cdot \phi: M \to k, \quad (a\cdot \phi)(m) = \phi (\iota(m)),$$
where $\iota: A \to A$ is the antipode of $A$; thus, the $A$-module structures on
$M^\#$ and $L^\perp$ depend upon the Hopf algebra structure on $A$, not just the structure
of $A$ as an algebra.

Although we assume throughout this paper that $kE$ is equipped with the Hopf algebra 
structure which is primitively generated (so that $kE$ is viewed as a quotient of the primitively
generated Hopf algebra $S^*(\CV)$), the following proposition is formulated to apply
as well to the usual group-like Hopf algebra structure of $kE$.

For the automorphism $\iota: kE\to kE$  defined by the antipode of $kE$, 
and a $kE$--module $M$, we denote by $\iota(M)$  the $kE$--module $M$ twisted by $\iota$. 
That is, $M$ coincides with $\iota(M)$ as a vector space but 
an element $x \in kE$ acts on $\iota(M)$ as $\iota(x)$ acts on $M$.  We denote an element of 
$\iota(M)$ corresponding to $m \in M$ by $\iota(m)$.

\begin{prop} \label{rad-dual}
Choose any Hopf algebra structure on $kE$, and let $\iota$ be the antipode of this structure.
For any $kE$-module $M$, let $\iota(M)$ denote the $kE$-module which coincides with $\iota(M)$ as 
a $k$-vector space and such that $x \in kE$ acts on $m \in \iota(M)$ as $\iota(x)\cdot m$.

For any $U \in \Grass(r,\CV)$ and any $j \geq 1$, there are natural isomorphisms of $kE$-modules
\begin{equation}
\label{perp}
\Soc^j_{U}(\iota(M)^{\#}) \ \simeq \  (\Rad_{U}^j(M))^{\perp}, \quad
\Rad^j_{U}(\iota(M)^{\#}) \ \simeq \ (\Soc^j_{U}(M))^{\perp},
\end{equation}
\end{prop}

\begin{proof}  
Choose a basis $\{u_1, \dots, u_r\}$ for $U$. 
An element  $\iota(f)$ in $\iota(M^{\#})$ is in  $\Soc^j_U(\iota(M^{\#}))$
if and only if  for any monomial $s$ of degree $j$ in the elements $u_1, \dots, u_r$,
we have that $s\cdot \iota(f) = 0$. This happens if and only if for any such $s$ and any
$m$ in $M$, $(s\cdot \iota(f))(m) = (\iota(s)f)(m) = f(sm) = 0$. In turn, this can happen 
if and only if $f$ vanishes on $\Rad^j_{U}(M)$. 
This proves the first equality; the proof of the second is similar. 
\end{proof} 

We introduce refinements of the $r$-rank variety $\Grass(r,\CV)_M$, thereby 
extending to $r > 1$ the generalized support varieties of \cite{FP4}.

\begin{defn}
\label{def:Gamma}
Let $M$ be a finite dimensional $kE$-module, and let $j$ be a positive integer.  
We define the {\it nonmaximal $r$--radical support variety} of $M$,
$\Rad^j(r,\CV)_M \ \subset \ \Grass(r,\CV)$, to be 
$$\Rad^j(r,\CV)_M \ \equiv  \ \{ U \in \Grass(r,\CV) \, | \, 
\Dim\Rad^j_U(M) < \max\limits_{U^\prime \in \Grass(r,\CV)} \Dim \Rad^j_{U^\prime}(M) \}.
$$
Similarly, we define the {\it nonminimal $r$-socle support variety} of $M$, $\Soc^j(r,\CV)_M 
 \ \subset \ \Grass(r,\CV)$, to be 
$$\Soc^j(r,\CV)_M \ \equiv \ \{ U \in \Grass(r,\CV) \, | \, 
\Dim\Soc^j_U(M) > \max\limits_{U^\prime \in \Grass(r,\CV)} \Dim \Soc^j_{U^\prime}(M) \}.
$$
\end{defn}

\vskip .1in

For $j=1$, we simplify this notation by writing 
$$\Rad(r,\CV)_M \ = \ \Rad^1(r,\CV)_M, \quad \Soc(r,\CV)_M \ = \Soc^1(r,\CV)_M.$$

The proof of upper/lower semi-continuity in the next theorem is an extension of the 
proof of Proposition \ref{closed}.

\begin{thm}
\label{semicont}
Let $M$ be a finite dimensional $kE$-module.   For any $j$, the function
$$U \in \ \Grass(r,\CV) \ \mapsto \  f_{M,j}(U) \equiv \Dim \Rad^j_U(M)$$
is lower semi-continuous:  in other words, there is a (Zariski) open subset $\cU \subset
\Grass(r,\CV)$ of $U$ such that $f_{M,j}(U) \leq f_{M,j}(U^\prime)$    for all $U^\prime \in \cU$.

Simillarly,   for any $j$ the function
$$U \in \ \Grass(r,\CV) \ \mapsto \  g_{M,j}(U) \equiv \Dim \Soc^j_U(M)$$
is upper semi-continuous.
\end{thm}

\begin{proof}
As in the proof of Proposition \ref{closed}, we may equip $\CV$ with an ordered basis,
replacing $\Grass(r,\CV)$ by $\Grass_{n,r}$.  It suffices to restrict to affine open
subsets  $\cU_\Sigma \subset \Grass(r,\CV)$.  Recall the notation $S_j(t_1, \ldots, t_r) \subset C$ 
for the linear subspace generated by all monomials on $\{t_1, \ldots, t_r\}$ of degree $j$, 
and  let $d(j) = \dim S_j(t_1, \ldots, t_r)$. 
We replace the  map (\ref{vector}) by 

\begin{equation}
\label{j-vector}
\sum_{\tiny{\begin{tabular}{cc} $d_1+\ldots+d_r=j$\\$0 \leq d_i<p$\end{tabular}}} \alpha_\Sigma(t_1)^{d_1}\ldots\alpha_\Sigma(t_r)^{d_r}: (M\otimes k[\cU_\Sigma])^{\oplus d(j)}
 \ \to \ M\otimes k[\cU_\Sigma],
\end{equation}

Let $\Phi^j(M) \in M_{m,d(j)m}(k[\cU_\Sigma])$ denote the associated matrix, where
$m = \Dim M$.   Then, as for (\ref{spec-rank}) with the same notation, we have the equality  
\begin{equation}
\label{j-spec-rank}
\rk(\Phi^j(M) \otimes_{k[\cU_S]} k)  \ = \  \Dim(\Rad^j(C)\cdot  \alpha_U^*(M)). 
\end{equation} 
The lower semi-continuity of $U \mapsto f_{M,j}(U)$ now follows immediately from the 
lower semi-continuity of $\Phi^j(M)$ as a function on $\cU_\Sigma$.

The upper
semi-cotinuity for $U \mapsto g_{M,j}(U)$ is a consequence of lower semi-continuity for
$U \mapsto f_{M^\#,j}(U)$ and Proposition \ref{rad-dual}.
\end{proof}

As an immediate corollary of Theorem \ref{semicont}, we conclude that the subsets introduced
in Definition \ref{def:Gamma} are Zariski closed subvarieties of $\Grass(r,\CV)$.

\begin{cor} \label{nonmax-closed}
For any finite dimensional $kE$--module $M$, and  any positive integer $j$, 
$\Rad^j(r,\CV)_M$ and $\Soc^j(r,\CV)_M$ 
are Zariski closed subsets of $\Grass(r, \CV)$.
\end{cor}

The reader should observe that the polynomial equations expressing the non-maximality of
$f_{M,j}(U)$  must be expressible
in terms of homogeneous polynomials in the Pl\"ucker coordinates. This fact is exploited in 
the appendix, where some computer calculations of 
nonminimal $r$-socle support varieties are presented.

\begin{ex} \label{nonmax-ex1}
We return to Example~\ref{nonmax-ex}, in which $n=4$ and $[x_1,x_2,x_3,x_4]$ is an
ordered basis of some $\CV \subset \Rad(kE)$ splitting the projection 
$\Rad(kE) \to \Rad(kE)/\Rad^2(kE)$.   As in Example~\ref{nonmax-ex}, we take 
$M = kE/\langle x_1,x_2 \rangle$. 

If  $r=2$, then an argument similar to the one in  Example \ref{nonmax-ex} shows that  
$\Rad(2,\CV)_M = \Grass(2,\CV)_M$.

Now, set $r=3$.  We have $\Grass(3,\CV) \simeq \bP^3$.   
Let $\langle x_1, x_2 \rangle \subset \CV$ be the 2-plane
spanned by $x_1,x_2$.   Observe that the module $M$ has dimension $p^2$. Let $U\subset \CV$ be 
any $3$-plane in $\CV$.  Then $\Rad_U M \subset M$ has codimension p if $\langle x_1, x_2 \rangle \subset U$ and codimension $1$ otherwise.  Hence, 
$\Rad(3,\CV)_M \ \not= \emptyset$; indeed, $\Rad(3,\CV)_M$ consists of all 3-planes 
which contain $\langle x_1, x_2 \rangle$. In Pl\"ucker coordinates $\Rad(3,\CV)_M $ is given as the zero locus of the equations
$\fp_{\{1,3,4\}} =0= \fp_{\{2,3,4\}}$.
\end{ex}

Our next example is more complicated and uses the identification of the rank variety
$\Grass(1,\CV)$ with  $\Proj \HHH^*(E,k)$.

\begin{ex} \label{nonmax-ex2}
Choose some $\CV \subset \Rad(kE)$ splitting the projection $\Rad(kE) \to \Rad(kE)/\Rad^2(kE)$, and
assume that $p=2$, $r = 2$.  Let  $\zeta \in \HHH^m(E,k)$ be a non-trivial 
homogeneous cohomology class of positive degree $m$.
Let $\zeta: \Omega^m(k) \to k$ be the cocycle 
representing $\zeta$  and let
$L_{\zeta}$ denote the kernel of the module map 
$\zeta$ (investigated in detail in Section \ref{se:Lzeta}). 
Recall that the support variety of $L_\zeta$
may be identified with the zero locus of $\zeta$, $Z(\zeta) \subset \Spec \HHH^*(E,k)$ 
(see \cite{C2}). 

There are two possibilities for the restriction of $L_\zeta$ along $\alpha_U: C \to kE$
for $U \in \Grass(2,\CV)$ (see Lemma~\ref{restr-Lzeta}): 
  
$$\left[\begin{array}{ll}\alpha^*(M) \simeq L_{\alpha^*(\zeta)} \oplus C^s & 
\text{ if } \alpha^*(\zeta) \neq 0  \\
\alpha^*(M) \simeq \Omega^m(k_C) \oplus \Omega(k_C) \oplus C^{ s-1} & \text{ if }  \alpha^*(\zeta) = 0
\end{array}
\right.,$$
where $2m+1 +4s = \Dim(\Omega^m(k))$. 
 In particular, 
 $\Grass(2,\CV)_{L_\zeta} \ = \ \Grass(2,\CV).$
 Since $C \simeq k (\Z/2 \times \Z/2)$, we can compute  
$\Dim \Rad(\Omega^m(k_C)) = \Dim(L_{\alpha^*(\zeta)}) = m$
and $\Dim \Rad(C) = 3$ (see \cite{He}). Hence, if $\alpha^*(\zeta) \neq 0$,
$\Dim \Rad_U(L_{\zeta}) = 3s+m$ while for $\alpha^*(\zeta) = 0$, 
$\Dim \Rad_U(L_{\zeta}) = 3(s-1)+m+1$. It follows 
that 
$\Rad(2,\CV)_{L_{\zeta}} \ \not= \  \emptyset, $ 
with $\Rad(2,\CV)_{L_{\zeta}} $ consisting
of exactly those $2$--planes that are contained in $Z(\zeta)$.
We can compute further that $\Dim \Rad^2_U(L_{\zeta}) = s$ 
in the first case 
and $s-1$ in the second. Hence, 
$$\Rad^2(2,\CV)_{L_{\zeta}}  \ = \ \Rad(2,\CV)_{L_{\zeta}}  \ \not= \  \emptyset.$$

Finally, we find a curious thing happens when we consider socles. 
The point is that $\Dim \Soc_U(L_{\zeta}) = s+m$ 
in both cases. Hence, $\Soc(2,\CV)_{L_\zeta} \  = \ \emptyset$, so that $L_\zeta$ has 
constant $2$--$\Soc$--rank in the terminology of Section \ref{se:const}.
However, $\Dim \Soc^2_U(L_{\zeta}) = 3s+2m$ if $\alpha^*(\zeta) \not = 0$ 
and $3s + 2m+1$ otherwise.  Consequently,   $\Soc^2(2,\CV)_{L_\zeta}$ is the 
same as the radical variety $\Rad(2,\CV)_{L_{\zeta}} $.   Thus,
$$\Soc^2(2,\CV)_{L_{\zeta}}  \ \not= \ \Soc(2,\CV)_{L_{\zeta}}  \ = \  \emptyset.$$

By taking duals, we can get a module $M$ with
the property that $\Rad(2,\CV)_M = \emptyset$ and $\Rad^2(2,\CV)_M$ is a 
proper non-trivial subvariety of $\Grass(2,\CV)$. 
\end{ex}

We conclude this section with a consideration of the dependence of the dimension
of radicals on the choice of $\CV$, continuing the investigation of \cite{FPS}.  
Our statements are given
for radicals, but using Proposition \ref{rad-dual} one immediately gets similar statements for
socles.

\begin{defn}
\label{def:abs} Fix a finite dimensional $kE$--module $M$. We say that $M$ has  
{\it absolute maximal radical rank}  at the $r$-plane $V \in \Grass(r,\CV)$ if 
\[\Dim \Rad_V(M) \geq \Dim \Rad_{W}(M)\]
 for any $\CW \subset \Rad(kE)$ splitting the 
 projection $\Rad(kE) \to \Rad(kE)/ \Rad^2(kE)$ and any $W \in \Grass(r,\CW)$.
\end{defn} 

The following theorem is a generalization to $r > 1$ of \cite[1.9]{FPS}.

\begin{thm}
\label{max-rad}  Let $\CV,\CW \subset \Rad(kE)$ be splittings of the projection 
$\Rad(kE) \to \Rad(kE) /\Rad^2(kE)$  and let $\psi: \CV \stackrel{\sim}{\to} \CW$ be the unique
isomorphism commuting with the projection isomorphisms to $ \Rad(kE) /\Rad^2(kE)$.    
Denote by $\Psi: \Grass(r,\CV)  \stackrel{\simeq}{\to}
\Grass(r,\CW)$  the induced isomorphism of Grassmannians.    Then
\begin{enumerate}
\item  
  $\Psi$ restricts to an isomorphism
$$ \Rad(r,\CV)_M \ \stackrel{\sim}{\to} \ \Rad(r,\CW)_M.$$
\item For any $U \not \in \Rad(r, \CV)_M$, $\Dim \Rad_U(M) \ = \ \Dim \Rad_{\Psi(U)}(M).$
\item
$$
\max\limits_{V \in \Grass(r,\CV)} \Dim \Rad_{V}(M) \quad = \quad
\max\limits_{W \in \Grass(r,\CW)} \Dim \Rad_{W}(M).
$$
\end{enumerate} 
\end{thm}

\begin{proof} 
We first assume that $\CV$ satisfies the condition that there exists some $r$-plane $V \subset
\CV$ at which $M$ has absolute radical rank.
Since for any $U, U^\prime \not \in \Rad(r, \CV)_M$, we have an equality 
$\Dim \Rad_U(M) = \Dim \Rad_{U^\prime}(M)$, we immediately conclude that any 
$U^\prime \not \in \Rad(r, \CV)_M$ 
satisfies the property that $M$ has absolute maximal radical rank at $U^\prime$. Hence, the 
validity of
statements (1) and (3) will follow from the validity of statement (2) since $\Psi$ is a bijection. 

Let $U \in \Grass(r,\CV)$ satisfy the property that $M$ has absolute maximal radical rank at $U$.
Choose an ordered basis $[u_1, \dots, u_r]$ of $U$.  For each $m, \ 0 \leq m \leq r$, we consider
$\alpha_m: C \equiv k[t_1,\ldots,t_r]/(t_i^p) \to kE$ defined as follows: 
$$\left\{
\begin{array}{ll}\alpha_m(t_i) \ = \ \psi(u_i)& i \leq m\\
\alpha_m(t_i) \ =  \ u_i & m+1\leq i\leq r.\\
\end{array}
\right.
$$ 
Since $\psi$ commutes with the projections to $\Rad(kE) \to \Rad(kE)/\Rad^2(kE)$, we conclude that
$$\psi(u_i) - u_i \ \in \ \Rad^2(kE), \ 1 \leq i \leq r.$$

Observe that
\begin{equation}
\label{obs} 
\Rad_U(M) = \Rad(\alpha_0^*(M)), \quad \Rad_{\psi(U)}(M) = \Rad(\alpha_r^*(M)).
\end{equation}

Consider the $kE$-module $N = M/(u_2M + \ldots + u_rM)$.
We have
\begin{equation}
\label{eq:split}
\Dim \Rad(\alpha_0^*(M)) \ = \ \Dim \sum\limits_{i=1}^r u_iM \ = \ \Dim (u_1N)
+ \Dim \sum\limits_{i=2}^r u_iM 
\end{equation}
and
\begin{equation}
\label{eq:split2}
\Dim \Rad(\alpha_1^*(M)) \ =  \ \Dim \sum\limits_{i=1}^r \alpha_1(t_i)M \ =\ \Dim (\psi(u_1)N)
+ \Dim \sum\limits_{i=2}^r u_iM .
\end{equation}
Our assumption  that $\Rad(\alpha_0^*(M)) \ = \  \Rad_U(M)$ has absolute 
maximal rank and  equation \eqref{eq:split}
imply that 
$$\Dim (u_1 N) \ \geq  \Dim (u N), \quad \forall u \in \Rad(kE).$$
Together with the fact that $u_1 \equiv \psi(u_1) \mod \Rad^2(kE)$, this  implies 
the equality
\begin{equation}
\label{beta}
\Dim (u_1\cdot N) \ = \ \Dim (\psi(u_1) \cdot N)
\end{equation}
by \cite[1.9]{FPS}.
Equalities \eqref{eq:split}, \eqref{eq:split2}, and \eqref{beta}  now imply
$$
\Dim \Rad(\alpha_0^*(M)) \ = \ \Dim \Rad(\alpha_1^*(M)).
$$
We proceed by induction on $m \geq 1$, replacing $u_m$ by $\psi(u_m)$ as we just replaced 
$u_1$ by $\psi(u_1)$.  We conclude that  $\Dim \Rad(\alpha_{m-1}^*(M)) \ =
 \ \Dim \Rad(\alpha_{m}^*(M))$
for $1 \leq m \leq r$.    Thus, by (\ref{obs}), we obtain
$$\Dim \Rad(\alpha_U^*(M)) \ = \ \Dim \Rad(\alpha_{\psi(U)}^*(M)).$$

To prove the theorem without the condition that $\CV$ contains an $r$-plane 
at which $M$ has absolute maximal radical rank, we consider two arbitrary $\CV, \CW \subset
\Rad(A)$ subspaces which split the projection $\Rad(kE) \to \Rad(kE)/\Rad^2(kE)$ 
and choose some third $\CV^\prime \subset \Rad(kE)$ which also splits the 
projection $\Rad(kE) \to \Rad(kE)/\Rad^2(kE)$
and does contain an $r$-plane $V^\prime \subset \CV^\prime$ at which 
$M$ has absolute maximal rank.  Then appealing to the above argument for the 
pairs $(\CV^\prime, \CV)$ and $(\CV^\prime,\CW)$, we conclude the theorem for
the pair $(\CV,\CW)$.  
\end{proof}

\begin{cor}
\label{idp}
Retain the notation of Theorem \ref{max-rad}.  Then $\Psi$ restricts to an isomorphism
$$\Grass(r,\CV)_M \ \stackrel{\sim}{\to} \ \Grass(r,\CW)_M.$$
\end{cor}

\begin{proof}
For any $U \in \Grass(r,\CV)$, 
$\alpha_U^*(M)$ is free if and only if $\Rad_U(M)$ has dimension equal
to $\frac{p^r-1}{p^r} \cdot \Dim(M)$.   For any $V \in \Grass(r,\CV)$ we have the
inequality 
$$\Dim (\Rad_V(M))\  \leq \ \frac{p^r-1}{p^r} \cdot \Dim(M).$$  The corollary
now follows immediately from Theorem \ref{max-rad} (2).
\end{proof}

\section{Modules of constant  radical and socle rank}
\label{se:const}
We continue our previous notation:  $E$ is an elementary abelian $p$-group of rank $n$ and
$\CV \subset \Rad(kE)$ is a choice of splitting of the projection $\Rad(kE) \to \Rad(kE)/\Rad^2(kE)$
providing the  identification $S^*(\CV)/\langle v^p, v\in \CV \rangle \cong kE$ of (\ref{identify}).
As in Theorem~\ref{semicont}, we can associate to  any finite dimensional $kE$--module 
$M$ and any $j>0$ the integer-valued functions 
$$U \in \ \Grass(r,\CV) \ \mapsto \  f_{M,j}(U) \equiv \Dim \Rad^j_U(M)$$
and
$$U \in \ \Grass(r,\CV) \ \mapsto \  g_{M,j}(U) \equiv \Dim \Soc^j_U(M).$$
We view these functions as defining the {\it local radical ranks} and {\it local socle ranks} of $M$.

In this section we introduce $kE$--modules  of constant $r$-radical 
(resp., $r$-socle) type and more generally of constant $r$-$\Rad^j$-rank 
(resp., $r$-$\Soc^j$-rank). By definition, these are the modules for which 
the functions $f_{M,j}$  (resp., $g_{M,j}$) whose value $f_{M,j}(U)$ in 
independent of $U$ in $\Grass(r, \CV)$.  These are natural analogues 
for $r>1$ of modules of {\it constant Jordan type} (see \cite{CFP}) 
which have many good properties and lead to algebraic vector bundles (see \cite{FP3}).  
In Section~\ref{se:construction}, we see how to associate 
vector bundles on $\Grass(r,\CV)$ to $kE$-modules  of constant $r$-$\Rad^j$-rank 
or constant $r$-$\Soc^j$-rank.

\vspace{0.1in}

\begin{defn} \label{const-rad}
We fix integers $r > 0$, $j, \ 1 \leq j \leq (p-1)r$, and let $M$ be a finite dimensional $kE$-module.
\begin{enumerate}
\item
 The module $M$ has constant $r$-$\Rad^j$ rank (respectively, $r$-$\Soc^j$-rank)
 if the dimension of $\Rad^j_U(M)$ (resp., $\Soc ^j_U(M)$)
is independent of choice of $U \in \Grass(r, \CV)$.
\item
$M$ has constant $r$-radical type (respectively, $r$-socle type) if it has constant 
$r$-$\Rad^j$ rank (resp. $r$-$\Soc^j$ rank) for all $j, \ 1 \leq j \leq (p-1)r$.
\end{enumerate}
\end{defn}

\vskip .1in

To simplify notation, we refer to constant $r$-$\Rad^1$ rank 
(respectively, $r$-$\Soc^1$ rank) 
as constant $r$-$\Rad$ rank (respectively, $r$-$\Soc$ rank).

\begin{remark} It is  immediate from the definitions that $M$ has constant $r$-$\Rad^j$-rank 
(respectively, $r$-$\Soc^j$-rank)  if and  only if $\Rad^j(r, \CV)_M=\emptyset$ (resp, $\Soc^j(r, \CV)_M=\emptyset$).
\end{remark}

The following proposition, stating that the property of constant $r$-$\Rad$ and $r$-$\Soc$-rank is independent  of the choice of $\CV$, 
is an immediate corollary of Theorem \ref{max-rad}.

\begin{prop}
Let $\CW \subset \Rad(kE)$ also provide a splitting of the projection $\Rad(kE) \to
\Rad(kE)/\Rad^2(kE)$.  Then for any $kE$-module $M$ and any $r \geq 1$,
$M$ has constant $r$-radical rank (respectively, constant $r$-socle rank) as above
if and only if $\Dim\Rad_W(M)$ (resp., $\Dim\Soc_W(M)$) is independent
of the choice of $W \in \Grass(r,\CW)$.
\end{prop}

The reader should observe that in the case that $r =1$, 
either one of the set of $1$-$\Rad^j$
ranks or the set of $1$-$\Soc^j$ ranks, for all $j$, is 
sufficient to determine the Jordan type. Also the Jordan
type determines all of the radical and socle ranks for $r=1$.
Consequently, a $kE$-module has constant $1$-radical type
if and only if it has constant $1$-socle type.  This is no longer 
true for $r \geq 2$ as we show in Examples~\ref{benzol}, \ref{ex-2-rad}.

We begin with particularly easy examples of modules 
of constant radical and socle types.  Since their identification does not
depend upon the choice of $\CV \subset \Rad(kE)$, we conclude that these 
examples are examples of constant radical and socle types for any choice
of $\CV \subset \Rad(kE)$.

\begin{ex} \label{ex-projmod}
For any finite dimensional projective $kE$-module $M$, the $r$-radical type 
and the $r$-socle type of $M$ are constant for every $r > 0$.  Indeed,  a projective 
module is free and its restriction along  $\alpha_U: C(U) \to   kE$ is a free module
for any $U \in \Grass(r,\CV)$ whose rank is determined by $r$ and
the dimension of $M$.

Another evident family of examples of modules of constant radical and socle type
arises from Heller shifts of the trivial module (see (\ref{omegaZ}).  For any $s \in \bZ$,
if $M \simeq \Omega^s(k)$, then $M$ has constant $r$-radical type and constant $r$-socle type 
for each $r > 0$.  Indeed, for any $U \in \Grass(r,\CV)$, we have  $\alpha_U^*(M) \simeq
\Omega^s(k) \oplus Q$ as a $C(U)$-module, where $Q$ is a free $C(U)$-module whose
rank is determined by the dimension of $M$ and the choice of $r$. 
\end{ex}

Recall that we identify $kE$ with $S^*(\CV)/\langle v^p, v\in \CV \rangle$; with this identification, 
any $kE$-module is equipped with the structure of an $S^*(\CV)$-module.  Moreover, we get an action 
of $\GL_n\simeq\GL(\CV)$ on $kE$ by algebra automorphisms induced by the standard representation 
of $\GL_n$ on $\CV$.  We view $S^*(\CV)$ as the 
coordinate algebra of the affine space $\CV^\# = \bA^n$.  Thus, any $kE$-module $M$ determines a 
quasi-coherent sheaf $\wt M$ of $\cO_{\CV^\#}$-modules. The  natural action of $\GL_n = \GL(\CV)$ 
on $S^*(\CV)$ determines an action of $\GL_n = \GL(\CV)$ on the variety 
$\CV^\#$.  As recalled in Definition \ref{de:equiv1}, there is a widely used concept of a $\GL_n$-
equivariant sheaf on a variety $X$ which is provided with a
$\GL_n$-action.   In the special case of $\GL_n = \GL(\CV)$ acting on
$\CV^\#$, this specializes to the following explicit definition of a 
{\it $\GL_n$--equivariant $kE$--module}.

\begin{defn} 
\label{eq}
Let $M$ be a $kE$-module, whose structure map is given by the $k$-linear pairing
\begin{equation}
\label{struc}
S^*(\CV)/\langle v^p, v\in \CV \rangle \otimes M \ \to \ M.
\end{equation}
We say that $M$ is $\GL_n$-equivariant  (or $\GL(\CV)$-equivariant) if it is provided with a second $k$-linear pairing
\begin{equation}
\label{gaction}
\GL(\CV) \times M \ \to \ M, \quad (g,m) \mapsto gm
\end{equation}
such that for any $g \in \GL(\CV), x \in kE$, and $m \in M$, 
we have  
\[g(xm) = (gx)(gm).\]
In other words, the $\GL(\CV)$-action on $M$ of (\ref{gaction}) is such that the pairing
(\ref{struc}) is a map of $\GL(\CV)$-modules with $\GL(\CV)$ acting diagonally on the tensor product. 
\end{defn}

We employ the following notation: 
if $M$ is a $\GL_n$--equivariant $kE$-module and $N\subset M$ is a subset, then 
we denote by $gN$ the image of $N$ under the action of $g\in \GL_n$; if $U \in \Grass(r, \CV)$, 
then we denote by $gU$ the image of $U$ under the action of $g \in \GL(\CV)$ .

As we see in the next proposition, the abundant symmetries of $\GL_n$-equivariant $kE$-modules
imply that they have constant radical and socle types.

\begin{prop}
\label{prop:homog}
 Let $M$ be a $\GL_n$-equivariant $kE$-module. Then the following holds.
\begin{enumerate}
\item $M$ has  constant $r$-radical and $r$-socle type for any $r>0$.
\item For any $U\in \Grass(r,\CV)$, any $g \in \GL_n$,  and any $\ell$, $1\leq \ell \leq r(p-1)$, 
\[\Rad^\ell_{gU}(M)=g\Rad^\ell_U(M), \quad \Soc^\ell_{gU}(M) =g\Soc^\ell_U(M).
\]
\end{enumerate} 
\end{prop}
\begin{proof}
Clearly, (1) follows from (2). 
We prove (2) for $\Rad_U(M)$, the  other statements are similar. Let $\{u_1, \ldots, u_r\}$ be a
basis of  $U$.
We have 
\[
\Rad_{gU}(M) = \sum\limits_{i=1}^r(gu_i)M = \sum\limits_{i=1}^r g(u_i (g^{-1}M)) = \]
\[
g(\sum\limits_{i=1}^r u_i(g^{-1}M) = g\Rad_U (g^{-1}M) = g\Rad_U(M)
\] 
where the second and last equality hold since $M$ is $\GL_n$-equivariant. 
\end{proof}

Examples of $\GL_n$-equivariant $kE$--modules arise as follows.  
The identification $kE \simeq S^*(\CV)/\langle v^p, v \in \CV \rangle$ provides 
the $kE$-module 
$$\Rad^i(kE)/\Rad^{i+j}(kE)$$ 
with a $\GL(\CV)$-structure.  Thus, the subquotients
$S^{*\geq i}(\CV)/S^{*\geq j}(\CV)$ for $i \leq j$, are naturally modules 
over $S^*(\CV)$ with a $\GL(\CV)$ action.  If $j -i \leq p$, then the action of $S^*(\CV)$ on
these subquotient factors through the quotient map $S^*(\CV) \to kE$, so that
\begin{equation}
\label{s*}
S^{*\geq i}(\CV)/S^{*\geq j}(\CV) \quad  \text{ for } j-i \leq p
\end{equation}
inherits a $kE$-module structure. 
 
Let $kG=k[y_1, \ldots, y_n]/(y_1^{p^m}, \ldots, y_n^{p^m}) \simeq k((\Z/p^{m})^\times n) \simeq S^*(\CV)/(v^{p^m}, v \in \CV)$ 
for some $m>0$. Arguing exactly as above, we give 
\begin{equation}
\label{g*}
\Rad^i(kG)/\Rad^{j}(kG)
\end{equation} 
the structure of a  $kE$ module for $j - i \leq p$.

If $\Lambda^*(\CV)$ denotes the exterior algebra
on $\CV$,  then the $\GL(\CV)$-module
\begin{equation}
\label{l*}
\Rad^i(\Lambda^*(\CV))/\Rad^{i+2}(\Lambda^*(\CV))
\end{equation}
also inherits a $kE$-module structure. Note that the anticommutativity of  
$\Lambda^*(\CV)$ causes no problem in the definition of the action 
because it gives a relation in $\Rad^2(\Lambda^*(\CV))$.

It is straightforward to check that the $kE$ and $\GL_n$-actions described above 
are compatible, so that the $kE$-modules of \eqref{s*}, \eqref{g*},  and \eqref{l*} are  $\GL_n$-equivariant.
Proposition~\ref{prop:homog} thus implies the following.

\begin{prop} \label{sym-algs}
Each of the following $kE$-modules $M$ is $\GL_n$-equivariant. Consequently, 
each has  constant 
$r$-radical type and constant $r$-socle type for every $r>0$.
\begin{enumerate}
\item  $M \ = \ \Rad^i(kE)/\Rad^j(kE)$ for any $0 \leq i < j$,
\item  
$M \ = \ \Rad^i(\Lambda^*(\CV))/\Rad^{i+2}(\Lambda^*(\CV))$ for any $0 \leq i$,
\item  $M \ = \ S^{*\geq i}(\CV)/S^{*\geq j}(\CV)$ for any $0 \leq i $, $1\leq j-i \leq p$,
\item $M \ = \ \Rad^i(kG)/\Rad^{j}(kG)$ for any $0 \leq i$, $1\leq j-i \leq p$.
\end{enumerate}
\end{prop}

We next see how to generate examples of modules of constant type arising from 
the consideration of (negative) Tate cohomology.  Once again, their formulation does
not depend upon a choice of $\CV$ so that we conclude that these examples are 
modules of constant radical and socle type independent of the choice of $\CV \subset
\Rad(kE)$ splitting the projection $\Rad(kE) \to \Rad(kE)/\Rad^2(kE)$.

\begin{prop} \label{negcoho}
Choose a Tate cohomlogy class
$0 \not= \zeta \in \widehat \HHH^{-t}(E,k), \ t > 0$, and denote by
\begin{equation}
\label{tate}
 0 \to k \to M \to \Omega^{-t-1}(k) \to 0
 \end{equation}
the corresponding extension of $kE$-modules.  Then
$M$ has  constant $r$-radical type and constant $r$-socle type for every $0 <r<n$.
 \end{prop}

\begin{proof}
Let $0 \to M_1 \to M_2 \to M_3 \to 0$ is a short exact sequence of $kE$-modules 
with the property that for every $U \in \Grass(r,\CV)$ the restriction of this sequence
along $\alpha_U: C(U)\to kE$ splits.  Then 
if $M_1$ and $M_3$ have constant $r$-radical type (respectively, $r$-socle type),
then so does $M_2$.    Consequently, by Example \ref{ex-projmod}
it suffices to prove that the sequence (\ref{tate})
splits along $\alpha_U$ for every $U \in \Grass(r,\CV)$.  As shown in \cite{BC90},
the splitting of  (\ref{tate}) is implied by 
 \begin{equation}
 \label{tate-restr}
 \alpha_U^*(\zeta) = 0 \in \widehat \HHH^{-t}(C,k), \quad \forall U \in \Grass(r,\CV),
 \end{equation}
 where  $C \equiv  k[t_1,\ldots,t_r]/(t_i^p) \simeq C(U)$.

To show that $\alpha^*(\zeta) = 0$, we employ the  non-degenerate pairing of 
Tate duality (see \cite{BC90}),
\begin{equation}
\label{td}
\widehat \HHH^{-t}(C,k) \otimes \widehat \HHH^{t-1}(C,k) \to \widehat \HHH^{-1}(C,k) = k.
\end{equation}
Suppose that $\alpha^*(\zeta) \neq 0$. Then there exists $\eta^\prime \in 
\widehat \HHH^{t-1}(C,k)$ such that $\alpha^*(\zeta)\eta^\prime \neq 0$
Since $t-1 \geq 0$, $\alpha^*: \widehat \HHH^{t-1}(E,k) \ 
\longrightarrow \ \widehat \HHH^{t-1}(C,k)$ is surjective.
Hence, there exists $\eta \in \widehat \HHH^{t-1}(E,k)$  such that
$\eta^\prime = \alpha^*(\eta)$.  This implies, by the non-degeneracy of (\ref{td}), that
$$\alpha^*(\zeta \eta) \ = \ \alpha^*(\zeta)\eta^\prime\  \neq 0.$$
However, this is a contradiction, because we know that the map
$\alpha^*: \widehat \HHH^{-1}(E,k) \to \widehat \HHH^{-1}(C,k)$  is the zero map \cite{BC90}.
Thus we conclude that
$\alpha^*(\zeta) = 0$.
\end{proof}

It certainly is not always the case that constant $r$-$\Rad$-rank is preserved 
by Heller shifts. For a very easy example, let $M$ be a 2-dimensional indecomposable 
$kE$-module where the rank $n$ of $E$ is at least $2$. Then $M$ does not have constant
$1$-$\Rad^1$-rank, but $\Omega(M)$ does have constant $1$-$\Rad^1$-rank.

A more complicated example is the following. In this case, 
$M$ is a  $kE$-module with $\Rad^2(M) = 0$ such that 
the 2-$\Rad^1$-rank of $M$ is constant (hence, $M$ has constant $2$-radical type) 
but the Heller shifts of $M$ do not have constant $2$-radical type.
Note that this also gives an example of a module with constant $2$-radical type
that does not have constant $1$-radical type, that is, constant Jordan type. 

\begin{ex} \label{ex-char2}
Assume that $k$ is a field of characteristic 2. 
Suppose that $kE = k[w,x,y,z]/(w^2,x^2,y^2,z^2)$ is the group algebra of 
an elementary abelian group of order $2^4 = 16$. We consider the module 
$\Rad^2(kE)$ which is spanned as a subspace of $kE$ by the monomials
\[
wx, \ wy, \ wz, \ xy, \ xz,  \ yz, \ wxy, \ wxz,  \ wyz, \ xyz, \ wxyz. 
\]
Let $L$ be the submodule 
generated by $wx$, which has $k$-basis $wx, wxy, wxz, wxyz$. 
Let $M$ be the quotient $\Rad^2(kE)/L$.  The reader can easily check 
that $M$ has constant 2-radical type. In particular, for any $U \in \Grass(r, \CV)$,
$\Rad_U(M) = \Rad(M)$ which is spanned by $wyz$ and $xyz$. 
Because $\Rad^2(M) = \{0\}$, it also has 
$2$-$\Rad^2$-type.

In terms of diagrams, the restriction of $M$ to  $kF_1 = k[x,w]/[x^2,w^2]$ has the form
\[
\xymatrix@-.8pc{
& yz \ar[dl]_x  \ar[dr]^w &&  \oplus & wy & \oplus & wz & \oplus & xy & \oplus & xz \\
xyz && wyz
}
\]
Thus we see that $M_{\downarrow kF_1} \cong \Omega^{-1}(k) \oplus k^{\oplus 4}$.
On the other hand, the restriction to $kF_2 = k[y,z]/(y^2,z^2)$ has the form
\[
\xymatrix@-.8pc{
wz \ar[dr]_y && wy \ar[dl]^z   & \oplus & xz \ar[dr]_y && xy \ar[dl]^z & \oplus & yz \\
& wyz &&&& xyz
}
\]
Thus we have that $M_{\downarrow kF_2} \cong (\Omega^{1}(k))^{\oplus 2}
\oplus k$.

Now consider the modules $\Omega^{t}(M)$ with $t = 2j$ an 
even non-negative integer.
First note that the dimension of $M$ is 7, and
so the dimension of $\Omega^{2n}(M)$ must be $3+4d$ for some number $d$ which
depends on $n$.  In what follows we use the facts that 
if $kF = k (\Z/2 \times \Z/2)$ then for any  $t > 0$
$$
\Dim\Omega^t(k_F) = 2t +1, \quad \Dim\Rad \Omega^t(k_F) 
= t, \quad \Dim\Rad(kF) = 3
$$
(see \cite{He}). The formula 
$$ 
\Omega^{2j}(M)_{\downarrow kF_1} = \Omega^{2j}(\Omega^{-1}(k) \oplus k^{\oplus 4})_{\downarrow kF_1} =\Omega^{2j-1}(k_{F_1}) \oplus
(\Omega^{2j}(k_{F_1}))^{\oplus 4} \oplus (kF_1)^{\oplus m_1}
$$
for some $m_1$ yields the dimension formula
$$ 
3+4d = (4j-1) + 4(4j+1) +4m_1
$$
which implies that  $m_1 = d - 5j$. When a similar thing is done
for the restriction to $kF_2$, we get that $m_2 = d -3j -1$.
We conclude that
$$
\Dim(\Rad(kF_1)\Omega^{2j}(M)) = 3d-5j-1 \not= 3d-3j-1 =
\Dim(\Rad(kF_2) \Omega^{2j}(M)).
$$
Consequently, the $2$-$\Rad^1$-rank of $\Omega^{2j}(M)$ is constant if and
only if $j = 0$.  A similar analysis can be performed on $\Omega^t(M)$, for 
$t$ odd or negative with the same result. 

In addition, $\Dim(\Rad^2(kF_i) \Omega^{2j}(M)) = m_i$, the rank of the projective
part of the restriction of $\Omega^{2j}(M)$ to $kF_i$. Thus, $\Omega^{2j}(M)$ has constant $2$-$\Rad^2$-rank if and only if $j=0$.
\end{ex}

\section{Modules from quantum complete intersections} 
\label{se:quantum}
In this section, we consider $kE$-modules  constructed as 
subquotients of quantum complete intersection algebras.   
We demonstrate how by varying  parameters, we get families of 
modules with interesting properties, such as   modules of constant 
Jordan type  or constant $r$-radical or $r$-socle type for $r>1$.  
We supplement our constructions with multiple specific examples. 

Let $E$ be an elementary abelian $p$-group of rank $n$, and  let $kE=k[x_1, \ldots, x_n]/(x_1^p, \ldots, x_n^p)$. 
Let $\fq=(q_{ij})_{i,j=1}^n$  be the matrix of quantum parameters: choose non-zero $q_{ij}\in k$ for $1 \leq i<j\leq n$, and set
$q_{ij} = q_{ji}^{-1}$ and $q_{ii}=1$. 
\sloppy
{

}

Let $k\langle z_1, \ldots, z_n\rangle$  be the algebra generated by $n$ (non-commuting) variables $z_1, \ldots, z_n$, 
and let  $s>1$ be an integer.  Let
$$S = \frac{k\langle z_1, \ldots, z_n\rangle}{(z_i^s, z_iz_j - q_{ij}z_jz_i)}$$ 
be a quotient of the quantum complete intersection algebra $k\langle z_1, \ldots, z_n\rangle/(z_iz_j -q_{ij}z_jz_i)$  with respect to the ideal generated by $(z_1^s, \ldots, z_n^s)$. 
Let $I = \Rad(S)$.  When this causes no confusion, we denote the generators of the augmentation ideal $I$ by the same letters $z_i$, $1 \leq i\leq n$. 
For $0 \leq a \leq n(s-1)-1$, we define 
\begin{equation}
\label{wa}
W_a(s, \fq)  \ =  \ I^a/I^{a+2}, \quad \text{frequently denoted by ~} W_a.
\end{equation}
As a vector space, $W_a$ is generated by the monomials $\{z_1^{a_1}\ldots z_n^{a_n}\}$ where $a_1+\ldots + a_n = a$ or $a+1$ and $a_i \leq s-1$ for $1\leq i\leq n$. 
We define the structure of a $kE$-module on $W_a(s, \fq)$ by letting $x_i$ act via multiplication  by $z_i$ :
\[
x_i w \stackrel{def}{=} z_iw \,\quad  (\mod \, I^{a+2})
\]
for any $w \in W_a$.  By construction, $\Rad^2(kE) W_a=0$.  We also note that for $a \leq s-2$, $W_a$ is independent of $s$.

\begin{ex}
Let $n=2$ and choose $s$ and $a$ such that $a< s-1$. Let $q = q_{1,2}$ be the quantization parameter. 
In this case $kE = k[x,y]/(x^p,y^p)$ and $S = k\langle z, t \rangle/(z^s, t^s, zt-qtz)$. 
Then the module $W_a(s, q)$ looks as follows: 
$$
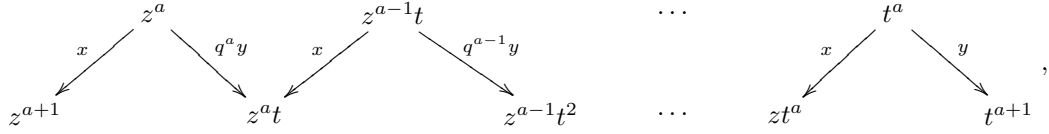
\begin{xy}*!C\xybox{%
\xymatrix{ &z^a
\ar[dr]^{q^{a}y} \ar[dl]_x&&
z^{a-1}t \ar[dl]_{x} \ar[dr]^{q^{a-1}y}
&&\dots &&t^a \ar[dl]_x \ar[dr]^y &\\
z^{a+1}  && z^{a}t  && z^{a-1}t^2  &\dots & zt^{a} &&t^{a+1}}}
\end{xy},
$$ where, for example, an arrow $\xymatrix{z^it^j \ar[r]^{q^iy}& z^it^{j+1}}$ indicates 
that the action of $y$ on $z^it^j$ is defined via $q^iy(z^it^j) = z^it^{j+1}$. 

It is easy to see that this $kE$--module is isomorphic to the ``zig-zag" module denoted 
$W_{a+1,2}$ in \cite{CFS}.   That is, in the case $n=2$ introducing the parameter $q$ does 
not lead to new isomorphism classes of modules.  For $n>2$, though, choice of the $q_{i,j}$ does make a
difference as we demonstrate in this section.
\end{ex}

We now show that if $a$ is sufficiently large, then  the module $W_a$ has a very strong 
property of having equal $r$-images independently of the choice of $q_{ij}$. In particular, 
it has constant $r$-radical type. 

\begin{prop} \label{W-eip}
Let $W_a = W_a(s,\fq)$ for some fixed
choice of $s \geq 1$ and elements $q_{ij} \in k$. 
If  $a \geq (n-r)(s-1)$, then the module $W_{a}$ has the equal 
$r$-images property,  meaning that for any  $U$ in $\Grass(r, \CV)$, 
we have that $\Rad_U(W_a) = \Rad(W_a) = \Rad(kE)W_a$. Hence, $W_a$
has constant $r$-$\Rad$-type. 
\end{prop}

\begin{proof}  Let $\CV\subset \Rad(kE)$ be the subspace generated by $\{x_1, \ldots, x_n\}$. 
Choose $U$ in $\Grass(r, \CV)$. For the purposes of the argument
we desire a basis for the subspace $U
\subseteq \CV$ that is chosen carefully as follows.  
Let $\ul u = [u_1, \dots, u_r]$ be an ordered
basis for $U$ and suppose that $u_i = \sum_{i=1}^n a_{i,j}x_j$ 
for $a_{i,j} \in k$. We may assume that the matrix $(a_{i,j})$
is in echelon form, so that there is some subset $\Sigma = \{i_1, \dots, i_r\}$
in  $\{1, \dots, n\}$ such that the $r \times r$ submatrix having
the columns indexed by $\Sigma$ is the identity matrix. We claim
that, without loss of generality, we may assume that that
$\Sigma = \{1, \dots, r\}$. That is, if $\Sigma$ is not of this form then 
we correct the situation by applying a suitable permutation 
to the basis $x_1, \dots, x_n$ of $\CV$. The same permutation 
must be applied to the generators $z_1, \dots, z_n$ of the algebra $S$.
Note that this changes the values of the $q_{ij}$, but 
because these are assumed to be non zero, 
the augmentation ideal $I \subset S$ is invariant under the permutation.
Hence, $W_a$ is unchanged.  

Let
\[
\alpha: \  kF \ = \ k[t_1, \dots, t_r]/(t_1^p, \dots, t_r^p) 
\quad \longrightarrow \quad kE
\]
be 
given by $\alpha(t_i) = u_i$ for ${\ul u} = [u_1, \dots, u_r]$ 
chosen as above. For $i \in \{1, \dots, r\}$, let $u_i = \sum 
a_{i,j}x_j$, and set $w_i = \sum a_{i,j}z_j \in S$,
so that $\alpha(t_i)$ acts on $W_a$ by multiplication by $w_i$ (mod $I^{a+2}$).
Because of the way that the basis was chosen, we have
that for each $i$, $1 \leq i \leq r$, $w_i = z_i + \sum_{j = r+1}^n a_{i,j}z_j$.

The module $W_a$ has a basis consisting of the monomials 
\[
Z_{s_1, \dots, s_n} \ = \ z_1^{s_1} \ldots z_n^{s_n} 
\]
where $s_1 + \dots +s_n =a$ or $a+1$ and $0 \leq s_i < s$ for all $i$,  
taken modulo $(z_1^s, \ldots, z_n^s)$ and $I^{a+2}$.
Since $\alpha(t_i)$ acts on $W_a$ via $w_i$ which is a linear conbination of $z_i$, we have 
\[
\alpha(t_i)I^a \subset I^{a+1}.
\]
Therefore, 
\[\Rad_U(W_a) \subset I^{a+1}/I^{a+2}=\Rad(kE)W_a
\]
Hence, we need to show that 
\[
 I^{a+1}/I^{a+2} \subset \Rad_U(W_{a}),
\]
or, equivalently, that every monomial $z_1^{s_1} \cdots z_n^{s_n}$ with $s_1 + \dots +s_n = a+1$ is in
$\Sigma_{i=1}^r w_iW_a$.  We accomplish this by an induction on the 
number $N = s_1 + \dots +s_r$. 

Because we assume that $a \geq (n-r)(s-1)$, the minimum value 
that $N$ can have is $a+1 - (n-r)(s-1)>0$ and that occurs when the 
monomial has the form $z_1^{s_1}\ldots 
z_r^{s_r}z_{r+1}^{s-1} \ldots z_n^{s-1}$
for $s_1 + \dots + s_r = a+1 - (n-r)(s-1)$. Let $i$ be the least integer such
that $s_i > 0$. Since $z_j^s=0$ in $S$ and $w_i = z_i + \sum\limits_{j=r+1}^na_{i,j}z_j$, we have 
\begin{equation} \label{induction}
z_i^{s_i}\cdots 
z_r^{s_r}z_{r+1}^{s-1} \ldots z_n^{s-1} \ - \ w_iz_i^{s_i-1}\ldots z_r^{s_r}z_{r+1}^{s-1} \ldots z_n^{s-1} \
= 
\end{equation}
\[ - \sum_{j = r+1}^n a_{i,j}z_j z_i^{s_i-1}\ldots z_r^{s_r}z_{r+1}^{s-1} \ldots z_n^{s-1} = 0
\]
Hence, the class of $z_i^{s_i}\ldots 
z_r^{s_r}z_{r+1}^{s-1} \ldots z_n^{s-1}$ is in $\Rad_U(W_a)$.

For the induction step, let 
$Z  = z_1^{s_1} \ldots z_n^{s_n}$ 
with $N = s_1 + \cdots +s_r>a+1 - (n-r)(s-1)$. If $i$ is the least integer with $s_i >0$,
then we get the exact same formula as in \ref{induction}. By induction,
the classes of the elements on the right hand side are all in $\Rad_U(W_a)$.
Hence, so too is the class of $Z$. 

We conclude that $\Rad_U(W_a) = I^1/I^{a+2}$ is independent of $U$.  On the other hand,
for  $r$-planes $U, V \subset \CV$ 
$$\Rad_U(I^2) \ = \ \Rad_U(\Rad_V(M_a)) \ = \ \Rad_V(\Rad_U(M_a)) \ = \ \Rad_V(I^2).$$
Continuing by induction on $j$, we conclude that $W_a$ has constant $r$-$\Rad^j$-rank for
all $j, 1\leq j < p$.
\end{proof}

The following lemma  (whose proof we leave to the reader) is proved by induction 
using the $q$-binomial formula: suppose $X,Y$ are $q$-commuting variables, 
that is $YX=qXY$. Then
\[(X+Y)^n =\sum_{i=0}^n {n \choose i}_q
  X^iY^{n-i},
\]
where ${n \choose i}_q=
   \frac{(n)_q!}{(i)_q! (n-i)_q!}$, $(i)_q = 1+q+\cdots+q^{i-1}$, 
and $(i)_q! =(i)_q(i-1)_q\ldots(1)_q$. 

\begin{lemma}
\label{le:binom}
Let $s>1$ be an integer prime to $p$, and let $\zeta \in k$ be a primitive $s^{th}$ root of unity. 
Let $z_1, \ldots, z_n$ be $\zeta$-commuting variables, that is  $z_iz_j=\zeta z_jz_i$ for 
$1 \leq i<j\leq n$.  Then for any $a_1, \ldots, a_n \in k$,
\[ (a_1z_1 + \cdots +a_nz_n)^s = a_1^sz_1^s+\cdots + a_n^sz_n^s. 
\]

\end{lemma}

This lemma enables us to show that the modules $W_{a,\fq}$ of (\ref{wa}) are of constant Jordan
type provided that our quantum parameters $\fq$ are given by a single $s^{th}$ root of unity.

\begin{prop} \label{quantum-cjt}
Let $s>1$ be an integer. Assume that one of the following holds
\begin{enumerate}\item[I.] $a<s-1$ or
  \item[II.] $(s,p)=1$ and $q_{i,j} = \zeta$ for $1\leq i < j \leq n$ where $\zeta$ be a primitive $s^{th}$ root of unity in k.
  \end{enumerate}   Then 
the module $W_a=W_a(s, \fq)$ has constant Jordan type. 
\end{prop}

\begin{proof} To prove that $W_a$ has constant Jordan type, we need to
show that for every non-trivial $n$-tuple $(a_1, \ldots, a_n) \in k^n$,
the Jordan type of the element $u = a_1x_1+\cdots + a_nx_n$ as an
operator on $W_a$ is the same.
Since $u^2$ acts trivially by construction of $W_a$, 
we just need to show that the rank of $u$ is constant. 

Let $\ell = \sum\limits_{i=1}^n a_iz_i \in S$.   
Choose some $i$ so that $a_i \neq 0$. Then $I=\Rad(S)$ is generated
by  the elements
$\ell, z_1, \dots, z_{i-1},z_{i+1}, \dots, z_n$.
By an argument as in the proof of Proposition \ref{W-eip}, 
we have that $W_a$ has a basis as a $k$-vector 
space consisting of 
the classes modulo $I^{a+2}$ of the monomials
\[
\ell^vz_1^{v_1} \ldots z_{i-1}^{v_{i-1}}
z_{i+1}^{v_{i+1}} \ldots z_n^{v_n}
\]
for $0 \leq v, v_i \leq s-1$, and $(v + \sum_{j \neq i} v_i) \in \{a, a+1\}$ 
under either one of  our two assumptions.
By the definition of the action, $u$ acts of $W_a$ via multiplication by $\ell$. 
We compute the kernel of the action of $u$ on $W_a$ in our two cases. 

\vspace{0.1in}
I. Assume $a<s-1$. In this case, the kernel of $u$ is precisely $\Rad(W_a)$ since multiplication 
by $\ell$ does not annihilate any linear combination of the monomials  $\ell^vz_1^{v_1} \cdots z_{i-1}^{v_{i-1}}
z_{i+1}^{v_{i+1}} \cdots z_n^{v_n}$ with $v + \sum_{j \neq i} v_i =a$. Hence, $W_a$ has constant Jordan type.  

\vspace{0.1in}
II. Now suppose $q_{ij} = \zeta$ for $1\leq i<j\leq n$.  Since we also assume $(s,p)=1$, 
Lemma~\ref{le:binom} implies that $\ell^s=0$ in this case.

The kernel of multiplication by $\ell$ on $W_a$ is 
precisely the space spanned by those monomials $\ell^vz_1^{v_1} \cdots z_{i-1}^{v_{i-1}}
z_{i+1}^{v_{i+1}} \cdots z_n^{v_n}$
for which either $v + \sum_{j \neq i} v_i = a+1$ or 
$v = s-1$. 
Since the number of such monomials is again independent of the choice of $\ell$
we conclude that $W_a$ has constant Jordan type. 
\end{proof}

The next example illustrates that the condition of Proposition~\ref{quantum-cjt} requiring that $\zeta$ is the $s$-th root of unity is crucial. 

\begin{ex} 
\label{ex:nonconst} 
Let $n=3$, $s=2$, and $a=1$. Let $kE = k[x,y,z]/(x^p, y^p, z^p)$. Pick $q \in k^*$ and let $q_{ij} =q$ for any $i<j$. 
Let $\tilde x, \tilde y, \tilde z$  be the algebraic generators of $S$, that is, $S = k\langle \tilde x, \tilde y, \tilde z\rangle/(\tilde x^2, \tilde y^2, \tilde z^2, \tilde x\tilde y - q\tilde y\tilde x, \tilde x\tilde z - q\tilde z\tilde x, \tilde y\tilde z - q\tilde z\tilde y)$.  Then $W_1(2,q)$ can be depicted as follows:
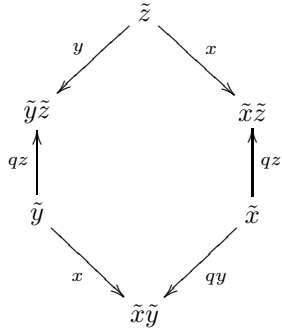
\[
\xymatrix{
&& \tilde z \ar[dl]_y \ar[dr]^x &&\\
& \tilde y\tilde z && \tilde x \tilde z \\
& \tilde y \ar[u]^{qz} \ar[dr]_x && \tilde x \ar[u]_{qz} \ar[dl]^{qy} \\
&& \tilde x\tilde y \\
}
\]

For $q=-1$, this module is a special case of modules in Proposition \ref{sym-algs}(2). In particular, it has constant Jordan type.
We show that for $q\not = -1$, $W_1(2,q)$ fails to have constant Jordan type. 
To achieve this, we compute the non-maximal support variety of $W_1(2,q)$ for a generic $q$.

Fix the following order of the linear generators of $W_1(2,q)$: $\tilde x, \tilde y, \tilde z, \tilde x \tilde y, \tilde x \tilde z, \tilde y \tilde z$.  Let $[a:b:c] \in \bP^2$ and let $\ell = ax+by+sz \in \CV$ be a generator of the corresponding line in $\Rad kE$.   
 The matrix of $\ell$ as an endomorphism of $W_1(2,q)$ with respect to our fixed basis  has the 
form 
\[
\ell \ \leftrightarrow \ \begin{pmatrix} 0 & 0 \\ A_{\ell} & 0 \end{pmatrix}
\]
where for $x, y$ and $z$ we have 
\[
A_x = \begin{pmatrix} 0& 1& 0 \\ 0 & 0 & 1 \\ 0 & 0 & 0 \end{pmatrix}, \quad
A_y = \begin{pmatrix} q& 0& 0 \\ 0 & 0 & 0 \\ 0 & 0 & 1 \end{pmatrix}, \quad
A_z = \begin{pmatrix} 0& 0& 0 \\ q & 0 & 0 \\ 0 & q & 0 \end{pmatrix}.
\]
For the general element $\ell = ax+by+cz$ we get 
\[ 
A_\ell = \begin{pmatrix} qb& a& 0 \\ qc & 0 & a \\ 0 & qc & b \end{pmatrix}.
\]
The determinant of $A_\ell$ is $q(q+1)abc$. Hence, for $q \neq -1$ 
the  nonmaximal support variety is a union of three lines: $a=0, b=0, c=0$.
In particular, $W_1(2, q)$ has constant  Jordan type if and only if $q=-1$. 

We finish this example recording the properties of radicals and socles of $W_1(2,q)$. 
First, since the condition $a \geq (n-r)(s-1)$ is satisfied for $r=2$ (we get $1 \geq (3-2)(2-1)$), 
Proposition~\ref{W-eip} implies that $W_1(2,q)$ has constant $2$-radical type.  Since the module $W_1(2,q)$ 
is self-dual,   it also has constant $2$-socle type.   So, in particular, we conclude that for $q \not = -1$, 
$W_1(2,q)$   does not have constant Jordan type but has constant $2$-radical and $2$-socle type.  
\end{ex}

In the following example, we construct a module of 
the form $W_a(s, q)$   that has constant Jordan type and constant $2$-socle type 
but fails to have constant $2$-radical type.  It follows that the dual of such 
$W_a(s,q)$ has constant Jordan type, constant 2-radical type, but 
not constant 2-socle type. 

\begin{ex}
\label{benzol}  Let $n=3$, $s\geq 3$, $a=1$. Let $q\not = 0$ be a quantum parameter, and set $q_{ij} = q$ for $1\leq i<j\leq 3$. 
Let $M_q = W_1(s, \fq)$. Let $kE = k[x_1, x_2,x_3]/(x_1^p, x_2^p,x_3^p)$ and $S = k\langle z_1, z_2, z_3\rangle/(z_i^s=0, z_iz_j=q_{ij}z_jz_i)$. 
Here is a depiction of $M_q$:
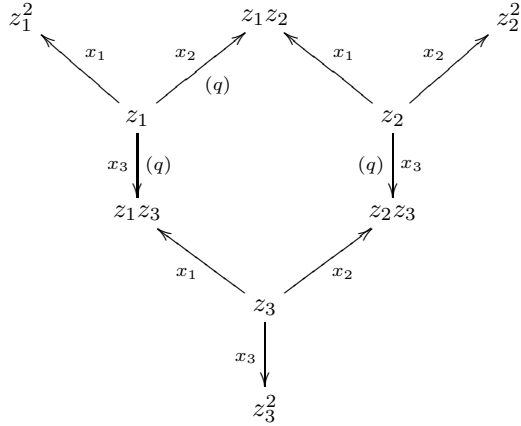
$$
\xymatrix{
z_1^2  && z_1z_2&& z_2^2  \\
& z_1\ar[ul]_{x_1}\ar[ur]^{x_2}_{(q)}\ar[d]_{x_3}^{(q)} && z_2 \ar[ul]_{x_1}\ar[ur]^{x_2}  \ar[d]^{x_3}_{(q)}\\
& z_1z_3 && z_2z_3 \\
&& z_3 \ar[d]_{x_3} \ar[ul]^{x_1} \ar[ur]_{x_2}\\
&& z_3^2 
} 
$$ 
Here, an arrow marked with $(q)$ means that the action is twisted by $q$. For example, $x_3\circ z_1 = q^{-1}z_1z_3$. 

\vspace{0.1in}
We make several observations about $M_q$.  
\vspace{0.1in}

I. By Proposition~\ref{quantum-cjt}, $M_q$ has constant Jordan type. 
\vspace{0.1in}

II. The module $M_q$ has constant $2$-Socle type.  Indeed, let $a=3s-5$ and consider 
the module $W_a(s, \fq)$ for arbitrary non-zero parameters $\fq=(q_{ij})$: 
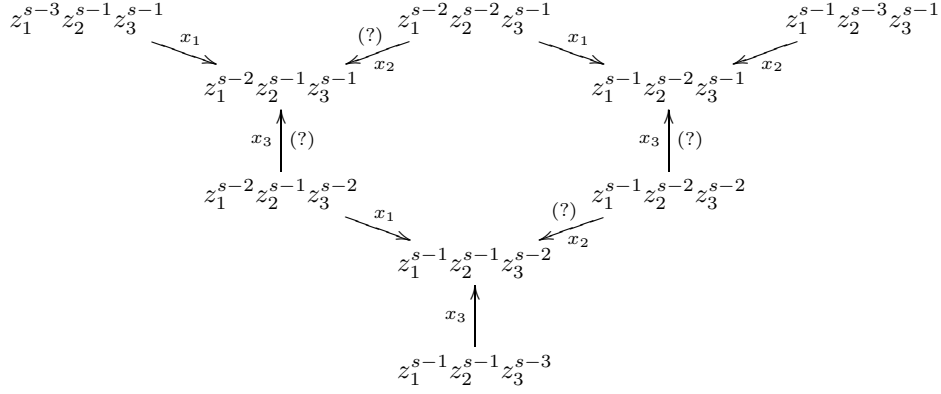
\[\xymatrix@=3mm{
z_1^{s-3}z_2^{s-1}z_3^{s-1} \ar[dr]^{x_1} && z_1^{s-2}z_2^{s-2}z_3^{s-1} \ar[dl]^{x_2}_{(?)} \ar[dr]^{x_1} && z_1^{s-1}z_2^{s-3}z_3^{s-1} \ar[dl]^{x_2} \\
& z_1^{s-2}z_2^{s-1}z_3^{s-1} && z_1^{s-1}z_2^{s-2}z_3^{s-1} \\
&&&\\
& z_1^{s-2}z_2^{s-1}z_3^{s-2} \ar[uu]^{x_3}_{(?)} \ar[dr]^{x_1} && z_1^{s-1}z_2^{s-2}z_3^{s-2} \ar[uu]^{x_3}_{(?)} \ar[dl]^{x_2}_{(?)} \\
&& z_1^{s-1}z_2^{s-1}z_3^{s-2} \\
&&&\\
&& z_1^{s-1}z_2^{s-1}z_3^{s-3} \ar[uu]^{x_3}
}
\]

The action along the arrows marked with $(?)$ is twisted by some monomials on $q_{ij}$. By choosing the parameters $q_{12}, q_{23}$ and $q_{13}$ appropriately,  we can arrange the twists so that 
$$W_{3s-5}(s, \fq) \simeq M^\#_q.$$ 
Since $3s-5>(3-2)(s-1)$ for $s \geq 3$, Proposition~\ref{W-eip} implies that $W_{3s-5}(s, \fq)$ has constant $2$-radical type.  By duality, $M_q$ has constant $2$-socle type. 
\vspace{0.1in}

III. Proposition~\ref{W-eip} does not apply to $2$-images of $M_q$  since the parameters $n=3$, $r=2$, $s\geq 3$, and $a=1$ fail to satisfy the condition $a \geq (n-r)(s-1)$.
In fact, we proceed to show that $M_q$ does not have constant $2$-radical type {\it unless} $q=1$.  

Let $U \in \Grass(2, \CV)$ be a $2$--plane in the three-dimensional space $\CV$. Let  

\vspace{0.1in}
\centerline{\begin{tabular}{ll}$u_1 = a_1x_1+a_2x_2+a_3x_3$\\
$u_2 =  b_1x_1+b_2x_2+b_3x_3$ 
\end{tabular}} 
\vspace{0.1in}

\noindent
be a basis of $U$, and  let 

\vspace{0.1in}
\centerline{\begin{tabular}{ll}$\ell_1 = a_1z_1+a_2z_2+a_3z_3$\\
$\ell_2 =  b_1z_1+b_2z_2+b_3z_3$ 
\end{tabular}} 

\vspace{0.1in} \noindent
be the corresponding elements in $S = k\langle z_1, z_2, z_3\rangle/(z_i^s, z_iz_j-q_{ij}z_jz_i)$. 

We fix the following order of the basis of $M_q$: $z_1, z_2, z_3$ for $M_q/\Rad(M_q)$ and 
$z_1z_2,z_1z_3, z_2z_3,z_1^2,z_2^2,z_3^2$
for $\Rad(M_q)$. 
Since $\Rad_U(M_q)\subset \Rad(M_q)$, we work inside $\Rad(M_q)$. 
We have
\vspace{0.1in}

\centerline{\begin{tabular}{lll}
$\ell_1 z_1 = a_1z_1^2+qb_1z_1z_2+qc_1z_1z_3$\\
$\ell_1 z_2 = a_1z_1z_2+b_1z_2^2+qc_1z_2z_3$\\
$\ell_1 z_3 = a_1z_1z_3+b_1z_2z_3+c_1z^2_3$\\
\end{tabular}} 

\vspace{0.1in} \noindent
and similarly for $\ell_2$. Hence, with respect to our fixed basis, 
$\Rad_U(M_q)$ is generated by the following six vectors:
$$R=\begin{pmatrix}
qb_1&a_1&0&qb_2&a_2&0\\
qc_1&0&a_1&qc_2&0&a_2\\
0&qc_1&b_1&0&qc_2&b_2\\
a_1&0&0&a_2&0&0\\
0&b_1&0&0&b_2&0\\
0&0&c_1&0&0&c_2\end{pmatrix}
$$
To compute the nonmaximal $2$-radical support variety of $M_q$, one would need to calculate the rank of this matrix for different parameters $a_i, b_i,c_i$. We leave such calculations to the Appendix and just show here that the rank of this matrix is not constant. 

First, take $u_1=x_1$ and $u_2=x_2$.  In this case we see from the picture that $\alpha_{U}^*(M_q)$ for $U = \langle x_1, x_2 \rangle$ splits as a direct sum of three ``zig-zag" modules: 
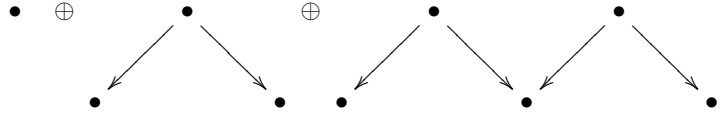
\[
\bu \quad \oplus \xymatrix{ &\bu \ar[dr]\ar[dl] & \\ \bu &&\bu} \oplus \xymatrix{ &\bu \ar[dr]\ar[dl] && \bu \ar[dr]\ar[dl] & \\ \bu &&\bu && \bu}.
\]
Hence, $\Dim \Rad_U(M_q)=5$. 

Second, take $u_1=x_1+x_2, u_2=x_2+x_3$. Hence, $a_1=b_1=1$, $b_2=c_2=1$ and $c_1=a_1=0$. In this case,
$$R=\begin{pmatrix}
q&1&0&q&0&0\\
0&0&1&q&0&0\\
0&0&1&0&q&1\\
1&0&0&0&0&0\\
0&1&0&0&1&0\\
0&0&0&0&0&1\end{pmatrix}.
$$
We have 
$\det R = q(1-q)$.
Hence, if $q\not = 1$, the rank of $R$ is $6$, and, therefore, for the $2$-plane $U$ spanned by $u_1, u_2$, we have 
$\Dim \Rad_U(M_q)=6$. We conclude that $M_q$ does  not have constant $2$-radical rank. 
\end{ex}

We give another example of the same phenomenon.  This time, we construct 
a module which has constant Jordan type, constant $2$-radical type but does 
not have constant $2$-socle type. 
\begin{ex} \label{ex-2-rad}
Assume that $p > 3$. 
Let $n = 4$,  $s = 3$ and $q_{i,j} = \zeta_3$
for all $1 \leq i < j \leq r$, where $\zeta = \zeta_3 \in k$ is a primitive
third root of unity.  Consider the module 
$$M= W_{6}(3, \zeta_3) = I^6/I^8.$$
By Proposition~\ref{quantum-cjt}, $M$ has constant
Jordan type. Since $6 > (4-2)(3-1)$, Proposition~\ref{W-eip} implies that $M$ has  
constant 2-radical type.  
We wish to show that 
$M$ fails to have constant 2-socle type. 

The module $M$ has dimension 14, and has a basis consisting of
the classes of the monomials of the form $z_1^az_2^bz_3^cz_4^d$ with
$0 \leq a,b,c,d \leq 2$ and where $a+b+c+d$ is either 6 or 7.
The radical of $M$, which is spanned by the monomials with 
$a+b+c+d = 7$, has dimension 4. Because the module has the 
equal $1$--images property by \ref{W-eip}, the image of multiplication by 
any nonzero $u = a_1x_1 + \dots + a_4x_4$ is the entire radical. 
Consequently the Jordan type of any such $u$ consists of 
4 blocks of size 2, and 6 blocks of size 1. Also, the dimension
of the kernel of multiplication by  $u$ is 10. 

Assume first that $U \subseteq \CV$ is the subspace spanned 
by $x_1$ and $x_2$. Then $\Soc_U(M)$ is the set of all elements
annihilated by multiplication by both $x_1$ and $x_2$. Clearly,
the monomials $z_1^2z_2^2 z_3^2$, $z_1^2z_2^2 z_3z_4$, 
and $z_1^2z_2^2 z_4^2$ are in $\Soc_U(M)$. Moreover,  
$\Rad(M) \in \Soc_U(M)$. From this we see that $\Soc_U(M)$
has dimension at least 7, and further investigation shows that 
the dimension is  exactly 7.

Next suppose that $U$ is the subspace spanned by the 
elements $u_1 = x_1+x_2$ and $u_2 = x_1+x_3$. We claim
that the dimension of $\Soc_U(M)$ is 6. Let $K_i$ denote
the kernel of multiplication by $u_i$ on $M$. Then $\Rad(M)$
is in both $K_1$ and $K_2$. In addition, the elements 
\[
z_1^2z_2^2 z_3^2, \ \ \
z_1^2z_2^2 z_3z_4, \ \ \
z_1^2z_2^2 z_4^2, \ \ \
z_1^2z_3^2 z_4^2,  \ \ \
z_1^1z_2^2 z_3^2z_4 -\zeta z_1^2z_2^1 z_3^2z_4 , \ \ \
z_1^2z_2^1 z_3^2z_4,
\]
\[
z_1^1z_2^2 z_3^1z_4^2 - \zeta z_1^2z_2^1 z_3^1z_4^2, \qquad
z_1^2z_3^2 z_4^2+ z_2^2z_3^2 z_4^2 - \zeta^2z_1^2z_2^1 z_3^1z_4^2, 
\]
\[
z_1^1z_2^1 z_3^2z_4^2 -  z_1^2z_2^1 z_3^1z_4^2, \qquad
z_1^1z_2^2z_3^1 z_4^2 - z_2^2z_3^2 z_4^2 - \zeta^2z_1^2z_1^2 z_3^2z_4^2 
\]
are in $K_1 +K_2$. That is, the reader may check that each of the above elements
is annihilated by either $u_1$ or by $u_2$. Moreover, it is straightforward to check that
these elements are linearly independent and independent of $\Rad(M)$. 
Therefore, $K_1+K_2$ has dimension 14, whereas each $K_i$ has dimension 10.
Hence $\Dim \Soc_U(M) = \Dim(K_1 \cap K_2) = 6$. 
\end{ex}

In the appendix, we calculate some nonminimal 
$r$-socle support varieties for modules of the form $W_a$.  
Whereas calculations in the two examples above were simple enough 
to do by hand, the calculations left in the appendix use 
computational software.


\section{Radicals of $L_\zeta$-modules}
\label{se:Lzeta}

As in previous sections, $E$ is an elementary abelian 
$p$--group of rank $n$ and  $\CV \subset \Rad (kE)$ is chosen as in (\ref{identify}).  
For a homogeneous cohomology class $\zeta \in \HHH^m(E,k)$,
we recall that the module $L_\zeta$ is defined to be
 \begin{equation}
 \label{lz}
L_\zeta \ \equiv \ \Ker\{ \zeta: \Omega^m(k) \to k \} 
\end{equation}
Here, we have abused notation by using
$\zeta: \Omega^m(k) \to k$ also to denote the map representing  $\zeta \in \HHH^m(G,k)$.
As we see in this section, the  $L_\zeta$-modules give good examples of
behavior of radical and socle ranks.  

If $\alpha: C = k[t_1,\ldots,t_r]/(t_i^p) \to kE$ is a flat map, we write $\Omega^m(k_C)$ 
for the $m^{th}$ Heller translate of the trivial $C$-module, thereby distinguishing this 
Heller translate from the restriction $\alpha^*(\Omega^m(k))$ of the $kE$-module
(which is stably equivalent to $\Omega^m(k_C)$).

\vspace{0.05in} 
We employ the following notation: 
\begin{equation}
\HHH^\bu(E,k) \ = \begin{cases} \HHH^*(E,k)  & \text{  if  }  p = 2, \\
\HHH^{\rm even}(E,k) & \text{  otherwise}.
\end{cases} 
\end{equation}
Thus, $\HHH^\bu(E,k)$ is a commutative algebra, and
$\Proj \HHH^\bu(E,k) \simeq \bP^{n-1} \simeq \Grass(1, \CV)$.

\vspace{0.05in}
For our analysis of the behavior of radicals of $L_\zeta$,
we need to exploit a somewhat finer structure 
of the cohomology ring of $kE=k[x_1, \ldots, x_n]/(x_i^p)$ and of the restriction map on cohomology.

Let $\xymatrix{f: k \ar[r]^-{ a \mapsto a^p}& k}$ be  the Frobenius map. For a k-vector space $V$ 
we use the standard notation $V^{(1)}$ for the {\it Frobenius twist} of $V$, a vector space obtained 
via base change $f: k \to k$
\[ V^{(1)} = V \otimes_f k \] 
If $R$ is a (finitely generated commutative) $k$-algebra, then we have a map of $k$-algebras 
\[R^{(1)} \to R\] 
which sends $ x \otimes a$ to $ a^px$. 
Hence, there is an induced map of $k$-varieties
\[F: \Spec A \to (\Spec A)^{(1)} \stackrel{def}{=} \Spec A^{(1)}.\]
The same construction applies globally. If $X$  is any $k$-variety, we obtain a Frobenius twist $X^{(1)}$ 
and a map of $k$-varieties
\[F: X \to X^{(1)}\] 
Moreover,  if $X$ is defined over $\mathbb F_p$, 
then we have a natural identification $X^{(1)} \simeq X$ and Frobenius 
becomes a self-map
\[F_X: X \to X. \]   
We direct the reader to \cite[I.2]{Jan} and \cite[\S1]{FS} for a detailed discussion of the properties of the 
Frobenius twist.  

We apply the above discussion to the algebra $S^*(V^\#)$,
so that the $k$-points of $\Spec S^*(V^\#)$ constitute the vector space $V$. 
Using the natural $k$-algebra isomorphisms 
$S^*((V^{(1)})^\#) \simeq (S^*(V^\#))^{(1)} = S^*(V^\#) \otimes_f k$ 
(see \cite[\S1]{FS}), we get a map of varieties  over $k$
\[ F: V \to V^{(1)}. \]
Suppose  that $V$ is given an $\mathbb F_p$-structure; in other words, $V$ is identified 
with $V_0 \otimes_{\mathbb F_p} k$ where $V_0$ is an 
$\mathbb F_p$-vector space. 
Then  we have a natural identification $V^{(1)} \simeq V$, 
and the Frobenius map becomes a self-map
\[F=F_{V}: V \to V
\] 
If we pick a basis $\{e_1, \ldots, e_n\}$ of $V_0$, then the Frobenius map is given explicitly via the formula 
\[\xymatrix{F_V: V \ar[r] & V\\
a_1e_1+ \cdots +a_ne_n \ar@{|->}[r]& a^p_1e_1+ \cdots + a^p_ne_n.}
\]  
Since $k$ is assumed to be algebraically closed (hence, perfect), the Frobenius map is a bijection on $V$.

The following description of the cohomology of $A=k[x_1,\ldots,x_n]/(x_i^p)$ 
can be found in \cite[I.4.27]{Jan}. Note that $V^{(1)}$ has a natural structure of a 
$\GL_n$--module given by pulling back the standard representation of $\GL_n=\GL(V)$ on $V$ via the 
Frobenius map $F: \GL_n \to \GL_n$.  

\begin{prop}
\label{idd}  Let  $V$ be an $n$-dimensional $k$-vector space
with a basis $\{x_1, \ldots, x_n\}$, and let  $A \ = \ S^*(V)/(v^p, v\in V)$.  
There is an isomorphism of graded $\GL_n$-algebras  
 
 \vspace{0.1in} 
\centerline{\begin{tabular}{cc}
$\HHH^*(A,k) \simeq S^*(V^\#)$ & for $p =2$,\\[5pt]
$\HHH^*(A,k) \simeq S^*((V^{(1)})^\#[2]) \otimes \Lambda^*(V^\#)$ & for $p > 2$,
\end{tabular}}
\vspace{0.1in}

\noindent
where $(V^{(1)})^\#[2]$ is the vector space $(V^{(1)})^\#$ placed in degree $2$.

Identifying $kE$ with $S^*(\CV)/(v^p, v\in \CV)$, we conclude that 
\begin{equation}
\HHH^*(E,k) \ = \begin{cases} k[\zeta_1, \dots, \zeta_n]  & \text{  if  }  p = 2, \\
k[\zeta_1, \dots, \zeta_n] \otimes \Lambda(\eta_1, \dots, \eta_n) & \text{  otherwise},
\end{cases} 
\end{equation}
where $\deg(\zeta_i)=1$ if $p=2$ and $\deg(\zeta_i)=2$ for $p>2$. Hence, $k[\zeta_1, \ldots, \zeta_n]$ 
is the homogeneous coordinate ring of $\Proj \HHH^\bu(E,k) = \Proj S^*(\CV^\#) \simeq \bP^{n-1}$
for $p=2$ and $\Proj (\HHH^\bu(E,k)_{red}) = \Proj S^*((\CV^{(1)})^\#) \simeq \bP^{n-1}$ 
(with $GL_n$ action twisted by Frobenius) for $p > 2$.

\end{prop}

The functoriality of the identifications of Proposition \ref{idd} immediately implies the following
corollary.

\begin{cor} 
\label{cor:twist}  
Let $U \subset V$ be an $r$-dimensional subspace with ordered basis $u_1, \ldots, u_r$, let 
$C = k[t_1, \ldots, t_r]/(t_1^p, \ldots, t_r^p)$, and let $\alpha: C \to A$ be a $k$-algebra 
map such that  $\{ \alpha(t_1)=u_1, \ldots, \alpha(t_r) = u_r \}$ is a basis for $U$. Then
there is a commutative diagram of $k$-algebras
\[\xymatrix{\HHH^\bu(A,k)_{\red} \ar[r]^\sim \ar[d]^{\alpha^*} & S^*((V^{(1)})^\#)\ar[d]\\
\HHH^\bu(C,k)_{\red}\ar[r]^\sim & S^*((U^{(1)})^\#)}
\]
for $p>2$  with the right vertical map induced by the Frobenius twist of 
the embedding $U \subset V$, and 
\[\xymatrix{\HHH^*(A,k)\ar[r]^\sim \ar[d]^{\alpha^*} & S^*(V^\#)\ar[d]\\
\HHH^*(C,k)\ar[r]^\sim & S^*(U^\#).}
\]
for $p=2$. 

Let \[\alpha_*: \Spec (\HHH^\bu(C,k)_{red}) \to \Spec (\HHH^\bu(A,k)_{red})\] 
be the map of $k$-varieties induced by $\alpha$.  
Then we have a commutative diagram of $k$-varieties 
\[\xymatrix{ (\Spec \HHH^\bu(C,k)_{red}) \ar[r]^{\alpha_*}\ar[d]^{\simeq}& 
(\Spec \HHH^\bu(A,k)_{red} )\ar[d]^{\simeq}\\
U^{(1)} \ar@{^(->}[r]& V^{(1)}
}
\]
for $p>2$ and 
\[\xymatrix{ \Spec \HHH^\bu(C,k) \ar[r]^{\alpha_*}\ar[d]^{\simeq}& \Spec \HHH^\bu(A,k) \ar[d]^{\simeq}\\
U \ar@{^(->}[r]& V
}
\]
for $p=2$.
\end{cor}

The following proposition is our key tool in determining whether the modules $L_{\zeta}$ 
have constant $r$-$\Rad^j$-rank.  In contrast to most of the results of this paper, this proposition
is proved for a general finite group scheme. 

\begin{prop} 
\label{dim}
Let $G$ be a  finite group scheme, and let $\zeta$  be a non-zero cohomology class of degree $m$.
Then 
$$
\Dim \Rad(\Omega^m(k)) - \Dim \Rad(L_\zeta) \ = \
\Dim \Ker \{\cdot \zeta: \HHH^1(G,k) \to \HHH^{m+1}(G,k)\}.
$$
In particular, if $\zeta: \HHH^1(G,k) \to \HHH^{m+1}(G,k)$ is injective, then 
\[\Rad^j(L_\zeta)=\Rad^j(\Omega^m(k))
\] for any $j>0$.
\end{prop}

\begin{proof} 
To prove the proposition, we construct a linear isomorphism
$$
\Psi: (\Rad (\Omega^m(k))/\Rad(L_\zeta))^\# \to \Ker \{\cdot \zeta: \HHH^1(G,k) \to \HHH^{m+1}(G,k)\}.
$$ 
Let
$$
\xymatrix{0 \ar[r] & L_\zeta \ar[r] & \Omega^{m}(k)\ar[r] &k\ar[r] &0}
$$
be the defining sequence for $L_\zeta$, and let
$$
\gamma: \xymatrix{0 \ar[r] & L_\zeta/\Rad(L_\zeta) \ar[r] & \Omega^{m}(k)/\Rad(L_\zeta)\ar[r] &k\ar[r] &0}
$$
be the induced sequence.
For a non-trivial map $f: L_\zeta/\Rad(L_\zeta) \to k$, we let
$$
\gamma_f: \xymatrix{0 \ar[r] & k \ar[r] & M \ar[r] &k\ar[r] &0}
$$
be  the pushout of the sequence $\gamma$
along the map $f$.  In other words, we have a commutative diagram with exact rows:
\begin{equation}
\label{semi}
\xymatrix{ &\Rad(\Omega^m(k))/\Rad(L_\zeta) \ar@{^(->}[d]_i\ar@{=}[r]&
\Rad(\Omega^m(k))/\Rad(L_\zeta) \ar@{^(->}[d]&\\
0\ar[r] &  L_\zeta/\Rad(L_\zeta) \ar[r]\ar@{->>}[d]_f &
\Omega^{m}(k)/\Rad(L_\zeta) \ar[r]\ar[r]^-\zeta\ar@{->>}[d]_{f^\prime} &k\ar@{=}[d]\ar[r]& 0\\
0\ar[r] & k \ar[r] & M \ar[r] & k \ar[r] & 0}
\end{equation}
The cohomology class $ \gamma_f\zeta \in \HHH^{m+1}(G,k)$ is represented by the composition
\begin{equation}
\label{tri}
\xymatrix{& & & \Omega^{m}(k) \ar[dr]^\zeta \ar@{->>}[d] &\\
& & k \ar[r] & M \ar[r] &k\ar[r]^-{\gamma_f} & \Omega^{-1}(k)}
\end{equation}
Because the composition $\gamma_f \zeta: \Omega^{m}(k) \to 
\Omega^{-1}(k)$ factors through $M$
and the bottom row of (\ref{tri}) is a distinguished triangle in the stable 
category ${\bf stmod}(G)$,
$\gamma_f \zeta$ must be zero.

Since $L_\zeta/\Rad(L_\zeta)$ is semi-simple, we have a splitting 
\[\rho: L_\zeta / \Rad(L_\zeta) \to \Rad(\Omega^m(k))/\Rad(L_\zeta)
\] 
of the map $i$. For any linear map $\phi: \Rad(\Omega^m(k))/\Rad(L_\zeta) \to k$ 
we thus have a map $\phi \circ \rho: L_\zeta/\Rad(L_\zeta) \to k$, and therefore an extension $\gamma_{\phi\circ\rho}$ such that  $\gamma_{\phi\circ\rho} \zeta =0$. 
We define
\[
\Psi : (\Rad(\Omega^m(k))/\Rad(L_\zeta))^\# \to
\Ker \{\cdot \zeta: \HHH^1(G,k) \to \HHH^{m+1}(G,k)\}, \quad \phi \, \mapsto \,\gamma_{\phi\circ\rho}.
\]

To show that  $\Psi$ is injective, let $\phi \in (\Rad(\Omega^m(k))/\Rad(L_\zeta))^\#$, 
and set $f = \phi \circ \rho$.
Observe that the extension $\gamma_f$ (the bottom row of (\ref{semi}))
is not split if and only  if the map $f^\prime: \Omega^m(k)/\Rad(L_\zeta) \to M$
does not factor through $\Omega^m(k)/\Rad(\Omega^m(k))$ which happens if and
only if $f \circ i \not = 0$.   Since $(\phi \circ \rho) \circ i = \phi$, we conclude
that $\Psi$ is injective.

To verify that $\Psi$ is surjective, consider some $\eta \in \HHH^1(G,k)$
such that $\eta \zeta = 0$.
Then $\zeta: \Omega^m(k) \to \Omega^m(k)/ \Rad(\Omega^m(k)) \to k$ must factor through
the extension $ k \to M \to k$ corresponding  to $\eta$. Let $f^\prime: \Omega^m(k) \to M$ be
the factorization map, and denote by $f: L_\zeta/\Rad(L_\zeta) \to k$ the restriction to
$L_\zeta/\Rad(L_\zeta)$. Then by construction $\eta = \gamma_f$.

Finally, if $\Dim \Ker \{\cdot \zeta: \HHH^1(G,k) \to \HHH^{m+1}(G,k)\} \ = \ 0$, we conclude
that $\Rad^j(L_\zeta) = \Rad^j(\Omega^m(k))$ for all $j$.  
\end{proof}

To apply Proposition \ref{dim}, we require the follow facts about restrictions of $L_\zeta$ modules.

\begin{lemma} \label{restr-Lzeta}
Suppose that $U \in \Grass(r, \CV)$ is an $r$-plane in $\CV$. Let $\alpha:
C = k[t_1, \dots, t_r]/(t_1^p, \dots, t_r^p) \to kE$ be a flat map such  that 
$\alpha(t_1), \dots, \alpha(t_r)$ is a basis for $U$. Suppose that 
$\zeta \in \HHH^m(E,k)$ is a non-zero homogeneous cohomology element 
of degree $m > 0$.  There exists a number $\gamma_m$ independent of $\alpha$  
such that
\[
\begin{array}{ll} \alpha^*(L_{\zeta}) \ \simeq \  C^{\oplus \gamma_m-1} \ \oplus \ 
\Omega(k_C) \ \oplus \ \Omega^m(k_C) & 
\text{ if }\alpha^*(\zeta) = 0\\
\alpha^*(L_{\zeta}) \ \simeq \  C^{\oplus \gamma_m} \ \oplus \  \ L_{\alpha^*(\zeta)} &
\text{ if } \alpha^*(\zeta) \neq 0.
\end{array}\]

Consequently,
\[\begin{array}{ll} 
\Dim \Rad(\alpha^*(L_\zeta)) = 
\gamma_m(p^r -1) -r + \Dim \Rad(\Omega^m(k_C)) & 
\text{ if }\alpha^*(\zeta) = 0\\[1pt]
\Dim \Rad(\alpha^*(L_{\zeta})) = \gamma_m(p^r-1) + 
\Dim \Rad(L_{\alpha^*(\zeta)}) &
\text{ if } \alpha^*(\zeta) \neq 0.
\end{array}\]
\end{lemma}

\begin{proof}
We have an exact sequence  
\[ 
\xymatrix{
0 \ar[r] & L_\zeta \ar[r] & \Omega^m(k) \ar[r]^{\quad \zeta} & \ar[r] k & 0
}
\]
defining $L_\zeta$. Restricting along $\alpha$, we get 
\[\alpha^*(\Omega^m(k)) \ \simeq \
C^{\oplus \gamma_m} \oplus \Omega^m(k_C),
\] 
where the rank  $\gamma_m$ of the free summand is determined entirely by
the dimensions of the other two modules. Explicitly, 
$\gamma_m = (\Dim \Omega^m(k_E)-\Dim \Omega^m(k_C))/p^r$, 
which depends only on $m$ and $r$. 
The case that 
$\alpha^*(\zeta) \neq 0$ is now clear from the restriction of 
the sequence. In the case that $\alpha^*(\zeta) = 0$,
we have that the map $\zeta$ in the sequence vanishes on the 
$C$-summand $\Omega^m(k_C)$. It is an easy exercise
to show that the restriction of the kernel of $\zeta$ in the 
sequence is as indicated (see also \cite[II, \S5.9]{B}).

For the computations of the dimensions of $\Rad(\alpha^*(L_\zeta))$, we
recall that $\Omega(k_C) \simeq \Rad(C)$ and, hence, 
$\Dim \Rad(\Omega(k_C)) = p^r -1-r$.
\end{proof}

The relevance of Proposition \ref{dim} to radical types of $L_\zeta$ modules is made 
explicit in the following theorem.

\begin{thm}
\label{constant-r}
Suppose that $\zeta \in \HHH^m(E,k)$ is a non-nilpotent  
cohomology class satisfying the condition that  the hypersurface
$$Z(\zeta)  \ \subset \ \Proj \HHH^\bu(E,k)$$  
does not contain a linear hyperplane of dimension $r-1$.  Then $L_\zeta$
has constant $r$-radical type.
\end{thm}

\begin{proof} For any $U \in \Grass(r, \CV)$, let $\alpha: C = 
k[t_1, \dots, t_r]/(t_1^p, \dots, t_r^p) \to kE$ be a homomorphism with
$\{ \alpha(t_1),\ldots, \alpha(t_r)] \}$ a basis for $U$.
By Corollary~\ref{cor:twist}, we may identify 
$\alpha_*: \Proj \HHH^\bu(C,k)  \  \to \  \Proj \HHH^\bu(E,k)$ with the linear 
embedding of projective spaces associated to the embedding 
$U^{(1)} \subset \CV^{(1)}$.
Hence, the image of $\alpha_*$ is a linear 
subspace of dimension $r-1$.
Our hypothesis implies that the image of $\alpha_*$  can not be in the zero set
of $\zeta$ and, therefore, the restriction $\alpha^*(\zeta)\in \HHH^*(C,k)$ 
is not nilpotent.   

Since $\HHH^*(C,k)$ is a product of a symmetric algebra 
and an exterior algebra, this implies that $\alpha^*(\zeta)$ is not a zero divisor. 
Hence, $\Ker\{ \cdot \alpha^*(\zeta): \HHH^1(C,k) \to \HHH^{m+1}(C,k) \} =0$. 
By Proposition \ref{dim}, we get that 
$\Rad^i(L_{\alpha^*(\zeta)}) = \Rad^i(\Omega^m k_C)$ for $i \geq  1$. 
Lemma~\ref{restr-Lzeta}  now implies that $\alpha^*(L_\zeta)$ has constant $r$-radical type. 
\end{proof}

We now see how $L_\zeta$-modules give us examples of modules of constant $r$-radical
type but not constant $s$-radical type for any $s$ with $1 \leq s < r$.

\begin{prop}
\label{notlower}
Suppose that $\zeta \in k[\zeta_1, \ldots, \zeta_n]$ is a homogeneous polynomial
of degree  $m$ such that the zero locus of $\zeta$ inside $\Proj k[\zeta_1, \ldots, \zeta_n] \simeq 
\Proj (\HHH^\bu(E,k)_{red})$ 
contains a linear subspace of dimension $r-2$ but not of  dimension $r-1$.
Then the $kE$-module $L_\zeta$ has constant $r$-radical type
but not  constant $s$-radical type for any $s, 1 \leq s< r$.
\end{prop}

\begin{proof} 
We view $\zeta$ as a homogeneous polynomial function on $\CV^{(1)}$ of degree $m$.
Theorem  ~\ref{constant-r}  implies that $L_\zeta$ has constant $r$-radical type. 

For $s<r$, we proceed to find $s$-planes $U, V \in \Grass(s, \CV)$ such that 
$\Dim \Rad_U(L_\zeta) \neq \Dim \Rad_V(L_\zeta)$.  
By assumption, we can find a linear $s$-subspace $\wt U \subset \CV \simeq \CV^{(1)}$ 
such that $\zeta$ vanishes on $\wt U$.   
Let $F_{\CV}: \CV \to \CV^{(1)}$  be the Frobenius map on $\CV$, and  let 
$U = F^{-1}(\wt U)$.  Note that $U$ is again a linear subspace of $\CV$, 
and by construction we have 
\[
U^{(1)} = F_{\CV}(U)  = \wt U
\]
   
Choose an ordered basis $\ul u=[u_1, \ldots, u_s]$ of $U$, and define
\[
\alpha: C = k[t_1, \dots, t_s]/(t_1^p, \dots, t_s^p) \to kE
\]
to be the flat $k$-algebra homomorphism defined by $\alpha(t_i) =u_i$.  
Corollary~\ref{cor:twist} enables us to identify 
\[
\alpha_*: \Spec (\HHH^\bu(C,k)_{red})  \to \  \Spec (\HHH^\bu(kE,k)_{red})
\]
with the inclusion $U^{(1)} \subset \CV^{(1)} \simeq \CV$ obtained by applying 
the Frobenius twist to $U \subset \CV$. Since $U^{(1)} = \wt U$, we conclude that 
$\alpha^*(\zeta)=0$. 
Applying Lemma~\ref{restr-Lzeta}, we get 
\[
\Dim\Rad_U(L_\zeta)=\Dim \Rad(\alpha^*(L_\zeta)) = 
\gamma_m(p^r -1) -r + \Dim\Rad(\Omega^m(k_C)).
\]

Now let $\wt W$ be a linear $s$-subspace in $\CV$ such that $\zeta$ 
does not vanish on $\wt W$, and let $W = F_{\CV}^{-1}(\wt W)$, so that $\wt W = W^{(1)}$.
Let $\ul w = [w_1, \ldots, w_s]$ be a basis of $W$, and let 
\[
\beta: C = k[t_1, \dots, t_s]/(t_1^p, \dots, t_s^p) \to kE
\]
be the flat $k$-algebra homomorphism defined by $\beta(t_i) =w_i$.   
Then $\beta^*(\zeta)$ is not 
nilpotent, and, in particular, 
\[
\Dim \Rad_W(L_\zeta) = \Dim\Rad(\beta^*(L_\zeta)) = 
\gamma_m(p^r-1) + \Dim\Rad(L_{\beta^*(\zeta)})
\]
by Lemma~\ref{restr-Lzeta}.
Since $\beta^*(\zeta)$ is not nilpotent, we conclude that 
$\Ker\{\beta^*(\zeta): \HHH^1(C,k) \to \HHH^{m+1}(C,k)\}=0$. 
Hence, by Prop.~\ref{dim}, 
\[\Dim\Rad_W(L_{\beta^*(\zeta)}) = \Dim \Rad(\Omega^m(k_C)).
\] Therefore,  
\[\Dim \Rad_W(L_\zeta) = \Dim \Rad_U(L_\zeta) + r
\]
which implies the desired inequality. 
\end{proof}

The following proposition provides examples of homogeneous polynomials 
which satisfy the condition of Proposition \ref{notlower}.
We are grateful to S\'andor Kov\'acs for suggesting the geometric argument in 
the proof that follows.

\begin{prop}
\label{prop:geom}
Let $n>r$ be positive integers. There exists a homogeneous 
polynomial $f \in k[X_0, \ldots, X_n]$ such that the zero locus 
of $f$, $Z(f) \subset \bP^n$, contains a linear subspace of 
dimension $r-1$ ($\bP^{r-1}$)  but not of  
dimension $r$. 
\end{prop}

\begin{proof}
Fix $L = \bP^{r-1}$ to be the projective subspace 
which is the zero set of the ideal
$$
\cI_L =(X_r,X_{r+1}, \ldots, X_n).
$$
Fix a positive degree $d$.   Then the set of polynomials 
$f$ of degree $d$  such that $L \subset Z(f)$ is the set 
of global sections of $\cI_L(d)$ on $\bP^n$, that is, 
$\HHH^0(\bP^n, \cI_L(d))$.  Indeed, we have an exact sequence
$$\xymatrix{0 \ar[r] & \cI_L \ar[r] &\cO_{\bP^n} \ar[r] & \cO_L \ar[r] & 0}$$
Twisting by $d$ and applying global sections, we get an exact sequence
$$\xymatrix{0 \ar[r] & \HHH^0(\bP^n, \cI_L(d)) \ar[r] 
&\HHH^0(\bP^n,\cO_{\bP^n}(d)) \ar[r] & \HHH^0(\bP^n,\cO_L(d))}$$
For a homogeneous polynomial of degree $d$ to 
vanish on $L$, it must go to zero under  the last map.  
Hence, it belongs to $\HHH^0(\bP^n, \cI_L(d))$.

We also note that the map $\HHH^0(\bP^n,\cO_{\bP^n}(d)) 
\to \HHH^0(\bP^n,\cO_L(d))$ is surjective since it is simply 
a projection on the first $r$ coordinates.  Hence,
$$
\Dim  \HHH^0(\bP^n, \cI_L(d)) = \Dim \HHH^0(\bP^n,\cO_{\bP^n}(d)) - 
\Dim \HHH^0(\bP^n,\cO_{L}(d)).
$$
We compute the  right hand side:  
$\Dim \HHH^0(\bP^n,\cO_{\bP^n}(d)) = 
\Dim k[X_0, \ldots, X_n]_{(d)} = {n+d \choose d}$,
$\Dim \HHH^0(\bP^n,\cO_{L}(d)) = 
\Dim k[X_0, \ldots, X_{r-1}]_{(d)} = {r+d -1 \choose d}$.  Hence,
$$ \Dim  \HHH^0(\bP^n, \cI_L(d)) = {n+d \choose d} - {r+d -1 \choose d}.$$

Now, let $L^\prime  = \bP^{r}$ be any linear subspace 
of dimension $r$.  Such subspaces are 
parametrized by $\Grass(r+1, n+1)$.
For each one,  the corresponding space of homogeneous 
functions of degree $d$ that vanish on $L^\prime$ has dimension  
$\Dim \HHH^0(\bP^n,\cI_{L^\prime}(d)) = {n+d \choose d} - 
{r+d \choose d}$.  
Let  $T \subset \Grass(r+1, n+1)\times \HHH^0(\bP^n, \cO_{\bP^n}(d))$
be a subspace defined as follows:
$$
T = \{ (L^\prime, f), \bP^r =L^\prime \subset \bP^n, f \in \HHH^0(\bP^n,
\cI_{L^\prime}(d)) \}.
$$
This is a vector bundle with the fiber of dimension $\Dim
 \HHH^0(\bP^n, \cI_{L^\prime}(d))$ and the base $\Grass(r+1, n+1)$, 
and it is precisely the space of functions we need to avoid.
Hence, altogether we need to avoid a 
total space of dimension
$$
\Dim T = \Dim \HHH^0(\bP^n,\cO_{L^\prime}(d)) + \Dim
\Grass(r+1, n+1) = 
$$
$$
{n+d \choose d} - {r+d \choose d} + (r+1)(n-r).
$$
Therefore, to prove the claim, we need to establish that for a large 
enough $d$, we have an inequality
$$
{n+d \choose d} - {r+d-1 \choose d} > 
{n+d \choose d} - {r+d \choose d} + (r+1)(n-r).
$$
This is equivalent to the conditions that 
$$ 
{r+d \choose d} -  {r+d-1 \choose d} > (r+1)(n-r). 
$$
and
$$
{r+d-1 \choose d-1} > (r+1)(n-r).
$$
Since $r$ and $n$ are fixed but $d$ can be chosen 
arbitrarily large, this is now evident.
\end{proof}

The following corollary is immediate from Prop.~\ref{notlower} and \ref{prop:geom}.

\begin{cor}
\label{ex-rad} Let $E$ be an elementary abelian $p$-group of rank $n$.
For any integer $r$, $1 <r < n$, there exists a module 
of constant $r$-radical type but not of constant $s$-radical type
for $s<r$.  
\end{cor}

We next construct examples of a $kE$-modules which have 
constant $r$-radical type for small $r$,
but not for large $r$.

\begin{prop}  \label{nothigher}
Assume that $p > 2$.  As before we write
$\HHH^*(E,k) = k[\zeta_1, \ldots, \zeta_n] \otimes \Lambda^*(\eta_1, \ldots, \eta_n)$. 
Let $\zeta = \eta_1\ldots \eta_s$ for some $s$ with $1 < s <n$.
Then $L_\zeta$ satisfies the following properties:
\begin{enumerate}
\item  $L_\zeta$ has constant $r$-radical type for any $r$, $r < s$.
\item  $L_\zeta$ has constant $s$-$\Rad$-rank,  but not constant $s$-radical type.
\item $L_\zeta$ does not have constant $r$-$\Rad$-rank for any $r$ such that $s<r<n$.
\end{enumerate}
\end{prop}

\begin{proof} Let $U$ be an $r$--plane in $\CV$, and let 
$\alpha: C = k[t_1,\ldots,t_r]/(t_1^p, \dots, t_r^p) \to kE$ 
be a map such that $\alpha(t_1), \dots, \alpha(t_r)$ is a basis for $U$. 
 
For $r < s$, the product of any $s$ elements of degree one  is necessarily 
zero in $\HHH^s(C,k)$. Hence,  $\alpha^*(\zeta) = 0$.  
By Lemma~\ref{restr-Lzeta},
$\alpha^*(L_\zeta) \simeq C^{\oplus \gamma_s-1} \oplus \Omega(k_C) \oplus \Omega^s(k_C)$ 
for some $\gamma_s$ which does not depend on the choice of $U$.
Consequently, $L_\zeta$  has constant $r$-radical type.

Assume that $s\leq r\leq n$. If $U$ is the subspace such as the one 
spanned by $x_1, \dots, x_r$, then $\alpha^*(\zeta) \neq 0$.  
Since $\alpha^*(\zeta)$ is a product of $s$ degree $1$ classes, 
it annihilates a subspace of dimension $s$ of $\HHH^1(C,k)$. Hence, 
Lemma \ref{restr-Lzeta} and Proposition \ref{dim} imply that 
\begin{equation} 
\label{eq:dim1}
\begin{array}{ll}\Dim \Rad_U(L_{\zeta}) & =  \gamma_s(p^r-1)  + \Dim \Rad(L_{\alpha^*(\zeta)}) \\[5pt]
& = \gamma_s(p^r-1) + \Dim \Rad(\Omega^r(k_C)) -s 
\end{array}
\end{equation}
If $U$ is the 
subspace spanned by $x_2, \dots, x_{r+1}$, then 
$\alpha^*(\zeta) = 0$ and
\begin{equation}\label{eq:dim2}
\Dim \Rad_U(L_{\zeta}) = \gamma_s(p^r-1)-r + \Dim \Rad(\Omega^r(k_C)).
\end{equation}
It follows that 
$L_\zeta$ has constant $r$-$\Rad$--rank
if and only if $r = s$. This proves (3) and the first part
of (2).

For the remainder of part (2), notice that
the dimension of  $\Rad^{r(p-1)}(M)$ of a $C$-module $M$ counts 
 the number of direct summands 
of $C$ in a decomposition of the module into indecomposable
submodules. In the case $r\geq s$, we can get two 
different values for $\Dim \Rad^{r(p-1)}(\alpha^*(L_\zeta))$ depending 
on whether $\alpha^*(\zeta)$  is zero or not, by Lemma \ref{restr-Lzeta}.
Therefore $L_\zeta$ does not have 
constant $r$-$\Rad^{r(p-1)}$-rank for any $r \geq s$. In particular, it does not 
have constant $s$-radical type.   
\end{proof}


\begin{cor} \label{r-supp1} Let $p>2$, and let $\zeta \in \HHH^s(E,k)$ be 
a product of $s$ degree one cohomology classes.  For any $r>s$ the 
nonmaximal radical support variety $\Rad(r, \CV)_{L_\zeta}$ consists of exactly 
those $r$-planes $U$ for which $\alpha^*(\zeta) = 0$, where 
$\alpha: k[t_1, \dots, t_r]/(t_1^p, \dots, t_r^p) \to kE$ is 
a map such that $\alpha(t_1), \dots, \alpha(t_r)$ form a 
basis for $U$. 
\end{cor}

\begin{proof}
This follows by comparing equalities \eqref{eq:dim1} and \eqref{eq:dim2} of the proof of Prop.~\ref{nothigher}.
\end{proof}

In a similar way, we get the following statement about nonmaximal radical 
support varieties.
\begin{cor} \label{r-supp2}
Let $\zeta \in \HHH^{2m}(E,k)$.
If $r=1,2,3$ or if $\zeta$ is a product of one-dimensional cohomology classes,
then
\[\Rad(r, \CV)_{L_\zeta} = \{ U \in \Grass(r, \CV) \, | \, \alpha^*(\zeta)=
0\ \text{ in } \HHH^*(C,k)_{\rm red}\}
\]
where  
$\alpha: C = k[t_1, \dots, t_r]/(t_1^p, \dots, t_r^p) \to KE$ is 
a map such that $\alpha(t_1), \dots, \alpha(t_r)$ form a 
basis for $U$.
 
On the other hand, for $r> 3$, there exists a homogeneous cohomology class $\zeta$ for which this equality is not valid.
\end{cor}

\begin{proof} If $\zeta \in \HHH^{2m}(E,k)$ satisfies the hypothesis of the corollary,
then  the condition that  $\Ker \{\cdot \alpha_{\ul u}^*(\zeta): \HHH^1(C,k) 
\to \HHH^{2m+1}(C,k)\}$ be zero
is equivalent to a simpler condition that $\alpha^*_{\ul u}(\zeta)$ is not  nilpotent. 
Hence, Prop.~\ref{dim}  
implies the desired equality.

On the other hand, suppose that $r>3$ and let $\eta_1,
\dots, \eta_r$ span  $\HHH^1(E,k)$.
Then $\eta = \eta_1\eta_2+\eta_3\eta_4$ is a nilpotent element in 
$\HHH^*(C,k)$ which  does not annihilate any non-zero class of degree 1.
\end{proof}

We finish this section with a simple observation about the socle series of $\alpha^*(L_\zeta)$.

\begin{prop} \label{soc-rank}
Suppose that $\zeta \in \HHH^m(E,k)$  is a non-zero cohomology class.
If $r >1$, then for any $U$ in $\Grass(r, \CV)$ we have that
$\Soc_U(L_\zeta) = \Soc_U(\Omega^m(k)).$

\vspace{0.05in}
Consequently, $L_\zeta$ has constant $r$-$\Soc$-rank for any $r>1$. 
\end{prop}

\begin{proof}
Choose $U$ in $\Grass(r, \CV)$. Let $\alpha: C= k[t_1, \dots, t_r]/(t_1^p, 
\dots, t_r^p) \to kE$ be a $k$-algebra homomorphism such that
 $\alpha(t_1), \dots, \alpha(t_r)$ is a basis for $U$. 
Suppose there is a simple submodule in $\alpha^*(\Omega^m k)$ 
which does not map to $0$ under $\alpha^*(\zeta)$ 
and, hence, is not a submodule in $\Soc\alpha^*(L_\zeta)$. 
Then it maps isomorphically onto $k$. This implies that 
the sequence 
$0 \to \alpha^*(L_\zeta) \to 
\alpha^*(\Omega^m(k)) \to k  \to  0$ 
splits. But if $r>1$, then this 
is not possible because $\alpha^*(\Omega^m(k))$ has no summand that 
is isomorphic to $k$. 
\end{proof}


\section{Construction of Bundles on $\Grass(r, \CV)$}
\label{se:construction}
This section opens the second part of the paper in which we 
discuss algebraic vector bundles on Grassmannians arising from 
finite dimensional $kE$-modules 
having either constant $r$-$\Rad^j$-rank or constant 
$r$-$\Soc^j$-rank for some $j$. We begin by 
developing two approaches of constructing vector bundles 
on $\Grass(r, \CV)$ which we then show determine isomorphic algebraic vector bundles. 
The first approach uses a 
local analysis on standard affine open subsets of the Grassmannian,
while the second is a global process defining the 
bundles by equivariant descent. In the next section we show that for the class 
of $\GL_n$-equivariant $kE$-modules discussed in Section~\ref{se:const}, 
our construction can be recognized as a familiar functor widely used 
for algebraic groups and homogeneous spaces.  Our first series of 
examples appears in the same section.  Finally, in Section~\ref{global} 
we introduce a formula that constructs
a graded module over the homogeneous coordinate ring
of the Grassmannian whose associated coherent sheaf is the kernel bundle
associated to a module of constant $r$-socle rank.  

We use notations and conventions for the Grassmannian discussed in 
detail in Section \ref{se:grass}. 

\subsection{A local construction of bundles}
\label{se:local} 
Let $x_1, \dots, x_n$ be a basis for the space $\CV \subset \Rad(kE)$ splitting the projection 
$\Rad(kE) \to \Rad(kE)/\Rad^2(kE)$. Let $C=k[t_1, \ldots, t_r]/(t_1^p, \ldots, t_r^p)$.  For 
\[
\xymatrix{\alpha_{\Sigma} \,:\, C \otimes k[\cU_{\Sigma}] \ar[r]& kE \otimes k[\cU_{\Sigma}]}
\]
as in Definition~\ref{alpha-S}, we  denote by 
$\theta_j^\Sigma$, $1 \leq j \leq r$, the $k[\cU_{\Sigma}]$--linear 
$p$-nilpotent operator on $M \otimes k[\cU_{\Sigma}]$ given by multiplication 
by $\alpha_{\Sigma}(t_j)$:
\begin{equation}
\label{theta}
\xymatrix@=2mm{M \otimes k[\cU_{\Sigma}] \ar[r]^-{\theta^\Sigma_j}& M \otimes k[\cU_{\Sigma}] \\
 m \otimes f \ar@{|->}[r]&
\sum\limits_{i=1}^n x_im \otimes Y_{i,j}^\Sigma f.
}\end{equation}
For any $r$-subset $\Sigma\subset \{1, \ldots, n\}$, and  any $\ell$, $1 \leq \ell \leq r(p-1)$, we define   $k[\cU_\Sigma]$--modules
\begin{align}
\label{ker_loc}
\Ker^\ell(M)_{U_\Sigma} = \bigcap\limits_{1\leq j_1, \ldots, j_\ell \leq r}
\Ker\{ \theta_{j_1}^\Sigma\cdots\theta_{j_\ell}^\Sigma: M \otimes k[\cU_\Sigma]  \to 
M \otimes k[\cU_\Sigma]\} \\
\Im^\ell(M)_{U_\Sigma} = \sum\limits_{1\leq j_1, \ldots, j_\ell \leq r}
\Im\{\theta_{j_1}^\Sigma\cdots\theta_{j_\ell}^\Sigma: M \otimes k[\cU_\Sigma] \to 
M \otimes k[\cU_\Sigma]\}
\end{align}

We denote by $\cO_{Gr}$ the structure sheaf of $\Grass(r,\CV)$.  For any
finite dimensional $kE$-module $M$, the coherent sheaf $M \otimes \cO_{Gr}$ is a free 
$\cO_{Gr}$-module of rank
equal to the dimension of $M$. In the next proposition, we define the $\ell^{\rm th}$ kernel and image sheaves,
\begin{equation}
\cKer^\ell(M) \quad \text{ and } \quad \cIm^\ell(M),
\end{equation}
associated to a $kE$--module $M$. 

\begin{prop} \label{localdef}  Let $M$ be a finite-dimensional $kE$--module.    There is a 
unique subsheaf $\cKer^\ell(M) \subset M \otimes \cO_{Gr}$ whose restriction to 
$U_\Sigma$ equals $\Ker^\ell(M)_{U_\Sigma} $ for each subset
$\Sigma \subset \{1,\dots, n\}$ of 
cardinality $r$.   We refer to $\cKer^\ell(M)$ as the $\ell^{th}$ kernel sheaf.

Similarly, there is a unique subsheaf $\cIm^\ell(M) \subset M \otimes \cO_{Gr}$ whose restriction to 
$U_\Sigma$ equals $\Im^\ell(M)_{U_\Sigma} $ for each subset $\Sigma \subset \{1,\ldots, n\}$ of 
cardinality $r$.  We refer to $\cIm^\ell(M)$ as the $\ell^{th}$ image sheaf.
\end{prop}

\begin{proof}
Let $\Sigma, \Sigma^\prime \subset \{1, \ldots, n\}$ be two $r$-subsets and let
$$
\tau_{\Sigma,\Sigma^\prime}: k[Y_{i,j}^\Sigma,\fp_{\Sigma^\prime}^{-1}] \ 
\simeq k[\CU_\Sigma\cap \CU_{\Sigma^\prime}] \
\simeq k[Y_{i,j}^{\Sigma^\prime},\fp_\Sigma^{-1}]
$$
denote the evident transition function.
Observe that on $\CU_\Sigma\cap \CU_{\Sigma^\prime}$, 
each $\theta_j^\Sigma$ can be written 
using the transition functions
$\tau_{\Sigma,\Sigma^\prime}$ as a $k[Y_{i,j}^{\Sigma^\prime},\fp_\Sigma^{-1}]$-linear 
combination of the $\theta_j^{\prime \Sigma}$'s:
\begin{equation}
\label{eq:rel}
\theta_j^{\Sigma^\prime} = \tau_{\Sigma, \Sigma^\prime}(\theta_j^\Sigma):  
M \otimes k[Y_{a,b}^{\Sigma^\prime},\fp_\Sigma^{-1}] \to
M \otimes k[Y_{a,b}^{\Sigma^\prime},\fp_\Sigma^{-1}].
\end{equation}
This enables us to identify $\Ker^\ell(M)_{U_\Sigma}$  and $\Im^\ell(M)_{U_\Sigma}$   
when restricted to $\CU_{\Sigma} \cap \cU_{\Sigma^\prime}$ as submodules 
of $M \otimes k[Y_{a,b}^{\Sigma^\prime},\fp_\Sigma^{-1}]$.
It can be verified that the kernels and images of the products of 
$\theta_j^\Sigma, \ \theta_j^{\Sigma^\prime}$ acting on  
$M \otimes k[Y_{a,b}^{\Sigma^\prime},\fp_\Sigma^{-1}]$
are equal by specializing to each point 
$x \in \CU_\Sigma\cap \CU_{\Sigma^\prime}$  and  using the relationship \eqref{eq:rel}.
\end{proof}

For $\ell = 1$, we write $\cKer(M)$ for $ \cKer^1(M)$, and we write $\cIm(M)$ for $\cIm^1(M)$.

\begin{thm}
\label{th:bundle} Let $M$ be a finite dimensional $kE$--module, and $U \subset \CV$ an $r$-plane.
Let $\ell$ be an integer, $1 \leq \ell \leq (p-1)r$. 
\begin{enumerate}
\item
 If $M$ has constant $r$--$\Soc^\ell$--rank, then 
\begin{itemize}
\item[$\circ$]   $\cKer^\ell(M)$   is an algebraic 
vector bundle on $\Grass(r, \CV)$, 
\item[$\circ$] $\rk \cKer^\ell(M) = \dim \Soc_U^\ell(M)$.
\end{itemize}
\item If  $M$ has constant $r$-$\Rad^\ell$-rank, then 
\begin{itemize}
\item[$\circ$]   $\cIm^\ell(M)$   is an algebraic 
vector bundle on $\Grass(r, \CV)$, 
\item[$\circ$] $\rk \cIm^\ell(M) = \dim \Rad_U^\ell(M)$.
\end{itemize}
\end{enumerate}
 \end{thm}

\begin{proof}
 First assume that $\ell = 1$. \\
(1). Let $\Sigma$ be an $r$-subset of $\{1, \ldots, n\}$. We proceed to define a map 
\begin{equation} 
\label{eq:theta}
\xymatrix{
\Theta^\Sigma: M \otimes k[\cU_\Sigma] \ar[rr]^-{[\theta_1^\Sigma, \ldots, \theta_r^\Sigma]}&
& (M \otimes k[\cU_\Sigma])^{\oplus r}}
\end{equation} 
such that $\Ker(M)_{\cU_\Sigma} = \Ker \Theta^\Sigma$.
Let $U\in \cU_\Sigma \subset \Grass(r, \CV)$ and let $\{u_1, \ldots, u_r\}$ 
be the unique choice of ordered basis for $U$ such that
the $\Sigma$--submatrix of $A_U=(a_{i,j})$ equals $ [u_1, \ldots, u_r]$ 
(expressed with respect to the fixed basis $\{x_1, \ldots, x_n\}$ of $\CV$) 
is the identity matrix.  
Then $\alpha_U: C \to kE$, defined by $\alpha_U(t_i)=u_i$, equals the result 
of specializing  $\alpha_\Sigma: C \otimes k[\cU_\Sigma] \to kE \otimes 
k[\cU_\Sigma]$ by setting the variables  $Y_{i,j}^\Sigma$ to values $a_{i,j} \in k$. Hence, 
the specialization of the  map $\Theta^\Sigma$ at the point  $U\in \cU_\Sigma$ gives the $k$-linear map 
$[\alpha_U(t_1), \ldots, \alpha_U(t_r)]: M \to M^{\oplus r}$. In other words,
\[\Theta^\Sigma \otimes_{k[\cU_\Sigma]} k = [\alpha_U(t_1), \ldots, \alpha_U(t_r)]\]
where the tensor is taken over the map $k[\cU_\Sigma] \to k$  corresponding to the point $U \in \cU_\Sigma$. 
Since specialization is right exact,  we have an equality
\[\Coker\{ \Theta^{\Sigma}\}  \otimes_{k[\cU_\Sigma]} k =  \Coker \{\Theta^{\Sigma}  \otimes_{k[\cU_\Sigma]} k \} = 
\Coker\{[\alpha_U(t_1), \ldots, \alpha_U(t_r)]: M \to M^{\oplus r}\}.\]

Let $f: W \to W^\prime$  be a linear map of $k$-vector spaces.  Then $\Dim \Coker f  =\Dim \Ker f - 
\Dim W + \Dim W^\prime$.  
Using this observation, we further conclude that
\[ \Dim \Coker\{[\alpha_U(t_1), \ldots, \alpha_U(t_r)]: M \to M^{\oplus r}\} = \]
\[\Dim \Ker\{[\alpha_U(t_1), \ldots, \alpha_U(t_r)]: M \to M^{\oplus r}\} + (r-1)\Dim M= \]
\[\Dim \Soc_U(M) + (r-1)\Dim M.\]
Therefore, all specializations of the $k[\cU_\Sigma]$--module $\Coker \Theta^\Sigma$ have the 
same dimension. By   \cite[4.11]{FP3} (see also \cite[5 ex.5.8]{Har}),  $\Coker \Theta^\Sigma$ 
is a projective module over $k[\cU_\Sigma]$.     Now the exact sequence 
\[\xymatrix{0 \ar[r] &\Ker \Theta^\Sigma\ar[r] & M \otimes k[\cU_\Sigma] 
\ar[r]^{\Theta^\Sigma} & (M \otimes  k[\cU_\Sigma])^{\oplus r} \ar[r] & \Coker \Theta^\Sigma \ar[r] & 0 } 
\]
implies that $\Ker(M)_{\cU_\Sigma} = \Ker \Theta^\Sigma$ is also projective.  
Since this holds for any $r$-subset $\Sigma \subset \{1, \ldots, n\}$, we conclude that $\cKer(M)$ is locally free. 

(2). For an $r$-subset $\Sigma \subset \{1, \ldots, n\}$, define a map 
$\Theta^\Sigma: (M\otimes k[\cU_\Sigma])^{\oplus r} \to M\otimes k[\cU_\Sigma]$ as the composition
\[
\xymatrix{\Theta^\Sigma: (M\otimes k[\cU_\Sigma])^{\oplus r}\ar[rr]^-{\diag [\theta^\Sigma_1, \ldots, \theta^\Sigma_r]}&& 
(M\otimes k[\cU_\Sigma])^{\oplus r}\ar[r]^-{\sum}& M\otimes k[\cU_\Sigma] }
\]
where the second map is the sum  over all $r$ coordinates.  Arguing as in (1) and 
using  that $\dim \Coker f = \dim W^\prime - \dim \Im f$ for a map of $k$-vector spaces 
$f: W \to W^\prime$, we conclude (2) for $\ell = 1$. 

Finally, the proof for $\ell>1$ is very similar with the map $\Theta^\Sigma$ replaced by its $\ell$-th
iterate.
\end{proof}

The two basic examples we give below can be justified directly from the local construction just 
described;  indeed, both are defined in terms of moving frames inside trivial bundles of 
appropriate ranks on the Grassmannian.  Formal verifications are given in Examples~\ref{ex:taut} and \ref{ex:uni}.   

\begin{ex}
\label{ex:first}
 (1) [Tautological/universal subbundle $\gamma_r$]. Let $kE = k[x_1, \ldots, x_r]/(x_1^p, \ldots, x_n^p)$, 
and let $M=kE/\Rad^2(kE)$. 
We can represent $M$ pictorially as follows:
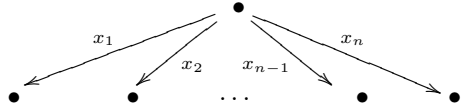
\[
\xymatrix{
&& \ \bu \ar[dll]_{x_1} \ar[dl]^{x_2} \ar[dr]_{x_{n-1}} \ar[drr]^{x_n}&&\\
\qquad \bu & \quad \bu & \ldots & \quad \bu & \bu
}
\]
Then $\Rad_U(M) \subset \Rad(M)$ can be naturally identified with the plane $U \subset \CV$ under  our fixed isomorphism $\Rad(M) =\Rad(kE)/\Rad^2(kE) \simeq \CV$.  Thus,
$$\cIm(M) \quad = \quad  \ \gamma_r, $$
where $\gamma_r \subset \ \cO_{Gr}^{\oplus n} \simeq \Rad(M) \otimes \cOG$
is the {\it tautological} (or universal) rank $r$ subbundle  
of the rank $n$ trivial bundle on $\Grass_{n,r}$.   

(2) [Universal subbundle $\delta_{n-r}$]. Let $\delta_{n-r}$ be the universal rank $n-r$ subbundle  
of the trivial bundle  of rank $n$ on $\Grass_{n,r}$, that is, the subbundle whose dual, $\delta_{n-r}^\vee$, 
fits into a short exact sequence   
\[\xymatrix{ 0 \ar[r]& \gamma_r \ar[r]& \cOG^{\oplus n} \ar[r]& \delta_{n-r}^\vee  \ar[r] & 0}. \] 
Let $M=kE/\Rad^2(kE)$. Note that $M^\#$ can be represented pictorially as follows:
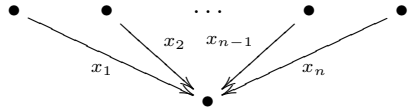
\[\xymatrix{
\bu\ar[drr]_{x_1}& \bu\ar[dr]^{x_2}& \ldots & 
\bu\ar[dl]_{x_{n-1}}& \bu\ar[dll]^{x_n}\\
&&\bu&&\\}
\]
We have 
$$\{ \cKer(M^\#)\  \subset \ M^\#\otimes \cO_{Gr} \} \quad = \quad \{\delta_{n-r} 
\oplus \cO_{Gr}  \ \subset \
\cO_{Gr}^{\oplus n+1}\}.$$

\end{ex}

\subsection{A construction by equivariant descent}
\label{se:descent} 
Our second construction has the 
advantage of producing bundles on $\Grass(r, \CV)$
by a  ``global" process rather 
than as a patching of locally defined kernels
or images.  In this sense, it  
resembles the global operator $\Theta$ in the case $r=1$ 
employed in \cite{FP3} to construct bundles on cohomological 
support varieties of infinitesimal group schemes.  However, 
the reader should be alert 
to the fact that the kernels (or images)
are not produced as kernels (or images) of a map 
of bundles on $\Grass(r, \CV)$ but rather by a descent process.

We begin by recalling the definition of a $G$-equivariant sheaf 
followed by a general lemma. We refer the reader to \cite[5]{CG} or \cite[I.0]{BL} for 
a detailed discussion of equivariant sheaves. 

\begin{defn}
\label{de:equiv1} Let $G$ be  a linear algebraic group and let $Y$ be a $G$-variety;
in other words, $Y$ is a variety equipped with an algebraic $G$--action
$\mu: G\times Y \to Y$. Denote  by $p: G\times Y \to Y$ the projection map, 
and by $m:G \times G \to G$ multiplication in $G$.  
A sheaf $\cF$ of $\cO_Y$--modules is $G$-equivariant if there 
is an isomorphism $f: \mu^*\cF \simeq p^* \cF$ satisfying the natural cocycle condition. 
Explicitly,  for  
\[\begin{array}{l}
p_1=  \id_G \times \mu: G \times G \times Y \to G\times Y \\
p_2=m \times \id_Y: G \times G \times Y \to G\times Y\\
p_3=\proj_{G\times Y}: G \times G \times Y \to G\times Y 
\end{array}\]
\noindent
(where $p_3$ is the projection along the first factor), $\cF$ satisfies the condition
\begin{equation}
\label{Gequiv}
p_1^*(f)\circ p_3^*(f) \ = \ p_2^*(f).
\end{equation}
\end{defn}

The following fact is well known although usually mentioned without proof (e.g., \cite[0.3]{BL} or \cite[5.2.15]{CG}).  
We provide a straight-forward proof for completeness.

\begin{lemma} \label{descent}   Let $G$ be a linear algebraic group and 
let $p: Y \to X$ be a principal homogeneous space for $G$ locally trivial
in the etale topology.  
There is an equivalence of categories
given by the pull-back functor
$$\xymatrix{p^*: \Coh(X) \ar[r]^-\sim & \Coh^G(Y)}$$
between coherent sheaves of $\cO_X$--modules and $G$-equivariant 
coherent sheaves of $\cO_Y$--modules.
\end{lemma}

\begin{proof}  Note that our assumption implies that 
$Y \to X$ is faithfully flat and quasi-compact.  Hence, 
we can use faithfully flat descent (\cite{SGAI}, VIII, \S.1). 
Therefore, we have an equivalence between  the category of 
coherent  sheaves of $\cO_X$--modules and the category of coherent 
sheaves of $\cO_Y$--modules with descent data.  Consider the diagram 
$$
\xymatrix{Y \times_{X} Y \ar[r]^-{\pi_1} \ar[d]^-{\pi_2} & Y \ar[d] \\
Y \ar[r] & X }
$$
Recall that the descent data for an $\cO_Y$--module $\cF$  is an isomorphism 
$\phi: \pi_1^*(\cF) \simeq \pi_2^*(\cF)$ such that
\begin{equation}
\label{descdata}
\pi^*_{23}(\phi)\pi^*_{12}(\phi) \  = \ \pi^*_{13}(\phi),
\end{equation}
where $\pi_{ij}: Y \times_X Y\times_X Y \to Y\times_X Y$ is the projection 
on the $(i,j)$ component. Since $p:Y \to X$ is a principal homogeneous space
for $G$ (i.e., a $G$-torsor for $G\times X$ over $X$), $G \times Y \ \to \  Y \times_X Y$ 
defined by sending $(g, y)$ 
to $(gy, y)$ is an isomorphism. With this isomorphism, the Cartesian square above becomes
$$\xymatrix{G \times Y \ar[r]^{\rm \mu} \ar[d]^-{p} & Y \ar[d] \\
Y \ar[r] & X }
$$
and the maps $\pi_{i,j}: Y \times_X Y\times_X Y \to Y\times_X Y$ become  precisely 
the maps in Definition~\ref{de:equiv1} with $\pi_{i,j}$ going to $p_\ell$ for $\ell \not = i,j$
\[\begin{array}{l}
p_1: G \times G \times Y \to G\times Y \\
p_2: G \times G \times Y \to G\times Y\\
p_3: G \times G \times Y \to G\times Y. 
\end{array}\]
\noindent
\noindent 
Consequently, the descent data (\ref{descdata}) is transformed into the condition
(\ref{Gequiv}) for $G$-equivariance.
\end{proof}

\begin{remark}
\label{re:triv} 
Suppose $p: Y \to X$ is a trivial $G$-fiber bundle, that is, there is a section 
$s: X \to Y$ such that $Y = s(X) \times G \simeq X \times G$. 
In this special case, $p^*$  is given simply by tensoring with the structure sheaf of $G$: 
for $\cF \in \Coh(X)$,
$$p^*(\cF) = s_*(\cF) \otimes \cO_G \simeq \cF \otimes \cO_G.$$
\end{remark}

\vspace{0.1in}
We fix an ordered basis of $\CV$ and  an $r$-plane $U_0\subset \CV$. As in \eqref{pinned}, 
we identify $\bM = M_{n,r}\simeq \Hom_k(U_0, \CV)$  with the 
affine variety of $n\times r$ matrices, and we set $\bM^o \subset \bM$  equal to the open 
quasi-affine subvariety of matrices  of maximal rank.  
Then $\Grass_{n,r} \simeq \bM^o/\GL_r$ and, moreover, 
$\bM^o \to \Grass_{n,r}$ is a principal $\GL_r$-equivariant bundle.
Hence, we have an equivalence of categories
\begin{equation}
\label{equiv}
\Coh(\Grass_{n,r}) \simeq \Coh^{\GL_r} (\bM^o).
\end{equation}
Moreover, using the action of $\GL_n$ on $\bM$ via multiplication on the left which 
commutes with the action by $\GL_r$ (via multiplication by the inverse on the right), 
we get an equivalence  between ($\GL_n$, $\GL_r$)-equivariant sheaves on $\bM^o$ 
(with $\GL_n$ acting on the left and $\GL_r$ on the right) 
and $\GL_n$-equivariant sheaves on $\Grass_{n,r}$ (with $\GL_n$ acting on 
$\Grass_{n,r}\simeq\GL_n/\GL_r$ via multiplication on the left).

We denote by
\begin{equation}
\label{eq:restr}
\xymatrix{\cR: \Coh^{\GL_r}(\bM) \ar[r]&  \ \Coh(\Grass_{n,r})}
\end{equation}
the functor defined as a composition 
\[\xymatrix{ \cR: \Coh^{\GL_r}(\bM) \ \ar[r]^-{\rm res}& \ \Coh^{\GL_r}(\bM^o)\ \ar[r]^\sim& \ \Coh(\Grass_{n,r})} \] 
of the restriction functor and the inverse to the pull-back functor which defines 
the equivalence of categories in Lemma~\ref{descent}.
Since $\bM$ is an affine scheme, the category of $\GL_r$-equivariant coherent 
$\cO_\bM$--modules is equivalent to the category of $\GL_r$--equivariant 
$k[\bM]$--modules. Using this equivalence, we apply the 
functor $\cR$ to $\GL_r$--equivariant $k[\bM]$--modules. 
Finally, recall that the choice of basis for $\CV$ determines the choice 
of the dual basis of $(\CV^{\oplus r})^\#$ which we denoted by $\{Y_{i,j}\}_
{\tiny\begin{array}{l}1\leq i\leq n\\ 1\leq j \leq r\end{array}}$ in Section~\ref{se:grass}. 
Since $k[\bM] = S^*(M_{n,r}^\#) = S^*((\CV^{\oplus r})^\#)$, we get that $\{Y_{i,j}\}$ are algebraic 
generators of $k[\bM]$.   We use the identification $k[\bM]\simeq k[Y_{i,j}]$.

Let $M$ be a finite dimensional $kE$-module, and let $\wt M = M \otimes k[Y_{i,j}]$ 
be a free module of rank $\Dim M$ over $k[Y_{i,j}]$.
We define a $k[Y_{i,j}]$-linear map
$$\wt \Theta=[\wt\theta_1, \ldots, \wt\theta_r]: \wt M \to (\wt M)^{\oplus r}$$
by
$$\wt\theta_j(m \otimes f) = \sum\limits_{i=1}^n  x_im \otimes Y_{i,j}f $$
for all $j$, $1 \leq j \leq r$.
We further define
\begin{equation}
\label{def:ker} 
\Ker \{ \wt \Theta, M\} = \Ker\{\wt \Theta: \wt M \to (\wt M)^{\oplus r} \}
\end{equation}
to be the $k[Y_{i,j}]$--submodule of $\wt M$ which is the kernel of the map 
$\wt \Theta$. Letting
$$
\wt \Theta^\ell=[\wt \theta_1^\ell, \ \wt \theta_1^{\ell-1}\wt\theta_2, \ 
\ldots, \ \wt\theta_r^\ell],
$$
(all monomials of degree $\ell$ in $\wt\theta_1, \dots, \wt\theta_r$)
we similarly define 
\begin{equation}
\label{der:kerl}
\Ker \{ \wt \Theta^\ell, M\} = \Ker\{\wt \Theta^\ell: \wt M \to (\wt M)^{\oplus {r+\ell -1  \choose \ell }} \}
\end{equation}
for any $\ell$, $1 \leq \ell \leq (p-1)r$. 
 
An analogous construction  is applied to the image. 
Let
\begin{equation}
\label{def:im}
\Im\{ \wt \Theta, M \} = \Im\{\xymatrix{ (M \otimes k[Y_{i,j}])^r\ar[rr]^-{
\diag [\wt\theta_1, \ldots, \wt\theta_r]}
&&(M \otimes k[Y_{i,j}])^r \ar[r]^-\sum& M \otimes k[Y_{i,j}]} \})
\end{equation}
Replacing $\wt \Theta$ with $\wt \Theta^\ell$, we obtain 
$k[Y_{i,j}]$--modules $\Im\{\wt\Theta^\ell, M\}$ for any $\ell$, $1 \leq \ell \leq (p-1)r$.

\begin{lemma}
\label{le:equiv} Let  $M$ be a $kE$-module. Then $\Ker\{\wt \Theta^\ell, M\}$, $\Im\{\wt \Theta^\ell, M\}$ are $\GL_r$-equivariant 
$k[\bM]$--submodules  of $M \otimes k[\bM]$  for any $\ell$, $1 \leq \ell \leq r(p-1)$, 
where the action of $\GL_r$ is trivial on $M$ and is given by the multiplication by the inverse on the right on $\bM$. 
\end{lemma}

\begin{proof} We prove the statement for $\Ker\{\wt \Theta, M\}$, other cases are similar.

Let $g \in \GL(U_0) \simeq \GL_r$ and denote the action of $g$ on $f \in k[\bM]\simeq k[Y_{i.j}]$ by $f \mapsto f^g$.  
Let $[A_g] \in \GL_r$ be the matrix that gives the action of $g$ on $M_{n,r}^\#$ with respect to the basis $\{Y_{i,j}\}$.
Consider the diagram (which is not commutative!)
\[\xymatrix{
M \otimes k[\bM] \ar[r]^-{\wt \Theta} \ar[d]^g& (M \otimes k[\bM])^{\oplus r}\ar[d]^g\\
M \otimes k[\bM] \ar[r]^-{\wt \Theta}& (M \otimes k[\bM])^{\oplus r}
}\]
Going to the right and then down, we get
\[
(\wt \Theta(m \otimes f))^g = \left(\begin{pmatrix} x_1& \ldots & x_n\end{pmatrix} \otimes
\begin{pmatrix}Y_{1,1}  & \ldots   & Y_{1, r}\\
\vdots & \ddots  & \vdots\\
&&\\
\vdots & \ddots  & \vdots\\
Y_{n,1}  &\ldots   &Y_{n, r} \\
\end{pmatrix}^g \right)(m \otimes f^g) = 
\] 
\[\left(\begin{pmatrix} x_1& \ldots & x_n\end{pmatrix} \otimes
\begin{pmatrix}Y_{1,1}  & \ldots   & Y_{1, r}\\
\vdots & \ddots  & \vdots\\
&&\\
\vdots & \ddots  & \vdots\\
Y_{n,1}  &\ldots   &Y_{n, r} \\
\end{pmatrix}\right)[A_g] (m \otimes f^g) = [\wt \theta_1(m\otimes f^g), \ldots, \wt \theta_r(m\otimes f^g)][A_g]. 
\]
Going down and to the left, we  get 
\[\wt \Theta(m \otimes f^g) = \left(\begin{pmatrix} x_1& \ldots & x_n\end{pmatrix} \otimes
\begin{pmatrix}Y_{1,1}  & \ldots   & Y_{1, r}\\
\vdots & \ddots  & \vdots\\
&&\\
\vdots & \ddots  & \vdots\\
Y_{n,1}  &\ldots   &Y_{n, r} \\
\end{pmatrix}\right)(m \otimes f^g) = 
\] 
\[ [\wt \theta_1(m\otimes f^g), \ldots, \wt \theta_r(m\otimes f^g)]. 
\]
Since the results differ by multiplication by an invertible matrix, we conclude that $\Ker\{\wt\Theta, M\}$ is a $\GL_r$-invariant submodule of $ M \otimes k[\bM]$.
\end{proof}

Lemmas~\ref{le:equiv} and \ref{descent}  imply that the $\GL_r$-equivariant sheaf $\Ker\{\wt\Theta^\ell, M\}$ (resp., $\Im\{\wt\Theta^\ell, M\}$) 
 descends to a 
coherent sheaf on $\Grass_{n,r}$ via the functor $\cR$. 
We denote the resulting sheaf by $\cR (\Ker \{ \wt \Theta^\ell, M\})$ 
(resp,  $\cR(\Ker \{\wt \Theta^\ell, M\})$).

Note that $\cKer^\ell(M)$ (resp., $\cIm^\ell(M)$) is a subsheaf of 
 $M \otimes \cOG$ by construction.    The equality 
 $\cR(\wt M) = M \otimes \cOG$ and the naturality of $\cR$ imply that 
 $\cR(\Ker \{\wt \Theta^\ell, M\})$ (resp., $\cR(\Im\{\wt \Theta^\ell, M \})$) 
 is also a subsheaf of $M \otimes \cOG$.     We now show that the 
 subsheaves   $\cKer^\ell(M)$ and $\cR(\Ker \{\wt \Theta^\ell, M\})$ (resp., 
 $\cIm^\ell(M)$ and $\cR(\Im\{\wt \Theta^\ell, M \})$) of $M \otimes \cOG$ are equal.

\begin{thm}
\label{pr:coincide}
 For any finite dimensional $kE$-module  $M$ and any interger $\ell$, $1 \leq \ell \leq (p-1)r$,
we have equalities of coherent $\cOG$--modules 
$$\cKer^\ell(M) = \cR(\Ker \{\wt \Theta^\ell, M\}),$$
$$\cIm^\ell(M) = \cR(\Im \{\wt \Theta^\ell, M\}).$$
\end{thm}

\begin{proof}   We establish the equality $\cKer(M) 
\simeq \cR(\Ker \{\wt \Theta, M\})$, other cases are similar.  

Let $\{x_1, \ldots, x_n\}$   be the fixed basis of $\CV$ so that $kE \simeq k[x_1, \ldots ,x_n]/(x_1^p, \ldots, x_n^p)$. 
Globally on $\bM$ the operator 
\[\wt\Theta = [\wt\theta_1, \ldots, \wt\theta_r]^T:\wt M \to (\wt M)^{\oplus r}
\] 
is given as a product
\begin{equation} 
\wt \Theta =  \begin{pmatrix} x_1& \ldots & x_n\end{pmatrix} \otimes
\begin{pmatrix}Y_{1,1}  & \ldots   & Y_{1, r}\\
\vdots & \ddots  & \vdots\\
&&\\
\vdots & \ddots  & \vdots\\
Y_{n,1}  &\ldots   &Y_{n, r} \\
\end{pmatrix} .
\end{equation}
Let $\Sigma=\{i_1, \ldots, i_r\}$,   $i_1 < \cdots < i_r$, be a subset of $\{ 1, \ldots, n\}$, 
and  let $\CU_\Sigma \subset \Grass_{n,r}$ be the corresponding principal open.
Let $\wt \CU_\Sigma \subset \bM^o\subset \bM$  be the 
principal open subset defined by the non-vanishing of the 
minor corresponsing to the columns numbered by $\Sigma$. Hence, $k[\wt \cU_\Sigma]$ 
is the localization of $k[\bM]$  at the determinant  of the matrix $[Y_{i_t,j}]_{1\leq t,j\leq r}$. 
Note that $\wt \cU_\Sigma$ is $\GL_r$-invariant subset of $\bM$ and 
that $\wt \cU_\Sigma \to \cU_\Sigma$ is a trivial $\GL_r$-bundle.
Denote by 
\[
\eta_{\cU_\Sigma}: \Coh^{\GL_r}(\wt \cU_\Sigma) \simeq \Coh(\cU_\Sigma)
\]
the corresponding equivalence of categories as in Lemma~\ref{descent}. 
As in Section~\ref{se:grass} (prior to Defn.~\ref{alpha-S}), we choose 
a section of $\bM^0 \to \Grass_{n,r}$ over $\cU_\Sigma$ defined 
by sending a $\GL_r$-orbit to its unique representative such that the 
$\Sigma$-matrix is the identiy matrix. This section splits the trivial 
bundle $\wt \cU_\Sigma \to \cU_\Sigma$ giving an isomorphism 
$\wt \cU_\Sigma \simeq \cU_\Sigma \times \GL_r$ and, hence, 
\[
k[\wt \cU_\Sigma] \simeq k[\cU_\Sigma] \otimes k[\GL_r] = k[Y_{i,j}^\Sigma]\otimes k[Y_{i_t, j}]\left[\frac{1}{\det(Y_{i_t,j})}\right]
\]
where $Y_{i,j}^\Sigma$ are as defined in \eqref{eq:YSigma}.    
Using the identification of $k[\wt \cU_\Sigma]$ as $k[\cU_\Sigma] \otimes k[\GL_r]$, we can write 
\[\begin{pmatrix}Y_{1,1}  & \ldots   & Y_{1, r}\\
\vdots & \ddots  & \vdots\\
&&\\
\vdots & \ddots  & \vdots\\
Y_{n,1}  &\ldots   &Y_{n, r} \\
\end{pmatrix}  = \begin{pmatrix}Y^\Sigma_{1,1}  & \ldots & Y^\Sigma_{1, r}\\
\vdots & \ddots  & \vdots\\
&&\\
\vdots & \ddots  & \vdots\\
Y^\Sigma_{n,1}  &\ldots   &Y^\Sigma_{n, r} \\ 
\end{pmatrix} \otimes \begin{pmatrix}Y_{i_1,1}  & \ldots   & Y_{i_1, r}\\
\vdots & \ddots  & \vdots\\
Y_{i_r,1}  &\ldots   &Y_{i_r, r} \\ 
\end{pmatrix}^{-1}
\]
Hence, we can decompose the operator $\wt \Theta\downarrow_{\wt \cU_\Sigma}$ on $M \otimes k[\wt \cU_\Sigma] \simeq M \otimes k[\cU_\Sigma]\otimes k[\GL_r]$ 
 as follows: 
\[\wt \Theta\downarrow_{\wt \cU_\Sigma} = \Theta^\Sigma \otimes [Y_{i_t,j}]^{-1}, \]
where $\Theta^\Sigma$ is as defined in \eqref{eq:theta}. 
Since the last factor is invertible, we conclude that
\[\Ker \{\wt \Theta\downarrow_{\wt \cU_\Sigma} \}= \Ker \Theta^\Sigma \otimes k[\GL_r] = \eta_{\cU_\Sigma}^{-1}
(\Ker \Theta^\Sigma),\]
where the last equality holds by the triviality of the bundle $\wt \cU_\Sigma \to \cU_\Sigma$ and Remark~\ref{re:triv}. 
Since localization is exact, we have 
$\Ker\{\wt \Theta, M\}\downarrow_{\wt \CU_\Sigma} = 
\Ker\{\wt \Theta\downarrow_{\wt \CU_\Sigma}\}$. 
Hence, 
\begin{equation}
\label{eq:eta}
 \eta_{\cU_\Sigma}(\Ker\{\wt \Theta, M\}\downarrow_{\wt \CU_\Sigma}) = \Ker \Theta^\Sigma. 
 \end{equation}

The Cartesian square
$$
\xymatrix{ 
\wt \CU_\Sigma \ar@{^(->}[r]\ar[d]& \bM^o\ar[d]\\ 
\CU_\Sigma\ar@{^(->}[r]&\Grass_{n,r}\\
}
$$  gives rise to a commutative diagram where  
the vertical arrows are equivalences of categories as 
in Lemma~\ref{descent}
$$
\xymatrix{
\Coh^{\GL_r}(\bM^o) \ar[d]^{\simeq}\ar[r]^-{\Res} & 
\Coh^{\GL_r}(\wt \CU_\Sigma)\ar[d]^{\simeq}  \\
\Coh(\Grass_{n,r})\ar[r]^-{\Res} & \Coh(\CU_\Sigma)
}
$$
Therefore, $ \cR(\Ker\{\wt \Theta, M\})\downarrow_{\CU_\Sigma} = 
\eta_{U_\Sigma}(\Ker\{\wt \Theta, M\}\downarrow_{\wt \CU_\Sigma})$. 
Combining this observation with the equality \eqref{eq:eta},  we conclude 
\begin{equation}
\label{equality}
\cR(\Ker \{\wt \Theta, M\}) \downarrow_{\CU_\Sigma}
=\eta_{\cU_\Sigma}(\Ker\{\wt \Theta, M\}\downarrow_{\wt \CU_\Sigma})
= \Ker \Theta^\Sigma = 
\cKer(M)\downarrow_{\CU_\Sigma},
\end{equation} 
where the last equality holds by the definition of $\cKer(M)$. 
Since this holds for any $r$-subset $\Sigma \subset \{1, \ldots, n\}$, we conclude that 
$\cKer(M) = \cR(\Ker \{\wt \Theta, M\})$.
\end{proof}


\section{Bundles for $\GL_n$-equivariant modules.} 
\label{se:equiv}
For the special class of $\GL_n$-equivariant $kE$--modules (see Definition \ref{eq}), the  
constructions from the previous section 
can be shown to coincide with a well known  construction of algebraic vector bundles arising in 
representation theory of algebraic groups.  This enables us to identify various algebraic
vector bundles on Grassmannians associated to such $\GL_n$-equivariant $kE$--modules.
We give many examples of the applicability of this approach:   Examples \ref{ex:taut1}, \ref{ex:uni},
\ref{ex:twist}, \ref{exsymm}, \ref{dual_symm}, \ref{ex:twist2}, \ref{delta-dual}, and \ref{ex:coh}.

We start by recalling some generalities.
Let $G$ be an algebraic group and $H\subset G$ 
be a closed subgroup.   For any rational 
$H$-module $V$, we consider the flat map of varieties
$$\pi:  G\times^H V \ \to G/H$$
with fiber $V$. We recall the functor (\cite[I.5]{Jan})
\[\xymatrix{\cL:  H\text{-}\mod \ar[r]& \cO_{G/H}\text{-}\mod}
\]
which sends a rational $H$-module $V$ to a quasi--coherent sheaf of
$\cO_{G/H}$--modules  
which is the sheaf of sections of $G\times^H V$. That is, 
for $U \subset G/H$  we have 
\[\cL(V)(U) = \Gamma(U, G\times^H V).\] 

We summarize properties of the functor 
$\cL$ in the following proposition.

 \begin{prop}  \label{prop:el}
Let $G$ be an algebraic group and $H \subset G$ be a closed subgroup.\\
(1) \cite[II.4.1]{Jan}. 
The functor $\cL$ is exact and commutes with tensor products, duals,
symmetric and exterior powers, and  Frobenius twists.\\
(2) \cite[I.5.14]{Jan}. Let $V$ be a rational $G$--module. Then
$\cL(V\downarrow_H)\simeq \cO_{G/H} \otimes V$ is a trivial bundle.  
\end{prop}

We say that an algebraic vector bundle $\cE$ on $G/H$ (i.e., a locally free,
coherent sheaf on $G/H$) is  $G$-equivariant 
if $G$ acts on $\cE$ compatibly with the action of $G$ on the  base $G/H$
(via multiplication on the left).  
That is, for all Zariski open subsets $U \subset G/H$ and each $g \in G$, there is 
an isomorphism 
\begin{equation}
\xymatrix{g^*: \cE(U) \ar[r]& \cE( g^{-1}\cdot U)}
\end{equation}
such that
\[g^* (f s) = g^*(f)g^*(s), \ 
s \in \cE(U), \ f \in \cO_{G/H}(U).
\]
In other words, the algebraic vector bundle $\cE$ on $G/H$ is $G$-equivariant
in the sense of  Definition \ref{de:equiv1}.

\begin{prop}
\label{prop:el2}
(1)  \cite[5.1.8]{CG}. Let $G$ be a linear algebraic group, $H$ be a closed
subgroup of $G$, and $V$ a rational $H$-module.
Then the sheaf of sections of $\pi: G \times^H V \to V$ (a quasi-coherent sheaf of
$\cO_{G/H}$-modules) is  $G$-equivariant.\\
(2) \cite[II.4.1]{Jan}.  Let $G$ be a reductive linear algebraic group, $P
\subset G$ be a parabolic subgroup, 
and $V$ be a rational $P$-module. Then $G\times^P V \to G/P$
is locally trivial for the Zariski topology of $G/P$.  Hence,
 $\cL(V)$ is an algebraic vector bundle on $G/P$.
\end{prop}

The following result complements the preceding recollections.

\begin{prop}
\label{prop:known}
Let $G$ be an algebraic group and $H \subset G$ be a closed subgroup such that
$p: G \to G/H$ is locally trivial with respect to the Zariski topology on $G/H$.  
Consider a $G$-equivariant algebraic vector bundle $\cE$ on $G/H$.  
Then there is an isomorphism $\cL(V) \stackrel{\sim}{\to} \cE$ of $G$-equivariant vector
bundles on $G/H$, 
$$\cL(V)(U) =  \Gamma(U,G\times^H V) \ \stackrel{\sim}{\to}  \ \E(U), \quad U \subset G/H,$$
where $V$ is the fiber of $\cE$ over $eH \in G/H$ provided with
the structure of a rational $H$-module by the restriction of the $G$-action on $\cE$.  
\end{prop}

\begin{proof}
Let $U \subset G/H$ be a Zariski open neighborhood of $eH \in G/H$ such that 
$p_{|U}: p^{-1}(U) \to U$ is isomorphic to the product projection
$U\times V \to U$ and $\cE_{|U} \simeq V\otimes_k \cO_U$ is trivial.  
Choices of trivialization of $p_{U}$ and
$\cE_{|U}$ determine an isomorphism $\phi: \cL(V)(U) \stackrel{\sim}{\to} \cE(U)$.   Some finite
collection of subsets $g_i \cdot U \subset G/H$ is a finite open covering of $G/H$.
For each $g_i$, we define $\phi_i: \cL(g_i\cdot U) \to \cE(g_i\cdot U)$ by
sending $g_i s \in \cL(V)(g_i\cdot U)$ for any $s \in \cL(U)$ to $g_i\phi(s)$; this is well defined, for
each $s^\prime \in  \cL(g_i\cdot U)$ is uniquely of the form $g_i s$ for some $s \in \cL(U)$. 
We readily check that each $\phi_i$ induces an isomorphism on fibers, and that 
$(\phi_i)_{|U_i\cap U_j} = (\phi_j)_{|U_i\cap U_j}.$
\end{proof}

Let $U_0\subset \CV$ be a fixed $r$-dimensional subspace, and let
$\P_0=\Stab(U_0)$. 
With $U_0$ chosen, we may identify $G$ as $\GL_n$ and $\P_0$ as the standard
parabolic
subgroup of type $(r,n-r)$ of $\GL_n$.
We consider the above construction of the functor $\cL$ with $G=\GL(\CV)$ and
$H=\P_0$.
Since $\GL(\CV)/\P_0 \simeq \Grass(r, \CV)$, we get a functor
\[
\xymatrix{\cL: \P_0\text{-}\mod \ \ar[r]&  \  {\rm locally \, free \,
}\cOG\text{-}\mod }
\]
where we denote by $\cOG$ the structure sheaf on $\Grass(r, \CV)$.

\begin{ex}
\label{ex:taut}
We revisit and supplement the examples of Example \ref{ex:first}.

(1)
Let $\gamma_r$ be  the universal subbundle (of $\cOG^{\oplus n}$) of rank $r$ on $\Grass_{n,r}$.
Then 
\begin{equation}
\label{eq:gamma}
\cL(U_0)= \gamma_r.
\end{equation}

(2)
Let $\delta_{n-r}$ be the universal subbundle (of $\cOG^{\oplus n}$) of rank $n-r$ on $\Grass_{n,r}$.
Then 
\begin{equation}
\label{eq:delta}
\cL(W_0) \ = \ \delta_{n-r}, \quad \text{where } \quad W_0 = \Ker\{\CV^\#\to U_0^\#\},
\end{equation}
as can be verified using Proposition~\ref{prop:el} and the short exact sequence 
$$
0 \to \gamma_r \to \cOG^{\oplus n} \to \delta_{n-r}^\vee \to 0.
$$

(3)
By Proposition~\ref{prop:el}, 
$$\cL(\Lambda^r(U_0)) \ = \ \Lambda^r(\gamma_r).$$   
Let $\fp: \Grass(r, \CV)  \to \bP(\Lambda^r (\CV))$ be the Pl\"ucker embedding, 
and let $\cO_{\bP(\Lambda^r (\CV))}(-1)$ 
be the tautalogical line bundle on $\bP(\Lambda^r (\CV))$.  Then by definition  
\[\cOG(-1) = \fp^*(\cO_{\bP(\Lambda^r (\CV))}(-1)).\] 
The fiber of $\cO_{\bP(\Lambda^r (\CV))}(-1)$  over a point 
$v_1\wedge \ldots \wedge v_r \in \bP^r(\Lambda^r(\CV))$ equals $k(v_1\wedge \ldots \wedge
v_r)$. Pulling back via $\fp$, we get that the fiber of $\cOG(-1)$  over the $r$-plane 
$U = kv_1+\cdots+kv_r = \fp^{-1}(v_1\wedge \ldots \wedge v_r)$ equals $\Lambda^r(U)$. 
Thus,
\begin{equation}
\label{eq:lambda}
\cL(\Lambda^r(U_0)) = \Lambda^r(\gamma_r) \simeq \cOG(-1).
\end{equation}
\end{ex}

\vskip .1in

\begin{prop} 
\label{prop:equiv}
Let $M$ be a $\GL_n$--equivariant $kE$--module. Then $\cIm^\ell(M)$, $\cKer^\ell(M)$ 
are $\GL_n$-equivariant $\cOG$-modules for any $\ell$, $1\leq  \ell\leq (p-1)r$. 
\end{prop}

\begin{proof} 
We first observe that Proposition~\ref{prop:homog} implies that
  $\Soc^{\ell}_{U_0}(M)$ and  $\Rad_{U_0}^\ell(M)$ are stable 
 under the action of the standard parabolic subgroup $\P_0 \subset \GL(\CV)$ on $M$
 for any $\ell$, $1\leq  \ell\leq (p-1)r$. 
  
We consider only $\cIm$; verification of the proposition for $\cIm^\ell(M)$, $\cKer^\ell(M)$  with 
$\ell$, $1\leq  \ell\leq (p-1)r$ is similar.  

Let $\bM = M_{n,r}$ be the affine variety of $n \times r$-matrices.  We identify
\[ 
\bM = \CV^{\oplus r} 
\]
as $k$-linear space and note that both $\GL_n$ and $\GL_r$ act on $\bM$: $\GL_n$ via 
multiplcation on the left and $\GL_r$ via multiplication on the right.
Moreover, these actions obviously commute.  Hence, the coordinate ring $k[\bM]$ is a 
$(\GL_n, \GL_r)$-bimodule. 

Recall the $\GL_r$-invariant submodule $\Im\{\wt \Theta, M\}$ of $M\otimes k[\bM]$  
defined in \eqref{def:im}. The $\GL_r$-action on $M\otimes k[\bM]$  is given via the 
trivial action on $M$ and the action on $k[\bM]$ induced by multiplication on $\bM$ 
on the right. There is also a $\GL_n$-action on $M\otimes k[\bM]$ which is diagonal: as 
given on the $\GL_n$-equivariant module $M$ and via the left multiplication on $\bM$. 
We first show that $\Im\{\wt \Theta, M\}$ is a $\GL_n$-invariant submodule of 
$M \otimes k[\bM]$ (and, hence, a $(\GL_n, \GL_r)$-submodule). 

Recall  $\wt \Theta = [\wt \theta_1, \ldots, \wt \theta_r]: \wt M \to (\wt M)^{\oplus r}$  
where for each $j$, $ 1 \leq j \leq r$, 
\[
\wt \theta_j (m \otimes f)  = \sum_i x_im \otimes Y_{i,j}f.
\]
Fix an element $g \in \GL_n$.     We proceed to compute the effect of the action of $g$ on 
$\wt \theta_j (m \otimes f)$.

Let $(y_{ij})_{1\leq i\leq n, 1 \leq j \leq r}$ be linear generators of $\bM = \CV^{\oplus r}$ 
chosen in such a way  that $y_{ij}$ is simply the generator $x_i$  of $\CV$ put in the $j^{\rm th}$ 
column.  Suppose the action of $g$ on $\CV$ with respect to the fixed basis $\{x_1, \ldots, x_n\}$ 
is given by a matrix $A = (a_{st})$. The action of $g$ on $y_{ij}$ is  then given by $gy_{ij} = 
\sum_\ell a_{\ell i}y_{\ell j}$, the same action on each factor $\CV$ in $\bM$.   
 
We identify the coordinate algebra $k[\bM]$ as $S^*(\bM^\#) \simeq k[Y_{i,j}]$ with the coordinate 
functions $Y_{i,j}$  defined as the linear duals of $y_{i,j}$.   For $f \in k[\bM]$, 
we have $g \circ f(-) = f(g^{-1} -)$.  Consequently,   the action of $g$ on $\bM^\#$ with respect to 
the basis $\{Y_{i,j} \}_{1\leq i\leq n, 1 \leq j \leq r}$  is given by multiplication 
on the right by $A^{-1}$. 
We compute
\[g(\sum_i x_i \otimes Y_{i,j}) = g([x_1, \ldots, x_n] \otimes [Y_{1,j}, \ldots, Y_{n,j}]^T) = \]
\[g([x_1, \ldots, x_n]) \otimes g([Y_{1,j}, \ldots, Y_{n,j}]^T) = [x_1, \ldots, x_n]
\cdot A^T \otimes ([Y_{1,j}, \ldots, Y_{n,j}] \cdot A^{-1})^T = \]
\vspace{0.08in}
\[  [x_1, \ldots, x_n]\cdot A^T \otimes (A^T)^{-1}\cdot [Y_{1,j}, \ldots, Y_{n,j}]^T = 
\sum_i x_i \otimes Y_{i,j}.
\]
Hence,
\[ g(\wt \theta_j(m \otimes f)) = g(\sum_i x_i \otimes Y_{i,j})g(m \otimes f) = 
(\sum_i x_i \otimes Y_{i,j})(gm\otimes gf) = \wt \theta_j(gm\otimes gf).\] 
With the given $GL_n$-actions on $k[\bM]$ and $M$,  we have $gm \otimes gf \in M \otimes k[\bM]$. 
Hence,  
\[
g(\wt \theta_j(m\otimes f)) = \wt \theta_j(gm\otimes gf) \in \Im \wt \theta_j. \] 
Since this holds for all $j$, we conclude that $\Im\{\wt \Theta, M\} = \sum\limits_{j=1}^r \Im \theta_j \subset M \otimes k[\bM] $ 
is  invariant under the $\GL_n$-action.
Hence, $\Im(\wt \Theta, M)$ determines a $(\GL_n, \GL_r)$-equivariant sheaf on $\bM$. Moreover, 
since $\bM^0 \subset \bM$ is a $(\GL_n, \GL_r)$-stable subvariety,  the restriction of 
$\Im\{\wt \Theta, M\}$  to $\bM^0$ is a $(\GL_n, \GL_r)$-equivariant sheaf on $\bM^0$.  
Since the actions of $\GL_n$ and $\GL_r$ commute,  the equivalence
\[
\Coh^{\GL_r}(\bM^0) \simeq \Coh(\Grass_{n,r}) 
\]
of Lemma~\ref{descent}
restricts to an equivalence of $\GL_n$-equivariant sheaves. Consequently, 
$\cIm(M) \ = \ \cR(\Im(\wt \Theta, M))$ (by Theorem \ref{pr:coincide}) is a 
 $\GL_n$-equivariant $\cOG$--module. 
 \end{proof}

The following theorem enables us to indentify kernel and image bundles as in
\eqref{se:local} with bundles obtained  via the functor $\cL$. 

\begin{thm}
\label{thm:constr}
Let $M$ be a $\GL_n$-equivariant $kE$-module, and let $U_0=kx_0+\cdots + kx_r \subset \CV$. 
Then for any $\ell$,  $1\leq \ell \leq r(p-1)$,  we have an isomorphism of $\GL_n$-equivariant algebraic vector bundles on $\Grass_{n,r} = \Grass(r,\CV)$
\[
\cKer^\ell(M) \simeq \cL(\Soc^\ell_{U_0}(M)), \quad \cIm^\ell (M) \simeq \cL(\Rad^\ell_{U_0}(M)).
\]
\end{thm}

\begin{proof}
By Proposition 
~\ref{prop:equiv}, $\cKer^\ell(M)$ is a $\GL_n$-equivariant vector bundle on $\Grass(r,\CV)$.   
The fiber of $\cKer^\ell(M)$ above the base point of 
$\Grass(r,\CV)$ equals 
$\Soc^\ell_{U_0}(M)$.   We now apply Proposition ~\ref{prop:known} to
conclude that $\cKer^\ell(M) \simeq \cL(\Soc^\ell_{U_0}(M))$.  

The proof that $\cIm^\ell (M) \simeq \cL(\Rad^\ell_{U_0}(M))$ is strictly analogous.
\end{proof}

In the following examples, we show how to realize various  ``standard" bundles on 
$\Grass(r, \CV)$ as kernel and image bundles associated to $\GL_n$-equivariant $kE$-modules.
For convenience, we fix a basis $\{x_1, \ldots, x_n\}$ 
of $\CV$ and choose $U_0$ to be the subspace generated by $\{x_1, \ldots, x_r\}$.  
As before, the action 
of $\GL_n\simeq \GL(\CV)$ on $kE$ is given via the identification $kE \simeq S^*(\CV)/\langle v^p, 
v \in \CV \rangle \simeq k[x_1, \ldots, x_n]/(x^p_1, \ldots, x_n^p)$.
 
\begin{ex}
[{\it Universal subbundle of rank $r$}]
\label{ex:taut1} Let $M = kE/\Rad^2 (kE)$.    
As a $\P_0$-module, $\Rad_{U_0}(M) \simeq U_0$.  Hence, $\cIm(M) \simeq \gamma_r$
by Example \ref{ex:taut}(1).
\end{ex}

\vskip .1in

\begin{ex}[{\it Universal subbundle of rank $n-r$}]
\label{ex:uni}
Let $M=\Rad^{n-1}(\Lambda^*(\CV))$.  Then 
\[\Soc_{U_0}(M) \simeq (\sum\limits_{j=r+1}^n kx_1\wedge\ldots\wedge x_{j-1}\wedge x_{j+1}\wedge \ldots \wedge x_n) \oplus \Lambda^n(\CV)
\]
 as a $\P_0$--module.  Moreover, the second direct summand is a $\GL_n$--module. The first direct summand 
 can be naturally identified with the $\P_0$-module 
\[W_0 = \Ker\{\CV^\#\to U_0^\#\}
\] 
as in Example~\ref{ex:taut}(2).
We get
\[\cKer(M) =  \cL(W_0) \oplus \cL(\Lambda^n(\CV)) = \delta_{n-r} \oplus \cOG.
\]
It is straightforward to see that $\cIm(M)$ is a trivial bundle of rank one. 
Hence, 
\[
\cKer(M)/\cIm(M) \simeq \delta_{n-r}.
\]
We also note that we have an isomorphism of $kE$--modules: 
$\Rad^{n-1}(\Lambda^*(\CV)) \simeq \left(kE/\Rad^2(kE)\right)^\#$. Hence, we have also justified 	
the second part of Example~\ref{ex:first}.
\end{ex}

\vskip .1in

The previous two examples are connected by a certain ``duality"  which we now state formally. 
As before, we fix the basis $\{x_1, \ldots, x_n\}$ of $\CV$. 
We give $kE \simeq k[x_1, \ldots, x_n]/(x_1^p, \ldots, x_n^p)$ 
the Hopf algebra structure of the truncated polynomial algebra. 
That is, the elements $x_i$ are primitive 
with respect to the coproduct, and the antipode sends $x_i$ to $-x_i$.  
In particular, $\CV \subset \Rad(kE)$ 
is stable under the antipode.  We emphasize that the $kE$-module structure of the 
dual $M^\#$ of  a $kE$--module $M$ utilizes this Hopf algebra structure.

\begin{prop}
\label{duality} 
Let $M$ be a $\GL_n$-equivariant $kE \simeq k[x_1, \ldots, x_n]/(x_1^p, \ldots, x_n^p)$--module. 
Then $M^\#$ is also a $\GL_n$-equivariant $kE$-module (with the standard $\GL_n$-action 
on the dual) and for any $j$, $1 \leq j \leq p-1$, we have 
a short exact sequence of algebraic vector bundles on $\Grass(r,\CV)$:
\[ 
\xymatrix{0\ar[r]&\cKer^j(M^\#) \ar[r]& M^\#\otimes \cOG \ar[r]& \cIm^j(M)^\vee\ar[r]&0.}
\]
\end{prop}
\begin{proof}
Let $U_0\subset \CV$ be the $r$-plane spanned by $\{x_1, \ldots, x_r\}$. 
Proposition~\ref{rad-dual} implies that the following sequence of  
$P_0=\Stab(U_0)$--modules
\begin{equation}
\label{eq:dual}
\xymatrix{0\ar[r]&\Soc_{U_0}^j(M^\#) \ar[r]& M^\# \ar[r]& \Rad_{U_0}^j(M)^\#\ar[r]&0}
\end{equation}
 is exact.
Applying the functor $\cL$ to the short exact sequence \eqref{eq:dual}, using the properties of $\cL$ 
given in Proposition~\ref{prop:el}, and appealing to Theorem~\ref{thm:constr}, 
we conclude the desired short  exact sequence of bundles. 
\end{proof}

\begin{remark}  Let $M=kE/\Rad^2(kE)$ as in Examples~\ref{ex:first}(1) and \ref{ex:taut1}. 
Then the short exact sequence of Proposition~\ref{duality} (with $j = 1$) takes the form
\[\xymatrix{ 0 \ar[r]& \delta_{n-r}  \oplus \cOG \ar[r]& \cOG^{\oplus n+1} \ar[r]& \gamma_r^\vee \ar[r]& 0 }.
\]
\end{remark}

\begin{ex}[{\it The Serre twist bundle $\cOG(-1)$}]
\label{ex:twist}
Let 
\[
M=\Rad^r (\Lambda^*(\CV))/\Rad^{r+2}(\Lambda^*(\CV)).
\]
 Then $\Soc_{U_0}(M) = \Lambda^r(U_0) 
\oplus \Rad^{r+1}(\Lambda^*(\CV))$ as a $\P_0$--module. 
Hence, 
\[ \cKer(M) \simeq \cL(\Lambda^r(U_0)) \oplus \cL(\Rad^{r+1}(\Lambda^*(\CV))). 
\] 
Since the structure of $\P_0$ on $\Rad^{r+1}(\Lambda^*(\CV))$ is the restriction of 
$\GL_n$-structure, Prop.~\ref{prop:el}.2 implies that $\cL(\Rad^{r+1}(\Lambda^*(\CV)))$ 
is a trivial bundle. Hence, Prop.~\ref{prop:el} and Example~\ref{ex:taut} imply 
that
\[
\cKer(M) \simeq \Lambda^r(\gamma_r) \oplus (\cOG\otimes \Lambda^{r+1}(\CV)) \
\simeq \ \cOG(-1) \oplus \cOG^{n \choose {r+1}}.
\]
\end{ex} 

\vskip .1in

\begin{ex}
\label{exsymm}
[{\it Symmetric powers}]  Let $j$ be a positive integer, $j \leq p-1$, and let 
\[M=S^*(\CV)/S^{*\geq j+1}(\CV).\]  
Then $\Rad^j_{U_0}(M)$ is isomorphic to  $S^j(U_0)$ as a $\P_0$--module. 
Hence, by Prop.~\ref{prop:el} and Example~\ref{ex:taut},
\[\cIm^j(M) =   S^j(\gamma_r).\]

\vspace{0.1in} 
More generally, let $M= S^{*\geq i}(\CV)/S^{*\geq i+j+1}(\CV)$. Consider 
the multiplication map
\[
\mu: S^j(U_0) \otimes S^i(\CV) \to S^{i+j}(\CV). 
\]
and the corresponding exact sequence of $\P_0$--modules
\[\xymatrix{
0 \ar[r] & \Ker \mu \ar[r] & S^j(U_0) \otimes S^i(\CV) \ar[r]^-\mu & S^{i+j}(\CV) \ar[r] & \Coker \mu \ar[r] &0.
}
\]
The image of the multiplication map $\mu$ is spanned by all monomials divisible by a monomial 
in $x_1, \ldots, x_r$ of degree $j$. 
Hence, $\Rad^j_{U_0}(M) \simeq \Im \mu$.    Applying the functor $\cL$ to the exact sequence 
above, we conclude that
\[\cIm^j(M) \simeq \Im \{ \cL(\mu)\}\]
where  $\cL(\mu): S^j(\gamma_r) \otimes  S^i(\CV) \subset S^j(\CV) 
\otimes S^i(\CV) \otimes \cOG \to S^{i+j}(\CV) \otimes \cOG$
is the multiplication map.

\vspace{0.1in} 
We now specialize to the case $j=1$. Then, 
\[M = S^{*\geq i}(\CV)/S^{*\geq i+2}(\CV).\]
In this case, the image of the multiplication map $\mu: U_0 \otimes S^i(\CV) \to S^{i+1}(\CV)$ 
is spanned by 
all monomials divisible by one of the variables $x_1, \ldots, x_r$.  Therefore, we have a 
short exact sequence of $P_0$-modules
\begin{equation}
\label{cokermu}
\xymatrix{ 0 \ar[r] & \Rad(M)=\Im \mu \ar[r]& S^{i+1}(\CV)  \ar[r] & S^{i+1}(\CV/U_0) \ar[r] & 0 }.
\end{equation}
In the notation of Example~\ref{ex:taut}(2), $\CV/U_0 \simeq W_0^\#$. Hence, Proposition
~\ref{prop:el}  and Example~\ref{ex:taut}(2) imply that 
$$\cL(S^{i+1}(\CV/U_0)) \ \simeq \ \cL(S^{i+1}(W_0^\#)) \ = \ S^{i+1}(\delta_{n-r}^\vee).$$ 
Applying the exact functor $\cL$ to (\ref{cokermu}), we conclude that 
$\cIm(M)$ fits into the following  short exact sequence of vector bundles
\begin{equation}
\label{eq:symm} 
\xymatrix{ 0 \ar[r] & \cIm(M) \ar[r]& S^{i+1}(\CV)\otimes \cOG  \ar[r] & S^{i+1}(\delta_{n-r}^\vee) \ar[r] & 0 }.
\end{equation}
\end{ex}

\vskip .1in

\begin{ex}\label{dual_symm} 
Let $i$ be a positive integer such that $i\leq p-1$, and let 
\[M=\frac{\Rad^{n(p-1)- i - 1 }(kE)}{\Rad^{n(p-1) - i + 1}(kE)}.\] 
Note that as a $kE$-module, 
\[M^\# \simeq \Rad^{i}(kE)/\Rad^{i+2}(kE).\]
Moreover, the restriction on $i$ implies that 
\[
\Rad^i(kE)/\Rad^{i+2}(kE) \simeq S^{*\geq i}(\CV)/S^{*\geq i+2}(\CV).
\]
Applying Proposition~\ref{duality}, we get a short exact sequence of bundles
\[
\xymatrix{0\ar[r]&\cKer(M) \ar[r]& M\otimes \cOG \ar[r]& \cIm(M^\#)^\vee\ar[r]&0.}
\]
Since the bottom radical layer of $M$ is in the socle for any $U\subset \CV$, 
the kernel bundle $\cKer(M)$ has a trivial subbundle $\Rad(M) \otimes \cOG \simeq 
S^{i}(\CV^\#) \otimes \cOG$ 
as a direct summand. Hence, we can rewrite the exact sequence above as 
\[
\xymatrix{0\ar[r]&\frac{\cKer(M)}{\Rad(M) \otimes \cOG} \oplus 
( \Rad(M) \otimes \cOG) \ar[r]& (S^{i}(\CV^\#) \oplus S^{i+1}(\CV^\#)) \otimes \cOG  }
\]
\[\xymatrix{&\ar[r]&\cIm(M^\#)^\vee\ar[r]&0.}
\]
Discarding the direct summand $\Rad(M) \otimes \cOG$ which splits off, we get 
\[
\xymatrix{0\ar[r]&\frac{\cKer(M)}{\Rad(M) \otimes \cOG} 
\ar[r]& S^{i+1}(\CV^\#) \otimes \cOG \ar[r]& \cIm(M^\#)^\vee\ar[r]&0.}
\]
Dualizing, we further get 
\[
\xymatrix{ 0 \ar[r] & \cIm(M^\#) \ar[r]& S^{i+1}(\CV)\otimes \cOG  \ar[r] & 
\left(\frac{\cKer(M)}{\Rad(M) \otimes \cOG}\right)^\vee  \ar[r] & 0 }.
\]
It follows from the construction that the embedding $\cIm(M^\#) 
\hookrightarrow S^{i+1}(\CV)\otimes \cOG$ in this short exact sequence 
coincides with the corresponding map in 
\eqref{eq:symm} which was induced by the multiplication map 
$\mu: \gamma_r \otimes S^i(\CV) \to S^{i+1}(\CV) \otimes \cOG$.  
Hence, 
\[\frac{\cKer(M)}{\cIm(M)} = \frac{\cKer(M)}{\Rad(M) \otimes \cOG} \simeq S^i(\delta_{n-r}).\]
\end{ex}

\vskip .1in

\begin{ex}[{\it The Serre twist bundle $\cOG(1-p)$}]
\label{ex:twist2}  Let  
\[M = \Rad^{r(p-1)}(kE)/\Rad^{r(p-1)+2} (kE).
\]
 Then 
\[
\Soc_{U_0}(M) = kx_1^{p-1}\ldots x_r^{p-1} \oplus \Rad(M)
\]
We have an obvious isomorphism of one-dimensional $\P_0$-modules
\[\underbrace{\Lambda^r(U_0)\otimes \ldots \otimes \Lambda^r(U_0)}_{p-1} \simeq kx_1^{p-1}\ldots x_r^{p-1} 
\]
given by sending $x_1\wedge\ldots\wedge x_r \otimes \ldots \otimes x_1\wedge\ldots\wedge x_r$ to $x_1^{p-1}\ldots x_r^{p-1}$. 
Hence, 
\[\cKer(M) \simeq \cL((\Lambda^r(U_0)^{\otimes p-1}) \oplus \cL(\Rad (M)) \simeq   \cOG(1-p) \oplus  \Rad (M) \otimes \cOG 
\] 
where the last equality follows from Example~\ref{ex:twist} and Proposition ~\ref{prop:el}.
\end{ex}
\vskip .1in

\begin{ex}[{\it $\delta_{n-r}^\vee$ via cokernel}]
\label{delta-dual}  Let $\cCoker(M) \stackrel{def}{=} (M \otimes \cOG)/\cIm(M)$.   
Let $M=kE/\Rad^2(kE)$. 
The exactness of $\cL$ together with Example~\ref{ex:taut} imply that
\[
\cCoker(M) \simeq \delta_{n-r}^\vee.
\]
 \end{ex}
 
 \vskip .1in
 
In the following example we study a bundle that comes not from a 
$\GL_n$-equivariant $kE$-module but from  
the cohomology of $E$ considered as a $\GL_n$-module. 
For the coherence of notation, assume that $p>2$. 
Recall that $\HHH^*(E,k)$ has a $\GL_n$-structure and, moreover, we have an 
isomorphism of $\GL_n$--modules
\[
\HHH^*(kE,k) \simeq     \Lambda^*(V^\#)\otimes S^*((V^{(1)}[2])^\#)
\]
as stated in Proposition \ref{idd}.

\begin{ex}   
\label{ex:coh} 
Let $\alpha_0: C=k[t_1, \ldots, t_r]/(t_i^p) \to kE$ be the map defined 
by $\alpha_0(t_i)=x_i $ for $1\leq i\leq r$, and let
\[
\alpha_0^*:\HHH^{2m}(kE,k) \to \HHH^{2m}(C,k)
\] be the induced  map on cohomology for some positive integer $m$.  Reducing modulo nilpotents, we
get a map
\[ 
\alpha_0^*:\HHH^{2m}(kE,k)_{\rm red} \simeq S^{m}((V^{(1)})^\#)  \to \HHH^{2m}(C,k)_{\rm red} \simeq S^m((U_0^{(1)})^\#)
\]
which is induced by $(\alpha_0^{(1)})^\#: (V^{(1)})^\# \to (U_0^{(1)})^\#$ by Proposition~\ref{idd}. 
This implies that 
the kernel of $\alpha_0^*$ is stable under the action of the standard parabolic $\P_0$. Hence, we can apply the functor $\cL$ to 
$\Ker \alpha_0^*$. Since $\cL$ is exact and commutes with Frobenius twist we obtain a short exact sequence of  bundles
\[
\xymatrix{0 \ar[r]& \cL(\Ker \alpha_0^*) \ar[r] & \cOG\otimes S^m((V^{(1)})^\#) \ar[r] & S^m(F^*(\gamma_r^\vee))\ar[r] & 0}
\]
where $F:\Grass(r,\CV) \to \Grass(r, \CV)$ is the Frobenius map.
\end{ex}


\section{A construction using the Pl\"ucker embedding}\label{global}  

We present another construction of bundles from modules of constant
$r$-socle rank, one that  applies only to kernel bundles.  This construction provides
``generators" for  graded modules for the coordinate algebra of the 
Grassmannian whose associated coherent sheaf is the kernel bundle of Theorem \ref{pr:coincide}.

We denote the homogeneous coordinate ring of 
$\Grass_{n,r}$ by $\cA$ and identify it with a quotient 
of $k[p_\Sigma]$ via the Pl\"ucker embedding  
$\p:\Grass_{r,n} \to \bP^{{n \choose r}-1}$.  As before, $\cOG$ denotes
the structure sheaf of $\Grass_{n,r}$. Since $\cA$ is 
generated in degree one, we have an equivalence of categories 
(the Serre correspondence) 
\begin{equation}
\label{Serre}
\Coh(\Grass_{n,r}) \simeq \frac{\text{Fin. gen. graded } \cA-\mod}{\text{fin. dim. graded }   \cA-\mod}  
\end{equation}
between the category of coherent $\cOG$--modules and the quotient 
category of finitely generated graded $\cA$--modules modulo the finite dimensional 
graded $\cA$--modules. The equivalence is given explicitly by sending an $\cOG$--
module $\cF$  to 
$\bigoplus\limits_{i \in \Z} \Gamma(\Grass_{n,r}, \cF(i))$ 
(see \cite[II.5]{Har}). 

Starting with a module of constant r-socle  rank, we construct a graded 
$\cA$--module $\Ker\{\Theta_\cA, M\}$
which is in the equivalence class of the kernel bundle  $\cKer(M)$ 
via the Serre correspondence (Theorem~\ref{pr:coincide2}).  
We then develop an algorithm that can be used to construct a collection 
of generators $w_1, \dots, w_t$, in degrees  $d_1, \dots, d_t$, of the 
graded module $\Ker\{\Theta_\cA, M\}$,  up to a finite dimensional quotient.  
Applying the Serre  correspondence again, we obtain
a surjective map of vector bundles 
\[
\xymatrix{\bigoplus\limits_{i=1}^{t} \cOG(-d_i) \ar[r]&  \cKer(M)}.
\]

\begin{defn} \label{def-theta}
Let $M$ be a $kE \simeq k[x_1, \ldots, x_r]/(x_1^p, \ldots, x^p_n)$--module. We define the  map
\[\Theta_\cA: \quad M \otimes \cA   \quad \longrightarrow  \quad
(M \otimes\cA )^{\binom{n}{r-1}}
\]
by components $\Theta_\cA = \{ \vartheta_W\}$ where the index is over the
subsets $W \subset \{1 \dots n\}$ having $r-1$ elements. For any such $W$,
and any $m \in M$, let
\[
\vartheta_W(m \otimes 1) \quad = \quad\sum_{i \notin W} (-1)^{u(W,i)} \ x_im \otimes \fp_{W \cup \{i\}}
\]where $u(W,i) = \#\{j \in W \ \vert \ j < i\}$. 
\end{defn}

Since the operator $\Theta_\cA$  is 
graded of degree one (with respect to the standard grading  of the 
homogeneous coordinate algebra $k[p_\Sigma]$ of
$\bP^{{n \choose r}-1}$ where the Pl\"ucker coordinates $\fp_\Sigma$ have 
degree $1$), the kernel of $\Theta_\cA$, denoted $\Ker\{\Theta_\cA, M\}$, 
is a graded $\cA$--module. 

\begin{thm}
\label{pr:coincide2} For any finite--dimensional $kE$--module $M$, 
the graded $\cA$-module $\Ker\{\Theta_\cA, M\}$ corresponds to 
the coherent sheaf $\cKer(M)$ as defined in \eqref{localdef} via 
the equivalence of categories \eqref{Serre}.
\end{thm}

\begin{proof} Let $\CU_\Sigma \subset \Grass_{n,r}$ be a principal open subset
indexed by some subset $\Sigma \subset \{1,\ldots,n\}$ of cardinality $r$.   Then
$$k[\CU_\Sigma] \ = \ (\cA[1/\fp_\Sigma])_0, \quad 
\Ker\{\Theta_\cA, M\}_{\CU_\Sigma} = (\Ker\{\Theta_\cA, M\} \otimes \cA[1/\fp_\Sigma])_0,$$
where  $\Ker(M)_{\CU_\Sigma} = \cKer(M){\downarrow_{\CU_\Sigma}}$.

We show that for any $r$-subset $\Sigma \subset \{1, \ldots, n\}$, 
\[ \Ker\{\Theta_\cA, M\}_{\CU_\Sigma} =  \Ker(M)_{\CU_\Sigma} \]
as submodules of $M \otimes k[\CU_\Sigma]$  which is sufficient to prove the theorem. 

Let $\CI_{r-1}$ be the set of all subsets $W$ of $\{1, \ldots, n\}$ of cardinality $r-1$. 
 Recall  that  $\Ker(M)_{\CU_\Sigma}$ is given 
as the kernel of the operator
$$
[\theta^\Sigma_1, \ldots, \theta^\Sigma_r] = [x_1, \ldots, x_n] \otimes \begin{pmatrix}Y^\Sigma_{1,1}  & \ldots & Y^\Sigma_{1, r}\\
\vdots & \ddots  & \vdots\\
&&\\
\vdots & \ddots  & \vdots\\
Y^\Sigma_{n,1}  &\ldots   &Y^\Sigma_{n, r} \\ 
\end{pmatrix}: M \otimes k[\CU_\Sigma] \to (M \otimes k[\CU_\Sigma])^{\oplus r}.
$$
On the other hand, the operator 
$\Theta_\cA: M \otimes k[\CU_\Sigma]  \to (M \otimes k[\CU_\Sigma])^{\binom{n}{r-1}}$ 
is given by localizing $[\vartheta_W]_{W \in \CI_{r-1}}$ 
as defined in (\ref{def-theta}) to $k[\CU_\Sigma]$. We show 
that the operators $[\theta^\Sigma_1, \ldots, \theta^\Sigma_r]$ and $[\vartheta_W]_{W \in \CI_{r-1}}$ 
are related by multiplication by a matrix $B$ 
(of size  ${n \choose {r-1}} \times r$) which 
does not change the kernel.

To simplify notation, assume  that 
$\Sigma = \{1, \ldots, r \}$. 
We define the matrix $B$ with columns indexed  by 
subsets $W=\{i_1, \ldots, i_{r-1} \}$  of $\{ 1, \ldots, n \}$ 
and rows  indexed by $j$, $ 1 \leq j \leq r$. Let
$B_{W,j}$ be the
$(-1)^j$ times the determinant
of the $(r-1) \times (r-1)$ submatrix obtained 
from $[Y^\Sigma_{i,j}]$ by taking the rows indexed by
$W$ and deleting the $j^{th}$ column. That is,
\[
B_{j,W} = (-1)^j \ \Det \left[\begin{array}{cccccc} Y^\Sigma_{i_1,1} 
& \ldots & Y^\Sigma_{i_1,j-1}& Y^\Sigma_{i_1,j+1}& \ldots &Y^\Sigma_{i_1, r} \\[0.5em]
Y^\Sigma_{i_2,1} & \ldots & Y^\Sigma_{i_2,j-1}& Y^\Sigma_{i_2,j+1}& \ldots &Y^\Sigma_{i_2, r}\\[0.5em]
\vdots &\ddots& \vdots &\vdots &\ddots&\vdots\\[0.5em]
\vdots &\ddots& \vdots &\vdots &\ddots&\vdots\\[0.5em]
Y^\Sigma_{i_{r-1},1} & \ldots & Y^\Sigma_{\tiny i_{r-1},j-1}& Y^\Sigma_{i_{r-1},j+1}& \ldots &Y^\Sigma_{i_{r-1}, r}
\end{array}\right]\]

We pick a special order on the subsets $W \in\CI_{r-1}$,  so 
that the first $r$ columns of $B$ are indexed by 
$\{1, \ldots, r-1\}, \{1, \ldots, r-2, r\}, \ldots, \{1, 3, \ldots, r\}, \{2, \ldots, r\}$. 
With this assumption, the first $r$ columns of $B$ form 
an identity matrix.  Indeed, since the  first $r$ rows 
of $[Y^\Sigma_{i,j}]$ form an identity matrix, we have
$$
B_{j, \{1, \ldots, j-1, j+1, \ldots r \}} = 1 \quad \text{ and }  
B_{j^\prime, \{1, \ldots, j-1, j+1, \ldots r \}} = 0 \text{ for } j^\prime \not = j.
$$
We rewrite
\begin{equation}B = \begin{pmatrix}
 I_r & | &  B^\prime
\end{pmatrix}
\end{equation}

Next we compute the $n \times {n \choose {r-1}}$ -- matrix 
$[Y_{i,j}^\Sigma] \cdot B$.  We have that
$$
([Y_{i,j}^\Sigma] \cdot B)_{i,W} = Y^\Sigma_{i,1} B_{1,W} +  Y^\Sigma_{i,2} B_{2,W}
+  Y^\Sigma_{i,r}B_{r,W} 
$$
which is the determinant of the matrix
$$
\begin{pmatrix} Y^\Sigma_{i,1} & Y^\Sigma_{i,2} & \ldots & Y^\Sigma_{i,r}\\[0.5em]  
Y^\Sigma_{i_1,1} & Y^\Sigma_{i_1,2} & \ldots & Y^\Sigma_{i_1,r}\\[0.5em]
\vdots& & \ddots &\vdots \\[0.5em]
Y^\Sigma_{i_{r-1},1} & Y^\Sigma_{i_{r-1},2} & \ldots & Y^\Sigma_{i_{r-1},r}\\[0.5em]
\end{pmatrix},
$$
where $W = \{i_1, \ldots, i_{r-1}\}$.
If $i$ is in $W$ then the matrix has two 
identical columns and its determinant
is zero. If $i$ is not in $W$ then the determinant is precisely
$(-1)^{u(W,i)} \p_{W \cup \{i\}}$. That is, 
the only difference between the
above matrix and the matrix whose 
determinant is $\p_{W \cup \{i\}}$ is
that the first row must be moved to the 
proper position so that the elements
$i, i_1, \dots, i_{r-1}$ are rearranged to 
be consecutive. This requires $u(W,i)$ moves.
We conclude that the matrix $B$ has an 
entry $(-1)^{u(W,i)} \p_{W \cup \{i\}}$ at the place $\{W, i\}$
(where we assume for convenience that 
$\p_{W \cup \{i\}} =0$  if $ i \in W$). Hence, 
\[
([x_i]\cdot[Y_{i,j}^\Sigma]\cdot B)_W  = \vartheta _W.
\]
The formula  $[\theta^\Sigma_1, \ldots, \theta^\Sigma_r] = [x_i]\cdot[Y_{i,j}^\Sigma]$
 now implies the equality
\begin{equation}
\label{transition}
[\theta^\Sigma_1, \ldots, \theta^\Sigma_r] \cdot B = 
[\vartheta_W]_{W \in \CI_{r-1}}.
\end{equation}
Since $B = [I \, \, | \,\,B^\prime]$ has maximal rank, we conclude that
$$
\Ker\{[\theta^\Sigma_1, \ldots, \theta^\Sigma_r]: 
M \otimes k[\CU_\Sigma]\to (M \otimes k[\CU_\Sigma])^{\oplus r}\} =
$$
$$ \Ker\{[\vartheta_W]: M \otimes k[\CU_\Sigma]\to (M \otimes k[\CU_\Sigma])^{n \choose r-1} \}.
$$ 
Hence, $\Ker\{\Theta_\cA, M\}_{\CU_\Sigma} =  \Ker(M)_{\CU_\Sigma}$. 
\end{proof}

Combining Theorems ~\ref{pr:coincide2} and 
\ref{th:bundle}, we get the following Corollary.

\begin{cor} \label{glob=bundle}
Assume that $M$ is a $kE$-module of constant $r$-$\Soc^1$ rank.
Then the Serre correspondent (via the equivalence \eqref{Serre}) of the graded $\cA$-module 
$\Ker\{\Theta_\cA, M\}$ is an algebraic vector bundle on $\Grass_{n,r}$.
\end{cor}

In some of our calculations,  we use the following 
variation of the operator $\Theta_\cA$ given in Definition \ref{def-theta}. 
Let $\cI_{r-1}$ be the set of all subsets of $\{1, \dots, n\}$ having $r-1$ elements.
Let $M$ be a $kE$-module of constant socle type
with the property that $\Rad^2(M) = \{0\}$.  Note that thr assumption $\Rad^2(M)=0$ 
implies that constant $r$-socle type is equivalent to constant  $r$-$\Soc^1$ rank. 

We define the map
\[
\ol\Theta_\cA: \quad  M/\Rad(M)  \otimes \cA \quad \longrightarrow  \quad
( \Rad(M)\otimes \cA)^{\binom{n}{r-1}}
\]
by its components $\ol \Theta_\cA = \{ \ol \vartheta_W\}$ where the index is over 
$W \in \cI_{r-1}$. For any such $W$,
and any $m \in M$, let
\begin{equation}
\label{formula}
\ol \vartheta_W((m+ \Rad(M)) \otimes 1) \quad = \quad
\sum_{i \notin W} (-1)^{u(W,i)} \ x_im \otimes \fp_{W \cup \{i\}}
\end{equation}
where $u(W,i) = \#\{j \in W \ \vert \ j < i\}$.

\begin{cor} \label{bund-2step} Let $M$ be a $kE$-module of constant socle type
with the property that $\Rad^2(M) = \{0\}$.
Then the graded $\cA$--module $\Ker \{ \ol\Theta_\cA, M \}$ corresponds to an 
algebraic vector bundle on
$\Grass_{n,r}$ via the equivalence \eqref{Serre}.
\end{cor} 

\begin{proof}
Because $\Rad^2(M) = \{0\}$, the free $\cA$--module $\Rad(M) \otimes \cA$
is a submodule of $\Ker \{\Theta_\cA, M \}$ with quotient 
$\Ker \{ \ol\Theta_\cA, M \}$. 
\end{proof}

\begin{remark}
\label{how_to_use} For certain $kE$-modules $M$ of constant $r$-socle rank, 
Corollary~\ref{bund-2step} can be used
to determine a  graded $\cA$-submodule of $M\otimes \cA$ with Serre correspondent 
$\cKer(M) \subset M\otimes \cO_{Gr}$. The process goes in two steps. First,  
a set of elements of the  kernel is calculated. 
This can be done using a computer seaching through the degrees. 
That is, we use (\ref{formula}) to calculate
a matrix of the map $\ol\Theta_\cA$ on the degree one
grading of $ M/\Rad(M) \otimes \cA$ to the degree two grading 
of $\Rad(M) \otimes \cA$.  A spanning set of elements of the null space of 
this matrix constitutes part of our set of ``generators".  We continue,
next looking for a spanning set of the null space of our matrix for  
$\ol\Theta_\cA$ on the degree two
grading of $M/\Rad(M) \otimes \cA$.  We proceed to higher and higher 
gradings.

The next step is to verify that we have found sufficiently many elements
in the kernel to generate a graded module with Serre correspondent $\cKer(M)$.
In certain examples, it is possible to show that the elements obtained by
considering gradings less than or equal to a given degree
generate a graded submodule $N \subseteq \Ker\{\ol\Theta_\cA, M\}$ 
with Serre correspondent $\cKer(M)$.   We start with the information that
the Serre correspondent of $N$
should have rank equal to  $d = \dim \Soc_U(M) -\dim \Rad(M)$ (which is 
independent of $r$-plane $U$ since $M$ has constant $r$-socle rank).

Because the module $M$ has constant $r$-socle rank,  for any extension 
$K$ of $k$ and any specialization
$\cA \to K$ at a homogeneous prime ideal of $\cA$, 
the induced inclusion map $N \otimes_{\cA} K \to  
M/\Rad(M) \otimes_\cA  K$ can not have rank more than $d$ by Corollary 
\ref{bund-2step}. If it can be shown that the rank of any such 
specialization is exactly $d$, then we have that $N$ is a graded 
module corresponding to a vector bundle of rank $d$ that is contained in  
$\Ker\{\ol \Theta_\cA, M\}$ which also has rank $d$. 
Consequently, the Serre correspondents of  $N$ and $\Ker\{\ol \Theta_\cA, M\}$ 
are equal.
\end{remark}

We revisit some of the examples of Section ~\ref{se:equiv} to illustrate how this method works. 

\begin{ex}[Universal subbundle $\delta_{n-r}$] \label{dual-image}
Let $M \simeq \Rad^{(p-1)n-1}(kE)$. Then $\Rad^2(M) = \Rad^{(p-1)n+1}(kE) = \{0\}$ and, hence, 
$M$ satisfies the hypothesis of Corollary~\ref{bund-2step}.  Pictorially, we can represent 
$M$ as follows: 
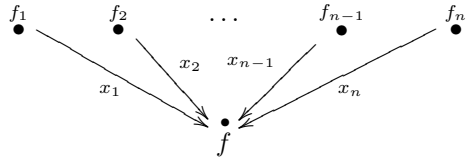
\[\xymatrix{
\stackrel{f_1}{\bu}\ar[drr]_{x_1}& \stackrel{f_2}{\bu}\ar[dr]^{x_2}& \ldots & 
\stackrel{f_{n-1}}{\bu}\ar[dl]_{x_{n-1}}& \stackrel{f_n}{\bu}\ar[dll]^{x_n}\\
&&\stackrel{\bu}{f}&&\\}
\]
It is then evident that $M \simeq \Rad^{n-1}(\Lambda^*(\CV))$. By \eqref{ex:uni}, 
$\cKer(M) / (\Rad(M) \otimes\cOG) \simeq \delta_{n-r}$, the universal subbundle 
of $\cOG^{\oplus n}$ or rank $n-r$.

We proceed to write down explicit generators for the kernel $\Ker\{\ol\Theta_\cA, M \}$ as 
a submodule of $M/\Rad(M) \otimes \cA$. 
Let $\{f, f_1, \ldots, f_n\}$ be linear generators of $M$ as indicated on the diagram 
above. 
  Let $\CI_{r+1}$ be the set of subsets of
$\{1, \dots, n\}$ having exactly $r+1$ elements.
For each $S \in \CI_{r+1}$ let $w_S$ be the element of $M/\Rad(M) \otimes \cA$ given as
\[
w_S \ = \ \sum_{j \in S} \  (-1)^{u(S,j)} f_j \otimes
\fp_{_{S\setminus\{j\}}} 
\]
where $u(S,j) = \#\{i \in S \vert i \leq j\}$.  These $w_S$, all of grading one, generate $\delta_{n-r}$
as we verify in the next proposition.

\begin{prop} \label{ex-delta}
The elements $w_S$ generate a graded $\cA$-module corresponding to the
algebraic vector bundle $\delta_{n-r}$ via \eqref{Serre}.
\end{prop}

\begin{proof} A proof proceeds as follows. We should note that the elements were
generated by computer in special cases, but it is a straightforward
exercise to check that these elements are in the kernel of $\ol\Theta_\cA$.
We leave this exercise to the reader.

The defining equations for the elements $w_S$
can be written as a matrix equation
\[
{\bf w} = {\bf f} \otimes {\bf P} 
\]
where ${\bf w} = [w_{_S}]_{_{S \in \CI_{r+1}}}$, ${\bf P} = (p_{j,{}_S})$ is the
$n \times \binom{n}{r+1}$ matrix with entries $p_{j,{}_S} =
(-1)^{u(S,j)} \fp_{_{S\setminus\{j\}}}$ if $j \in S$ and $p_{j,{}_S} = 0$
otherwise, and  ${\bf f} =[f_1, \ldots, f_n]$.  Because the elements $f_1, \dots,
f_n$ are linearly independent, the dimension of the image depends
entirely on the rank of the matrix ${\bf P}$. As was noted Remark~\ref{how_to_use}, 
at any specialization $\phi: \cA \to
K$, $K$ an extension of $k$, the rank of $\phi({\bf P})$
can not be greater than $n-r$ which is $\dim(\Soc_U(M))-1$ for
any $U$. So the task is to show that the rank of the matrix
${\bf P}$ at any specialization is at least $n-r$.

In any specialization, one of the Pl\"ucker coordinates, call it $\fp_\Sigma$,
must be nonzero. 
Consider the $(n-r)\times (n-r)$ submatrix of $\bf P$ determined by the columns indexed by subsets
$T \in \cI_{r+1}$ that contain a fixed $\Sigma \in \cI_{r}$ and the rows indexed
by all $j$ such that $j \notin \Sigma$. The $(i,T)$ entry in this matrix is
$(-1)^{u(T,j)} \fp_\Sigma$ if $i = j$, and is 0 if $i \neq j$. Consequently,
the determinant of this submatrix is $\pm \fp_\Sigma^{n-r}$ which
is not zero. So we have proved that the elements $w_S$ generate
a locally free graded module whose corresponding bundle is the
kernel bundle $\cKer(M)/ (\Rad(M)  \otimes \cOG)$.
\end{proof} 
\end{ex}
\vspace{0.1in}

\begin{remark}  In the notation of the Example~\ref{ex:uni}, it is not difficult 
to see that the kernel $\Ker\{\ol \Theta_\cA, \Rad^{n-1}(\Lambda^*(\CV))\}$ is generated 
by the element
\[
v \ = \ \sum_{i \notin \Sigma}  \ (-1)^{u(\Sigma,i)} (y_\Sigma \otimes \fp_\Sigma) ,
\]
where $u(\Sigma,i)$ is the number of elements in $\Sigma \in \CI_r$ that are less than $i$,
and the sum is over all subsets of $ \{1, \dots, n\}$ having exactly 
$r$ elements. Here $y_\Sigma = x_{i_1} \wedge \dots \wedge x_{i_r}$
where $\Sigma = \{i_1, \dots, i_r\}$.  Hence, in this case we have  a graded $\cA$-module  
corresponding to the universal bundle $\delta_{n-r}$ generated by only one element. 
\end{remark}

\begin{ex}\label{bundex7}
Set $p = 3, n=4, r=2$ and consider $M = \Rad^4kE/\Rad^6kE$.  By Example~\ref{ex:twist2}, 
$\cKer(M)/(\Rad(M) \otimes \cOG) \simeq \cOG(-2)$.  The following generator of 
the graded module $\Ker\{\ol\Theta_\cA,M\}$ whose associated bundle is $\cKer(M)/(\Rad(M) 
\otimes \cOG)$, was constructed with the aid of the computational algebra package
Magma \cite{BC}.
\begin{align*}
v = & \ x_1^2x_2^2 \otimes \fp_{12}^2         -
x_1^2x_2x_3 \otimes\fp_{12}\fp_{13}  -
x_1^2x_2x_4 \otimes \fp_{12}\fp_{14}  +
x_1x_2^2x_3 \otimes\fp_{12}\fp_{23}   + \\
&x_1x_2^2x_4  \otimes\fp_{12}\fp_{24}  +
x_1^2x_3^2 \otimes\fp_{13}^2     -
x_1^2x_3x_4\otimes\fp_{13}\fp_{14}      -
x_1x_2x_3^2\otimes\fp_{13}\fp_{23}   - \\
&x_1x_2x_3x_4  \otimes\fp_{13}\fp_{24}    +
x_1x_3^2x_4\otimes\fp_{13}\fp_{34}                       +
 x_1^2x_4^2\otimes\fp_{14}^2                     -
x_1x_2x_3x_4\otimes\fp_{14}\fp_{23}     - \\
&x_1x_2x_4^2 \otimes\fp_{14}\fp_{24}                    -
x_1x_3x_4^2\otimes \fp_{14}\fp_{34}                    +
x_2^2x_3^2 \otimes\fp_{23}^2              -
x_2^2x_3x_4\otimes\fp_{23}\fp_{24}                   +  \\
& x_2x_3^2x_4 \otimes\fp_{23}\fp_{34}     +
x_2^2x_4^2 \otimes\fp_{24}^2           -
x_2x_3x_4^2 \otimes\fp_{24}\fp_{34}    +
x_3^2x_4^2 \otimes\fp_{34}^2                            .
\end{align*}
Note that the degree of this generator is 2, which is consistent 
with the fact that the associated bundle is 
$\cOG(-2)$.
\end{ex}

We end this section with nontrivial computation of the graded module of a 
vector bundle of rank 3 over $\Grass(2,\CV)$. It confirms the intuition that
modules become more complicated as the rank and degree increase. The 
generators in this example were calculated using Magma \cite{BC} for 
specific fields, but were checked for general fields by hand. 

\begin{ex}\label{bundex5}
Assume that $r =2$ and $n = 4$.
We consider the module 
\[
M = \Rad^{n(p-1)-2}(kE)/\Rad^{n(p-1)}(kE).
\]
and look at the kernel of the operator 
\[
\xymatrix{\ol \Theta_\cA: M/\Rad(M) \otimes \cA \ar[r]& 
\Rad(M)^4 \otimes \cA.}
\]
as in \ref{bund-2step}.   Taking $i=1$ in Example~\ref{dual_symm}, 
we see that 
\[\frac{\cKer(M)}{\Rad(M) \otimes\cOG}  \simeq S^2(\delta_{n-r}) = S^2(\delta_2).\]
This gives us a rank 3 vector bundle on $\Grass(2,\CV)$. For a 
plane spanned by vectors $u_i = \sum_{j=1}^4 a_{i,j}x_j \in \CV = k^4$, $i = 1,2$,
we have that $\Soc_U(M)$
is spanned by a basis for $\Rad(M)$ together with the classes 
of the elements 
\[
u_1^{p-1}u_2^{p-1}u_3^{p-2}u_4^{p-1}, \ \   
u_1^{p-1}u_2^{p-1}u_3^{p-2}u_4^{p-2}, \ \ 
u_1^{p-1}u_2^{p-1}u_3^{p-1}u_4^{p-3}, 
\]where 
$u_3$ and $u_4$ are two elements of $\CV$ which
together with $u_1$ and $u_2$ span $\CV \simeq \Rad(kE)/\Rad^2(kE)$.

We proceed to write down generators of the graded $\cA$--module
$\Ker\{\ol\Theta_{\cA}, M\}$. They
come in two types. We note that 
neither of the collections of all generators of a single type generates
a subbundle. That is to say, if we specialize the Pl\"ucker coordinates
to a random point, then (in general) the subspace of $k^{10}$ spanned by the 
specialized generators of each type has dimension 3 and hence is equal
to the subspace spanned by all of the specialized generators. The 
generators are described as follows. 
\vskip.1in
\noindent{\it Generators of Type 1.} These are indexed by the set 
$\{1,2,3,4\}$. For each $\ell \in \{1,2,3,4\}$, let $i, j$ and $k$
denote the other three elements. In what follows, we are not
assuming that $i,j,k$ are in any particular order. 
The generator corresponding to the choice of $\ell$ has the form 
\[
\gamma_\ell = \sum \mu_{a,b,c} \otimes 
x_i^{p-1-a}x_j^{p-1-b}x_k^{p-1-c}x_{\ell}^{p-1}
\]
where the index is over all tuples $(a,b,c)$ such that $a,b,c$ 
are in $\{0,1,2\}$ and $a+b+c = 2$. The coefficient $\mu_{a,b,c}$
is determined by the following rule. First, $\mu_{2,0,0} = \p_{j,k}^2$.
In the other cases, $\mu_{1,1,0} = \beta\p_{i,k}\p_{j,k}$, where $\beta$ is 
$1$ if $k$ is between $i$ and $j$ and $-1$ otherwise. The other
coefficients are obtained by permuting $i,j$ and $k$. The notational
convention is that $\p_{i,j} = \p_{j,i}$ in the event that $i>j$. 
So in the case that $\ell = 2$, the generator has the form
\begin{align*}
\gamma_2 = & \ \p_{1,3}^2 \otimes x_1^{p-1}x_2^{p-1}x_3^{p-1}x_4^{p-3} \ \
+ \p_{1,4}^2 \otimes x_1^{p-1}x_2^{p-1}x_3^{p-3}x_4^{p-1} \\
& + \p_{3,4}^2 \otimes x_1^{p-3}x_2^{p-1}x_3^{p-1}x_4^{p-1} \ \
- \p_{1,3}\p_{1,4} \otimes x_1^{p-1}x_2^{p-1}x_3^{p-2}x_4^{p-2} \\
& + \p_{1,3}\p_{3,4} \otimes x_1^{p-2}x_2^{p-1}x_3^{p-1}x_4^{p-2} \ \ 
 - \p_{1,4}\p_{3,4} \otimes x_1^{p-2}x_2^{p-1}x_3^{p-2}x_4^{p-1}
\end{align*}

\noindent{\it Generators of Type 2.} The generators of the second 
type are indexed by subsets $S = \{i,j\}$ with two elements in 
$I = \{1,2,3,4\}$. Let $k,\ell$ denote the other two elements in $I$. 
Again, we are not assuming any ordering on $i,j,k$ and $\ell$.
The generator corresponding to $S$ has the form
\[
\gamma_S = \sum \mu_{a,b,c,d} \otimes 
x_i^{p-1-a}x_j^{p-1-b}x_k^{p-1-c} x_{\ell}^{p-1-d}
\]
where the sum is over the set of all tuples $(a,b,c,d)$ such that
$\{a,b\} \subset \{0,1,2\}$, $\{c,d\} \subset \{0,1\}$, 
and $a+b+c+d=2$. The coefficients $\mu_{a,b,c,d}$ are determined
by the following rules.  
\begin{enumerate}
\item Let $\mu_{0,0,1,1} = \p_{i,j}^2$.
\item Let $\mu_{0,1,1,0} = \beta\p_{i,j}\fp_{i,k}$, 
where $\beta = 1$ if $i$ is between $j$ and 
$k$ ({\it i. e.} $j<i<k$ or $k<i<j$) and $\beta = -1$ otherwise.
\item Assume that $i<j$ then $\mu_{1,1,0,0} = 
\beta_1\p_{i,k}\p_{j,\ell}+\beta_2\p_{i,\ell}\p_{j,k}$
where $\beta_1 = \gamma_1\delta_1$ for 
\[ 
\gamma_1 = \begin{cases} 1 & \text{ if } j<\ell \\ -1 & \text{otherwise} 
\end{cases}, \qquad 
\delta_1 = \begin{cases} 1 & \text{ if } i<k  \\ -1 & \text{otherwise},
\end{cases}
\]
and $\beta_2$ is given by the same formula with $k$ and $\ell$ interchanged. 
\item Let $\mu_{0,2,0,0} = \beta\p_{j,k}\p_{j,\ell}$ where $\beta$ is $2$
if $j$ is between $k$ and $\ell$ and is $-2$ otherwise.
\end{enumerate}
So, for example, if $S = \{2,4\}$, then 
\begin{align*}
\gamma_S =  & \ \p_{2,4}^2 \otimes x_1^{p-2}x_2^{p-1}x_3^{p-2}x_4^{p-1} \ \
  -2 \p_{1,2}\p_{2,3} \otimes x_1^{p-1}x_2^{p-1}x_3^{p-1}x_4^{p-3} \\
& \p_{1,2}\p_{2,4} \otimes x_1^{p-1}x_2^{p-1}x_3^{p-2}x_4^{p-2} \ \
  - \p_{1,4}\p_{2,4} \otimes x_1^{p-1}x_2^{p-2}x_3^{p-2}x_4^{p-1} \\  
&  2 \p_{1,4}\p_{3,4} \otimes x_1^{p-1}x_2^{p-3}x_3^{p-1}x_4^{p-1} \
  - \p_{2,3}\p_{2,4} \otimes x_1^{p-2}x_2^{p-1}x_3^{p-1}x_4^{p-2} \\
& - \p_{2,4}\p_{3,4} \otimes x_1^{p-2}x_2^{p-2}x_3^{p-1}x_4^{p-1}
 + (- \p_{1,2}\p_{3,4}+\p_{1.4}\p_{2,3}) \otimes x_1^{p-1}x_2^{p-1}x_3^{p-1}x_4^{p-4}.
\end{align*}
\end{ex}


\section{APPENDIX (by J. Carlson).\\ Computing nonminimal $2$-socle support varieties using MAGMA}
\label{appendix}

We reveal the results of computer calculations of the
nonminimal 2-socle support variety of some modules. Our aim is 
to illustrate the computational method and to show some examples 
using modules that have been discussed in this paper.
All of the calculations were made using the computer algebra 
system Magma \cite{BC}.

Our first interest is the module $M = W_6 = I^6/I^8$ of 
Example \ref{ex-2-rad}.  In that example, we showed that 
the module has constant 2-radical
type, but not constant 2-socle type. The collection of all $U\in \Grass(2,\CV)$ 
for which the dimension of $\Soc_U(M)$  is more than minimal form a closed
subvariety of $\Grass(2,\CV)$, $\Soc(2,\CV)_M$.

\begin{ex}
Assume that $p > 3$.  We recall the situation in Example \ref{ex-2-rad}.
Let $\zeta$ be a primitive third root of unity in $k$. 
Let  $q_{i,i} = 1$, $q_{i,j} = \zeta$, and $q_{j,i} = \zeta^{-1}$ 
for $1 \leq i < j \leq 4$. Then 
$$S =  k\langle z_1, \dots, z_4\rangle /J$$
where $J$ is the ideal generated by  $z_i^3$ and by all $z_jz_i -q_{i,j}z_iz_j$
for $i, j \in {1,2,3,4}$. Let $I$ be the ideal generated by the 
classes of $z_1, \dots, z_4$. Let
the generator $x_i$ of $kE$ act on $M = I^6/I^8$ 
by multiplication by $z_i$. This is a module with 
constant 2-radical rank but not constant 2-socle rank. 
Recall from the proof of \ref{ex-2-rad} that $M$ has dimension 14, and 
$\Rad(M)$ has dimension 4, so $M/\Rad(M)$ has dimension 10.
The matrix of multiplication by any $x_i$ has rank 4. 

If $U \in \Grass(2,\CV)$, then $U$ is spanned by two elements which we 
can denote $u_1 = ax_1 + bx_2 + cx_3 + dx_4$ and that
$u_2 = Ax_1 + Bx_2 + Cx_3 + Dx_4$ where $a,b,c,d$ and $A,B,C,D$
are elements of $k$. In the generic case we consider them to be
indeterminants. We are interested in the maps 
\[
u_i : \ M/\Rad(M) \to \Rad(M)
\] 
of multiplication by $u_i$ for $i = 1,2$. 
If $Y_1$ is the $4 \times 10$ matrix of $u_1$ for this map
 (which is computed
by taking the indicated linear combination of the 
matrices for $x_1, \dots, x_4$) and
$Y_2$ is the matrix for $u_2$, then the intersection of the kernels
of multiplication by $u_1$ and $u_2$ is the null space of the 
$8 \times 10$-matrix $Y$ obtained by
stacking $Y_1$ on top of $Y_2$. (Note here that we are taking 
a vertical join of the matrices rather than a horizontal join 
as we would everywhere else in the paper
because the computer algebra system takes
right modules rather than left modules.) Generically, this matrix 
has rank 8. That is, when $U$ has minimal socle type on $M$, 
then $\Rad_U(M)$ has dimension
6, which counts 4 for the dimension of $\Rad(M)$ 
and another 2 for the dimension of the
intersection of the kernels of $u_1$ and $u_2$ on
$M/\Rad(M)$.  The dimension of $\Soc_U(M)$ is more than minimal
precisely when the rank of $Y$ is less than 8.

Hence, the exercise of finding the nonminimal 
2-socle support variety of $M$ is
reduced to that of finding all $8 \times 8$ 
minors of the matrix $U$. These are
polynomials in $a,b,c,d,A,B,C,D$ and to 
make sense of them in terms of the
Grassmanian, they should be converted to 
Pl\"ucker coordinates. The variety is the
zero locus of the converted polynomials. 
The Pl\"ucker coordinates are
$\fp_{12}, \fp_{13},\fp_{14}, \fp_{23}, \fp_{24}, \fp_{34}$ 
which are the determinants of the
$2\times2$ minors of the basis matrix of the plane. So, for example,
$\fp_{14} = aD-dA$. One example is the following.

\begin{prop} \label{computer1}
Suppose that $p=7$ and that $M$ is the module given above. Then the
nonminimal 2-socle support variety of $M$ is the zero locus of the ideal
generated by the elements
\[
   \fp_{12}\fp_{14}\fp_{24},  \
    \fp_{12}\fp_{13}\fp_{23}, \
    \fp_{12}\fp_{14}\fp_{34}, \
    \fp_{23}\fp_{24}\fp_{34}, \
    \fp_{13}\fp_{14}\fp_{34},
\]
\[
    \fp_{12}\fp_{14}\fp_{23}, \
    \fp_{13}\fp_{14}\fp_{34}, \
    \fp_{14}\fp_{23}\fp_{34}, \
    \fp_{12}\fp_{23}\fp_{34}.
\]
\end{prop}

With a little work we can interpret the zero 
locus in terms of the geometric
model for the Grassmannian. Thinking of a 
point in the zero locus as a plane in
four space we get that it consists of planes 
satisfying any one of the conditions below.
For notation, let $V_{ij}$ be the two 
dimensional subspace of $k^4$ spanned by
the $i^{th}$ and $j^{th}$ coordinate vectors. 
So $V_{23}$ consists of all vectors
of the form $(0,a,b,0)$ for $a,b \in k$. Then 
a closed point (plane defined over $k$) is
in the variety of the proposition if and only if it satisfies
one of the following:
\begin{itemize}
\item it contains one of the coordinate vectors, or
\item it has a basis $u_1,u_2$ where 
$u_1 \in V_{12}$ and $u_2 \in V_{34}$, or
\item it has a basis $u_1,u_2$ where 
$u_1 \in V_{14}$ and $u_2 \in V_{23}$.
\end{itemize}

At first it may seem surprising that the 
description is not symmetric. That is, it
does not include the case that 
$u_1 \in V_{13}$ and $u_2 \in V_{24}$. However,
we should recall that the algebra $S$ is 
not symmetric. There is no automorphism
that interchanges the variables.

Some similar calculations have been made 
in other cases. The identical result
was obtained when $p = 13$. We conjecture that Proposition~\ref{computer1}
is true for all primes $p>3$. 

We also got a very similar outcome in the case
that $p=3$, $s = 4$ (That is where 
relations satisfied by the variables of
$S$ consist of $z_i^4 =0$ and $z_iz_j = qz_jz_i$ 
for $i>j$ and $q$ a primitive
$4^{th}$ root of 1) and we consider the 
module $M = \Rad^{10}(S)/\Rad^{12}(S)$.
For the case that $E$ is an elementary abelian 
group of rank 5, $p = 7$ and
$M = \Rad^8(S)/\Rad^{10}(S)$, the variety 
again appears to be generated by
monomials which are the products of three 
distinct Pl\"ucker coordinates. This case
was not fully completed in that not all of the 
relations were converted to Pl\"ucker
coordinates. However, the experimental evidence suggests
that the closed points in the variety consist of
planes which contain a coordinate vector 
or have a basis $u_1, u_2$ where
$u_1$ is in the subspace $V_{ij}$ for 
$\{i,j\}$ one of the sets $\{1,2\}, \{1,5\},
\{2,3\}, \{3,4\}$ or $\{4,5\}$ and in each 
case $u_2$ is the subspace spanned
by the other three coordinate vectors.

Finally, we can also experiment with 
changing the commutativity relations in
the ring $S$ defined as above. These are the relations 
with the form $z_jz_i  = q_{ij}z_iz_j$ for
$j > i$. If $q_{ij} = 1$ for all $i$ and $j$, 
so that $S$ is commutative, then the
module $M$ has constant 2-socle type. 
In another experiment, we made random
choices of the elements $q_{ij}$ in the 
field $ k = \bF_7$. For one such choice
the module $M$ has a nonminimal 
2-socle support variety which is the zero
locus of the ideal generated by the 
polynomials $\fp_{12}\fp_{13}\fp_{23},$
$\fp_{12}\fp_{13}\fp_{24},$  $\fp_{12}\fp_{14}\fp_{24},$  
$\fp_{12}\fp_{23}\fp_{24},$
$\fp_{12}\fp_{23}\fp_{34},$  $\fp_{12}\fp_{24}\fp_{34},$   
$\fp_{13}\fp_{23}\fp_{24},$
$\fp_{14}\fp_{23},$    $\fp_{23}\fp_{24}\fp_{34}$   
and $\fp_{13}\fp_{24}.$   This variety includes
all planes that contain a coordinate vector 
(except that if it is the second coordinate
vector, then the other spanning vector must 
have zero in one of its other
coordinates). It also included all planes 
contained in the subspace $V_{134}$.
\end{ex}

We end with the remark that several other examples similar
to Example \ref{ex-2-rad} were checked for constant 2-socle
rank. In every experiment 100 random planes $U \in \Grass(2,\CV)$ 
were chosen and the value of  $d = \Dim \Soc_U(M) - \Dim \Rad(M)$ 
was calculated for each. Here $M = W_a(s, \{q_{i,j}\})$, with
$q_{i,j} = \zeta_s$, a primitive $s^{th}$ root of unity. 
For example, for $k = \bF_7$, the value of $d$ was calculated 
in the cases for which $n = 4$, $s=3$, $a = 4,5,6$
and $n = 5$, $s = 3$ and $a = 6,7,8$.
For $k = \bF_5$, $d$ was calculated for $n= 4$, 
$s = 4$, $a = 6,7,8,9$.
In all of these and in other cases, 
the module $M = W_a = I^a/I^{a+2}$, failed to have 
constant 2-socle type, even though it has constant Jordan type
and constant 2-radical type. With this evidence in hand, we 
conjecture that $M$ never has constant $2$-socle type for $(n-r)(s-1) \leq a \leq n(s-1) -2$.


\begin{thebibliography}{20}


\bibitem[\sf AS82]{AS} G. Avrunin, L. Scott, 
Quillen stratification for modules,
{\em Invent. Math.} 66 (1982), 277--286.

\bibitem[\sf Ba05]{Bal}  P. Balmer, The spectrum of prime ideals in tensor triangulated categories,
{\em J. f\"ur die Reine und Ang. Math. (Crelle)}, 588, (2005), 149--168.

\bibitem[Ben91]{Ben} D. Benson, 
 {Representations and Cohomology I, II}, {\em Cambridge University Press}, 1991.

\bibitem[Ben]{Ben2} D. Benson, Representations of elementary abelian $p$-groupd and vector bundles, in preparation.


\bibitem[BoC95]{BC} W. Bosma and J. Cannon, {Handbook of Magma Functions},
{\em Sydney: School of Mathematics and Statistics}, University of Sydney, 1995.


\bibitem[BC90]{BC90} D. Benson, J. Carlson, Products in negative cohomology, {\em J.
Pure Appl. Algebra} {\bf 82} (1992), 107-129.

\bibitem[BP]{BP} D. Benson, J. Pevtsova, Realization theorem for modules of constant Jordan type and vector bundles, to appear.

\bibitem[BL94]{BL}  J. Bernstein, V. Luntz, Equivariant sheaves  and functors,  
{\em Lecture Notes in Mathematics}, {\bf 1578}, Springer-Verlag, Berlin, (1994).

\bibitem[Car83]{C} J. Carlson, The varieties and the cohomology ring of a
module, {\em J. Algebra} {\bf 85} (1983), 104-143.

\bibitem[Car84]{C2} J. Carlson, The variety of an indecomposable module is
connected, {\em Invent. Math.}, {\bf 77} (1984), 291\--299.

\bibitem[CF09]{CF} J. Carlson, E. Friedlander, Exact category of modules of
constant Jordan type,  {\em Algebra, arithmetic, geometry: in honor of Yu. Manin}, 
Progress in Mathematics {\bf 269} (2009), 267\--290.

\bibitem[CFP08]{CFP}  J. Carlson, E. Friedlander, J. Pevtsova, Modules of 
Constant  Jordan type, {\em Journal f\'ur die reine und angewandte 
Mathematik (Crelle)} {\bf 614} (2008), 191\--234.


\bibitem[CFS11]{CFS} J. Carlson, E. Friedlander, A. Suslin, Modules for $\bZ/p \times \bZ/p$.  
{\em Comment. Math. Helvetica}, to appear.

\bibitem[CTVZ03]{CTVZ} J. Carlson, L. Townsley, L. Valero-Elizondo, M. Zhang,
{\em Cohomology rings of finite groups}, Kluwer, 2003.

\bibitem[CG97]{CG} N. Chriss, V. Ginzburg, Representation theory and complex geometry, {\em Birkhauser Boston, Inc.}, 
Boston, MA, (1997).

\bibitem[Dade78]{Dade} E.~C.~Dade,
Endo-permutation modules over $p$-groups, I, II,
{\em Ann. Math.} {\bf 107} (1978), 459--494, {\bf 108} (1978), 317\--346.

\bibitem[\sf FPa86]{FPa86} E. Friedlander, B. Parshall, 
Support varieties for restricted Lie algebras, 
{\em Invent. Math.} 86 (1986), 553--562. 

\bibitem[FP05]{FP1}  E. Friedlander, J. Pevtsova, Representation-theoretic 
support spaces for finite group schemes, {\em Amer. J. Math.} {\bf 127} (2005), 
379-420.

\bibitem[FPe]{FPe} E. Friedlander, J. Pevtsova, Erratum: 
Representation-theoretic support spaces for finite group schemes, 
{\em Amer. J. Math.} {\bf 128} (2006), 1067-1068. 


\bibitem[FP07]{FP2} E. Friedlander, J. Pevtsova, $\Pi$-supports for modules 
for finite group schemes,  {\em Duke. Math. J.} {\bf 139} (2007), 317--368.

\bibitem[FP11]{FP3} E. Friedlander, J. Pevtsova, {Constructions for infinitesimal group schemes}, 
to appear in {\em Trans. of the  AMS}.

\bibitem[FP10]{FP4} E. Friedlander, J. Pevtsova, {Generalized support varieties for finite group schemes},
{\em Documenta Mathematica - Extra volume Suslin} (2010), 197-222.


\bibitem[FPS07]{FPS}  E. Friedlander, J. Pevtsova, A. Suslin,  Generic and 
Maximal Jordan types, {\em Invent. Math.} {\bf 168} (2007), 485--522. 

\bibitem[FS97]{FS}  E. Friedlander, A. Suslin,  Cohomology of finite group schemes over a field, {\em Invent. Math.} {\bf 127} no. 2, (1997), 209\--270. 

\bibitem[Har77]{Har} R. Hartshorne, { Algebraic Geometry},  {\em Graduate Texts in Mathematics}, No. {\bf 52} Springer-Verlag, New York-Heidelberg, (1977). 

\bibitem[He61]{He} A. Heller and I Reiner, {Indecomposable representations},
{\em Illinois J. Math.,} {\bf 5} (1961), 314-323.

\bibitem[Jan03]{Jan} J. Jantzen, {Representations of Algebraic groups}, Second edition, {\em Mathematical Surveys and Monographs}, {\bf 107} American Mathematical Society, Providence, RI, (2003).

\bibitem[\sf Quillen71]{Q} D. Quillen, 
The spectrum of an equivariant cohomology ring, I, II, 
{\em Ann. of Math.} 94 (1971), 549-572, 573--602.


\bibitem[SGAI]{SGAI} Revetements etales et groupe fondamental. (French) Seminaire de Geometrie Algebrique du Bois Marie 1960–1961 (SGA 1). Dirige par Alexandre Grothendieck. Augmente de deux exposes de M. Raynaud. Lecture Notes in Mathematics, Vol. 224. Springer-Verlag, Berlin-New York, 1971.  



\bibitem[SFB1]{SFB1} A. Suslin, E. Friedlander, C. Bendel,
Infinitesimal 1-parameter subgroups and cohomology,
{\em J. Amer. Math. Soc.} {\bf 10} (1997) 693-728.



\bibitem[SFB2]{SFB2} A. Suslin, E. Friedlander, C. Bendel, Support
varieties for infinitesimal group schemes,  {\em J. Amer. Math. Soc.}
{\bf 10} (1997) 729-759.








\end{thebibliography}
\end{document}